\documentclass[letterpaper, 10pt,reqno]{amsart}

\usepackage[usenames,dvipsnames]{xcolor}

\usepackage[portrait,margin=3cm]{geometry}
\usepackage{mathrsfs}

\usepackage[foot]{amsaddr}
\usepackage{amssymb,amsthm,amsfonts,amsbsy,latexsym}
\usepackage[reqno]{amsmath}

\usepackage{color}
\usepackage{graphicx}
\usepackage{bbm}
\usepackage{comment}
\usepackage{bm}
\usepackage{a4wide,graphicx, listings,lscape, epstopdf, units, textcomp, fancyhdr, url, soul, color,subcaption}
\usepackage[textsize=small]{todonotes}
\usepackage{enumitem}

\usepackage{array}

\usepackage{orcidlink}

\usepackage{mathtools}
\mathtoolsset{showonlyrefs}

\definecolor{MyDarkblue}{rgb}{0,0.08,0.50}
\definecolor{Brickred}{rgb}{0.65,0.08,0}

\usepackage{hyperref}
\hypersetup{
	colorlinks=true,       
	linkcolor=MyDarkblue,          
	citecolor=Brickred,        
	filecolor=red,      
	urlcolor=cyan           
}

\newtheorem*{theorem*}{Theorem}
\newtheorem{theorem}{Theorem}[section]
\newtheorem{lemma}[theorem]{Lemma}

\newtheorem{proposition}[theorem]{Proposition}
\newtheorem{corollary}[theorem]{Corollary}
\newtheorem{problem}[theorem]{Open Problem}

\theoremstyle{definition}
\newtheorem{definition}[theorem]{Definition}
\newtheorem{assumption}[theorem]{Assumption}
\newtheorem{remark}[theorem]{Remark}

\newenvironment{assKalpha}{
	\textbf{Assumption $\boldsymbol{\cK_\alpha}$. }}{}

\renewcommand{\P}{\mathbb{P}}

\newcommand{\eps}{\varepsilon}

\newcommand{\cA}{\mathcal{A}}\newcommand{\cB}{\mathcal{B}}
\newcommand{\cD}{\mathcal{D}}\newcommand{\cE}{\mathcal{E}}\newcommand{\cF}{\mathcal{F}}

\newcommand{\cK}{\mathcal{K}}
\newcommand{\cO}{\mathcal{O}}
\newcommand{\cP}{\mathcal{P}}\newcommand{\cR}{\mathcal{R}}
\newcommand{\cS}{\mathcal{S}}\newcommand{\cT}{\mathcal{T}}\newcommand{\cU}{\mathcal{U}}
\newcommand{\cW}{\mathcal{W}}

\newcommand{\Var}{{\rm Var}}

\newcommand{\e}{{\mathrm e}}

\newcommand{\R}{\mathbb{R}}
\newcommand{\N}{\mathbb{N}}
\newcommand{\Z}{\mathbb{Z}}

\newcommand{\dd}{\mathrm{d}}

\renewcommand{\emptyset}{\varnothing}

\newcommand*{\wt}{\widetilde}

\newcommand*{\be}{\begin{equation}}
	\newcommand*{\ee}{\end{equation}}
\newcommand*{\ba}{\begin{aligned}}
	\newcommand*{\ea}{\end{aligned}}
\newcommand*{\barr}{\begin{array}{c}}
	\newcommand*{\earr}{\end{array}}
\newcommand{\toinp}{\overset{\mathbb P}{\longrightarrow}}
\newcommand{\toindis}{\overset{\mathrm{d}}{\longrightarrow}}
\newcommand{\toas}{\overset {\mathrm{a.s.}}{\longrightarrow}}

\newcommand*{\ind}{\mathbbm{1}}
\makeatletter
\def\namedlabel#1#2{\begingroup
	#2%
	\def\@currentlabel{#2}%
	\phantomsection\label{#1}\endgroup
}

\makeatother

\newcommand{\bes}{\begin{equation*}}
	\newcommand{\ees}{\end{equation*}}

\renewcommand{\P}[1]{\mathbb{P}\!\left(#1\right)}
\newcommand{\E}[1]{\mathbb{E}\left[#1\right]}

\renewcommand{\N}{\mathbb{N}}

\newcommand{\I}{\mathbb{I}}

\renewcommand{\th}{\mathrm{th}}
\newcommand{\cont}{\mathrm{ct}}

\numberwithin{equation}{section}

\renewcommand{\e}{\mathrm{e}}

\newcommand{\bp}{\mathrm{BP}}
\newcommand{\floor}[1]{\lfloor #1\rfloor}
\newcommand{\ceil}[1]{\lceil #1\rceil}
\newcommand{\Ps}[1]{\mathbb P_{\cS}\left(#1\right)}
\newcommand{\Es}[1]{\mathbb E_{\cS}\left[#1\right]}
\newcommand{\toinps}{\xrightarrow{\mathbb P_\cS}}
\newcommand{\pre}{\mathrm{Pre}}

\newcommand{\ensymboldremark}{\hfill$\blacktriangleleft$}

\newcommand{\invisible}[1]{}

\makeatletter
\newcommand{\leqnomode}{\tagsleft@true\let\veqno\@@leqno}
\newcommand{\reqnomode}{\tagsleft@false\let\veqno\@@eqno}
\makeatother
\newlength{\tagmarginsep} 
\title[Persistence in Preferential Attachment Trees with Vertex Death]{Preferential Attachment Trees with Vertex Death: Persistence of the Maximum Degree}
\author{Bas Lodewijks\orcidlink{0000-0001-5624-2410}}
\address{School of Mathematical and Physical Sciences, University of Sheffield.}
\email{bas.lodewijks@sheffield.ac.uk}
\date{\today} 

\begin{document}
	
	\begin{abstract}
		We consider an evolving random discrete tree model called \emph{Preferential Attachment with Vertex Death}, as introduced by Deijfen~\cite{Dei10}. Initialised with an alive root labelled $1$, at each step $n\geq1$ either a new vertex with label $n+1$ is introduced that attaches to an existing \emph{alive} vertex selected preferentially according to a function $b$, or an alive vertex is selected preferentially according to a function $d$ and \emph{killed}. 
		In this article we introduce a generalised concept of \emph{persistence} for evolving random graph models. Let $O_n$ be the smallest label among all alive vertices (the oldest alive vertex), and let $I_n^m$ be the label of the alive vertex with the $m^{\mathrm{th}}$ largest degree. We say a \emph{persistent $m$-hub} exists if $I_n^m$ converges almost surely, we say that \emph{persistence} occurs when $I_n^1/O_n$ is tight, and that \emph{lack of persistence occurs} when $I_n^1/O_n$ tends to infinity.
		
		We identify two regimes called the \emph{infinite lifetime} and \emph{finite lifetime} regimes. In the infinite lifetime regime, vertices are never killed with positive probability. Here, we provide conditions under which we prove the (non-)existence of \emph{persistent $m$-hubs} for any $m\in\N$. This expands and generalises recent work of Iyer~\cite{Iyer24} which covers the case $d\equiv 0$ and $m=1$. In the finite lifetime regime, vertices are killed after a finite number of steps almost surely. Here we provide conditions under which we prove the occurrence of \emph{persistence}, which complements recent work of Heydenreich and the author~\cite{HeyLod25}, where lack of persistence is studied for preferential attachment with vertex death.
	\end{abstract}
	
	\maketitle

	\tableofcontents
	
	\section{Introduction}\label{sec:intro}
	
	Random graph models based on preferential attachment are quite well-understood nowadays due to a large body of research formed over the last three decades (see~\cite{Hofstad1,Hofstad2} and the references therein for a good overview of the literature). These models have become a staple in random graph theory due to their use in modelling real-world networks. Still, a clear gap between most of these models and networks that form in the real world is the fact that preferential attachment graphs allow for \emph{growth}  only. Vertices and edges are sequentially attached to the graph, causing the graph to grow. However, this is not particularly realistic, as nodes and bonds in real-world networks can be \emph{removed} in almost all contexts as well. As an example, people are born but also pass away; friendships are formed but can also be broken; users of social media can befriend or follow others, but also unfriend or unfollow them and even remove their accounts all together.

	A limited body of work focusses on (preferential attachment type) models that allow for vertices and/or edges to also be removed from the network. Generally speaking, in such models at each step, 
	\begin{enumerate}[label=(\arabic*)] 
		\item A new vertex is added to the graph with probability $p_1$, which connects to $m$ already existing vertices (preferentially or uniformly at random).
		\item A new edge is added to the graph between existing vertices with probability $p_2$, where the existing vertices are selected preferentially and/or uniformly. 
		\item A vertex is selected preferentially or uniformly and is removed with probability $p_3$ (including the edges incident to it).
		\item An edge is selected uniformly and is removed with probability $p_4$.
	\end{enumerate}
	Examples include the models studied by Cooper et al.~\cite{CooFrieVer04} (with $m\in\N$) with further work of Lindholm and Vallier~\cite{LinVal11} and Vallier~\cite{Val13}, Chung and Lu~\cite{ChungLu04} (with $m=1$), Deo and Cami~\cite{DeoCami07} (with $m=1$, $p_2=p_4=0$, and vertex deletion is done anti-preferentially, i.e.\ favouring vertices with low degree), and Deijfen and Lindholm~\cite{DeiLind09} (with $m=1$, $p_3=0$). Another model of Britton and Lindholm~\cite{BritLind10}, with further work in~\cite{BritLin11,KuckSchu20}, assigns a random fitness value to each vertex, and edges are added between existing vertices with a probability proportional to the fitness values of the vertices. Work on duplication-divergence type models with edge deletion includes, among others, \cite{Thor15,HerPfaf19,BarLo21,LoReiZha25}. 	
	
	The work on such models with vertex and/or edge removal is \emph{limited} for a number of reasons. First, it generally only considers the degree distribution of the model, and if further properties are studied, this is under restrictive conditions and assumptions.  Second, all models mentioned only study affine preferential attachment or uniform attachment for attaching edges. Third, in most models the choice to add or remove vertices and edges at each step is determined by fixed model parameters (the $p_i$) and does not depend on the evolution of the graph itself. 
	
	Some recent work partially addresses these concerns. Diaz, Lichev, and the author study a model with uniform attachment and uniform vertex removal ($m$ is random, $p_2=p_4=0$), where the local weak limit, existence of a giant component, and the size and location of the maximum degree are studied~\cite{DiazLichLod22}. Work of Bellin, Kammerer et al.~\cite{BelBlaKamKor23,BelBlaKamKor23II,KamKorSen25} studies a tree model of uniform attachment and uniform vertex removal ($m=1$, $p_2=p_4=0$), but where vertices are not removed but are `frozen' and can then no longer make new connections, and where the choice to freeze vertices need not be random. This allows them to study the model in a `critical window' where they obtain precise results for the height of the tree and its scaling limit.
	
	Finally,  Deijfen~\cite{Dei10} studies a Preferential Attachment tree model with Vertex Death (PAVD). Here, death is equivalent to freezing as in the work discussed above but simply carries a different name. In this model, the probability of killing a vertex or adding a vertex, as well as to which alive vertex this new vertex connects itself, is dependent on the state of the tree. In particular, it depends on the in-degrees of the alive vertices in a general way. Deijfen studied the limiting degree distribution of the tree conditionally on survival, and shows with a number of examples how introducing death can yield novel behaviour compared to preferential attachment trees without death.  	
	
	In this article we focus on the PAVD model introduced and studied by Deijfen. To this end, let us provide a definition of the model. For a tree $T$, we naturally think of its edges as being directed towards the root. We then let $\deg_T(v)$ denote the in-degree of a vertex $v$ in $T$. For a sequence of trees $(T_n)_{n\in\N}$, we  write $\deg_n(v)$ rather than $\deg_{T_n}(v)$ for ease of writing.
	
	\begin{definition}[Preferential Attachment with Vertex Death]\label{def:pavd}
		Let $b\colon\N_0\to (0,\infty)$ and $d\colon \N_0\to [0,\infty)$ be two sequences. We recursively construct a sequence of trees $(T_n)_{n\in\N}$ and a sequence of sets of vertices $(\cA_n)_{n\in\N}$ as follows. We initialise $T_1$ as a single vertex labelled $1$ and $\cA_1=\{1\}$. For $n\geq 1$, conditionally on $T_n$ and $\cA_n$, if $\cA_n\neq \emptyset$, we select a vertex $i$ from $\cA_n$ with probability 
		\be 
		\frac{b(\deg_n(i))+d(\deg_n(i))}{\sum_{j\in \cA_n} b(\deg_n(j))+d(\deg_n(j))}.
		\ee 
		Then, conditionally on $T_n$ and $i$, we either \emph{kill} vertex $i$ with probability 
		\be 
		\frac{d(\deg_n(i))}{b(\deg_n(i))+d(\deg_n(i))}, 
		\ee 
		and set $T_{n+1}=T_n$ and $\cA_{n+1}=\cA_n\setminus\{i\}$, or otherwise construct $T_{n+1}$ from $T_n$ by introducing a new vertex $n+1$ which we connect by a directed edge to $i$ and set $\cA_{n+1}=\cA_n\cup\{n+1\}$.
		
		If $\cA_n=\emptyset$, we terminate the recursive construction, set $T_i=T_n$ for all $i>n$, and say that \emph{the tree has died}. When $\cA_n\neq \emptyset$ for all $n\in\N$ we say that the \emph{tree process survives}.
	\end{definition} 
	
	We see from the definition that the probabilities to both select and kill a vertex depend on the evolution of the tree itself. Furthermore, in our analysis of the model, we do not restrict ourselves to the uniform case ($b$ and/or $d$ constant) or the affine case ($b(i)=b_1i+b_2$ and/or $d(i)=c_1i+c_2$). Finally, we recover the standard preferential attachment tree model when $d\equiv 0$.
	
	\textbf{Persistence of the maximum degree.} The main focus of this article is \emph{persistence of the maximum degree}, which in evolving random graphs is defined as the emergence of a \emph{fixed} vertex that attains the largest degree in the graph for all but finitely many steps. When there is lack of persistence, such a fixed vertex does not exist and the maximum degree changes hands infinitely often (i.e.\ the maximum degree is attained by different vertices infinitely often). Persistence of the maximum degree is also known as degree centrality and is often leveraged in network archaeology and root-finding algorithms (see e.g.~\cite{BanBha22,BriCalLug23,BubMosRac15} and references therein). 
	
	The notion of persistence in evolving random graphs has been studied from different perspectives, starting with the work of Dereich and M\"orters~\cite{DerMor09}. Later, conditions under which persistence and the lack thereof hold have been weakened  by Galashin~\cite{Gal13}, Banerjee and Bhamidi~\cite{BanBha21}, and Iyer~\cite{Iyer24}. In short (and omitting minor technical assumptions), using the PAVD formulation with $d\equiv 0$, so that there is no vertex death and we recover the classic preferential attachment model, persistence of the maximum degree occurs almost surely if and only if 
	\be \label{eq:sumcond}
	\sum_{i=0}^\infty\frac{1}{b(i)^2}<\infty.
	\ee 
	Heuristically, persistence occurs when high-degree vertices have a sufficiently strong advantage over low-degree vertices when establishing new connections, so that they are most likely to increase their in-degree further. This allows one vertex to attain the largest degree for all but finitely many steps, instead of the largest degree switching between vertices infinitely often. The transition between whether the advantage is strong enough or not lies exactly at the point when the series in~\eqref{eq:sumcond} goes from finite to infinite.
	
	When vertices are killed and can no longer make new attachments afterwards, as in the PAVD model, the notion of persistence as stated above is not sufficiently general. In this case, we distinguish between whether vertices can live forever with positive probability, or whether they are killed after a finite (random) number of steps almost surely. In the former case, it may be possible that a vertex attains the $m^{\mathrm{th}}$ largest degree for all but finitely many steps. We call such an individual a \emph{persistent $m$-hub}. In the latter case, no individual lives forever, so that in particular no persistent $m$-hub can exist for any $m\in\N$. Here, a related question is whether the \emph{oldest alive individual} and the the \emph{richest alive individual}, that is, the alive individual with the smallest label and the alive individual with the largest degree, respectively, have labels that are `close' or `far apart'. More precisely, whether the ratio of their labels is bounded or unbounded. When the ratio is bounded, we say that \emph{persistence} occurs, whereas \emph{lack of persistence} occurs when the ratio diverges to infinity. 
	
	In the case of preferential attachment without death all individuals are always alive, so that the oldest alive vertex is always the vertex with label $1$. The ratio of their labels thus equals the label of the maximum degree vertex, which implies that the existence of persistent $1$-hub is essentially equivalent to persistence. This is, however, not the case in models that incorporate vertex death.
	
	\textbf{Our contribution. } In this article we study the two regimes outlined above. In the \emph{infinite lifetime} regime, in which vertices can live forever with positive probability, we provide conditions, similar to~\eqref{eq:sumcond}, under which a persistent $m$-hub exists almost surely conditionally on the tree process $(T_n)_{n\in\N}$ surviving, for each $m\in\N$, as well as conditions under which persistent hubs do not exist. This analysis generalises and expands recent work of Iyer~\cite{Iyer24}, which focusses on the case $d\equiv 0$ and $m=1$.
	
	The second regime is the \emph{finite lifetime} regime, in which vertices are killed after a finite number of steps almost surely. Here, we provide conditions, similar to~\eqref{eq:sumcond}, under which persistence occurs. This builds on recent work by the author and Heydenreich~\cite{HeyLod25}, in which lack of persistence is studied for the PAVD model in both the infinite and the finite lifetime regime, and extends the methodology developed by Deijfen in~\cite{Dei10} using an embedding of the discrete tree process in a  continuous-time branching process known as a Crump-Mode-Jagers branching process. We use a more precise formulation of this embedding compared to Deijfen, which allows us to simplify certain proofs of results in~\cite{Dei10}, but mainly to develop novel and precise results regarding the behaviour of the branching process embedding that were unattainable previously. Furthermore, we extend ideas used by Banerjee and Bhamidi in~\cite{BanBha21} to precisely analyse the optimal window in which the vertex with the largest degrees is found in the finite lifetime regime.
	
	\textbf{Structure of the paper. } We state the main results and the necessary assumptions in Section~\ref{sec:results}. The methodology used in the analysis in presented in Section~\ref{sec:embed}. Section~\ref{sec:inflife} studies the existence of persistent $m$-hubs in the infinite lifetime regime, whereas Section~\ref{sec:finlife} focusses on the occurrence of persistence in the finite lifetime regime. We conclude with a discussion of our work, state several open problems, and discuss possibilities for future directions in Section~\ref{sec:disc}.
	
	\textbf{Notation. } Throughout the paper we use the following notation: we let $\N\coloneq \{1,2,\ldots\}$ denote the natural numbers, set $\N_0\coloneq \{0,1,\ldots\}$ to include zero and let $[t]\coloneq \{i\in\N: i\leq t\}$ for any $t\geq 1$. For $x,y\in\R$, we let $\lceil x\rceil\coloneq \inf\{n\in\Z: n\geq x\}$ and $\lfloor x\rfloor\coloneq \sup\{n\in\Z: n\leq x\}$, and let $x\wedge y\coloneq \min\{x,y\}$ and $x\vee y\coloneq \max\{x,y\}$. For sequences $(a_n)_{n\in\N},(b_n)_{n\in\N}$ such that $b_n$ is positive for all $n$ we say that $a_n=o(b_n)$ and $a_n=\mathcal{O}(b_n)$ if $\lim_{n\to\infty} a_n/b_n=0$ and if there exists a constant $C>0$ such that $|a_n|\leq Cb_n$ for all $n\in\N$, respectively. We write $a_n=\Theta(b_n)$ if $a_n=\cO(b_n)$ and $b_n=\cO(a_n)$. For random variables $X,(X_n)_{n\in\N}$, and $Y$ we let $X_n\toindis X, X_n\toinp X$ and $X_n\toas X$ denote convergence in distribution, probability and almost sure convergence of $X_n$ to $X$, respectively. $X\preceq Y$ (resp.\ $X\succeq Y$) denotes that $X$ is stochastically dominated by $Y$ (resp.\ $X$ stochastically dominates $Y$). For a sequence $(\cE_t)_{t\in \I}$ of events, where $\I=\N$ or $\I=[0,\infty)$, we say that $\cE_t$ holds with high probability when $\lim_{t\to\infty}\P{\cE_t}=1$.  Finally, for an event $\cS$ such that $\P{\cS}>0$ we let  $\Ps{\cdot}\coloneq \mathbb{P}(\cdot\, |\,\cS)$ and  $\mathbb E_\cS{}{\cdot}\coloneq \E{\cdot\,|\,\cS}$ denote the conditional probability measure and conditional expectation, respectively. Then, we let $\xrightarrow{\mathbb P_\cS-\mathrm{a.s.}}$ respectively $\overset{\mathbb  P_\cS}{\longrightarrow}$ denote almost sure convergence and convergence in probability with respect to the probability measure $\mathbb P_\cS$. 
	
	\section{Results}\label{sec:results}
	
	In this section we present the main results and the necessary assumptions that we use. Recall the PAVD model, as in Definition~\ref{def:pavd}. We let 
	\be \label{eq:surv}
	\cS\coloneq \bigcap_{n\in\N}\{\cA_n\neq \emptyset\}
	\ee 
	denote the event that the tree process \emph{survives}, that is, the construction is never terminated. We assume throughout that $\P{\cS}>0$. We let $\mathbb P_{\cS}$ denote the probability measure $\mathbb P$ conditionally on $\cS$. Upon $\cS$, the quantities of interest are the \emph{oldest alive individual} (i.e.\ with smallest label) and the \emph{alive individuals with the largest degrees} in $T_n$. We thus define
	\be \label{eq:On}
	O_n\coloneq \min \cA_n
	\ee 
	as the `oldest' alive individual (i.e.\ with the smallest label), and recursively define 
	\be \label{eq:In1}
	I_n^{(1)}\coloneq \min\{i\in\cA_n\colon  \deg_n(i)\geq \deg_n(j) \text{ for all }j\in \cA_n\},
	\ee\label{eq:Inm}
	and for any fixed integer  $m\geq2$, 
	\be I_n^{(m)}\coloneq \min\{i\in\cA_n\setminus\{I_n^{(i)}\}_{i\in[m-1]}\colon \deg_n(i)\geq \deg_n(j) \text{ for all }j\in \cA_n\setminus\{I_n^{(i)}\}_{i\in[m-1]}\}.
	\ee 
	We can interpret $I_n^{(m)}$ as the oldest alive individual with the $m^{\mathrm{th}}$ largest degree in $T_n$ and we call $I_n^{(m)}$ the \emph{$m$-hub} of $T_n$. We note that $I_n^{(m)}$ is well-defined only if $|\cA_n|\geq m$. As we shall see, this condition is satisfied for any fixed $m\in\N$ and all $n$ sufficiently large almost surely conditionally on survival. 
	
	We are interested in the long-term behaviour of  $I_n^{(m)}$ and $O_n$. Since $O_n$ is non-decreasing in $n$, it follows that $O_n$ converges $\mathbb P_\cS$-almost surely to a limit $O$. We say that a \emph{persistent elder} exists  when  
	\be\tag{PE}\label{eq:persOn}
	O\coloneq \lim_{n\to\infty} O_n \text{ is finite }\mathbb P_\cS\text{-almost surely}.
	\ee
	When $O<\infty$ holds, we can think of $O$ as the first vertex that is added to the tree that will never be killed. It is clear that this limit exists if and only if vertices can live forever with positive probability (i.e.\ by never being both selected and killed). Otherwise, $O=\infty$ $\mathbb P_\cS$-almost surely. In a similar manner, we say that    a \emph{persistent $m$-hub} exists when there exists a $\mathbb P_\cS$-almost surely finite random variable $I^{(m)}$ such that
	\be \tag{PH-$m$}\label{eq:pershub}
	\lim_{n\to\infty}I_n^{(m)}=I^{(m)}\qquad \mathbb P_\cS\text{-almost surely}.
	\ee 
	Intuitively, the existence of the limit in~\eqref{eq:pershub} implies that the vertex with the $m^{\th}$ largest degree eventually \emph{stabilizes}, in the sense that vertex $I^{(m)}$ attains the $m^{\th}$ largest degree in $T_n$ for all but finitely many $n$. 
	
	By noting that that $I_n^{(m)}\geq O_n$ for all $m,n\in\N$, it follows that the limit $I^{(m)}$ in~\eqref{eq:pershub} cannot exist when the limit $O$ in~\eqref{eq:persOn} equals infinity. That is, if no persistent elder exists, then also no persistent $m$-hub exists.  In this case, we instead focus on the growth rate of the variables $I_n^{(m)}$ and $O_n$. We say that \emph{persistence} occurs in $(T_n)_{n\in\N}$ when 
	\be\tag{P}\label{eq:pers} 
	\Big(\frac{I_n^{(1)}}{O_n}\Big)_{n\in\N}\text{ is a tight sequence of random variables with respect to $\mathbb P_\cS$}, 
	\ee 
	and that \emph{lack of persistence} occurs in $(T_n)_{n\in\N}$ when 
	\be \tag{LP}\label{eq:nopers}
	\frac{I_n^{(1)}}{O_n}\xrightarrow{\mathbb P_\cS} \infty. 
	\ee
	We intuitively think of~\eqref{eq:pers} as the vertex with the largest degree being `close' to the oldest alive vertex, in terms of their label, whereas~\eqref{eq:nopers} implies that they are `far apart`.
	
	Clearly, when a persistent $1$-hub exists,  as in~\eqref{eq:pershub} with $m=1$, then the existence of a persistent elder, as in~\eqref{eq:persOn}, and the existence of a  persistent $1$-hub, as in~\eqref{eq:pers}, are implied, since $1\leq O_n\leq I_n^{(1)}$ and $O_n$ is non-decreasing. Similarly, lack of persistence, as in~\eqref{eq:nopers}, implies that no persistent $1$-hub exists. 
	
	In this paper, we study:
	\begin{itemize}
		\item A class of PAVD models for which we prove when it contains a persistent $m$-hub for any $m\in\N$. (Theorem~\ref{thrm:inflife})
		\item A class of PAVD models for which we prove that it does not contain a persistent $m$-hub for any $m\in\N$, but persistence does occur. (Theorem~\ref{thrm:conv})
	\end{itemize}
	To be able to state our main results, we need to define the following quantities.  Let $(E_i)_{i\in\N_0}$ be a sequence of independent exponential random variables, where $E_i$ has rate $b(i)+d(i)$, and let $(B_i)_{i\in\N_0}$ be a sequence of independent Bernoulli random variables (also independent of the $E_i$), such that 
	\be 
	\P{B_i=1}=\frac{b(i)}{b(i)+d(i)}, \qquad\text{for }i\in\N_0. 
	\ee 
	Then, define 
	\be \label{eq:D}
	S_k\coloneq \sum_{i=0}^{k-1} E_i, \qquad  \ D\coloneq \inf\{i\in\N_0: B_i=0\}, \qquad \text{and}\qquad \cR\coloneq \sum_{j=1}^D\delta_{S_j}, 
	\ee 
	where $\delta$ is a Dirac measure. Note that $\cR$ is an empty point process when $D=0$. We let  $\mu$ denote the density of the point process $\cR$, i.e.
	\be \label{eq:mu}
	\mu(t)\coloneq \lim_{\eps \downarrow 0}\eps^{-1}\P{\cR((t,t+\eps))\geq 1}, \qquad t\geq 0. 
	\ee 
	Also, let $\widehat\mu(\lambda)$ denote the Laplace transform of $\mu$, for $\lambda\geq0$. That is, 
	\be \label{eq:muhat}
	\widehat \mu(\lambda)\coloneq \int_0^\infty \e^{-\lambda t}\mu(t)\, \dd t=  \E{\sum_{k=1}^D\exp\bigg(-\lambda \sum_{i=0}^{k-1} E_i\bigg)}=\sum_{k=1}^\infty \prod_{i=0}^{k-1}\frac{b(i)}{b(i)+d(i)+\lambda}. 
	\ee 
	The explicit expression for $\widehat\mu$ in terms of the sequences $b$ and $d$ follows from~\cite[Proposition $1.1$]{Dei10}. Finally, we define for $k\in\N$ the sequences
	\be \label{eq:seqs1}
	\varphi_1(k)\coloneq  \sum_{i=0}^{k-1}\frac{1}{b(i)+d(i)}, \quad \varphi_2(k)\coloneq \sum_{i=0}^{k-1}\Big(\frac{1}{b(i)+d(i)}\Big)^2, \quad \rho_1(k)\coloneq \sum_{i=0}^{k-1}\frac{d(i)}{b(i)+d(i)}.
	\ee
	We split the presentation of the results into two parts, based on the \emph{infinite lifetime} regime and the \emph{finite lifetime} regime. As their names suggest, vertices are able to \emph{survive forever} with positive probability in the `infinite lifetime' regime, whereas vertices have an almost surely finite lifetime (i.e.\ they are killed after a finite number of steps a.s.) in the `finite lifetime' regime. Note that preferential attachment without death is a special case of the former regime, where all individuals survive forever deterministically. 
	
	The proof techniques for the results related to each regime are entirely distinct. In the `finite lifetime' regime we use a similar methodology as used in~\cite{HeyLod25}, though we now focus on the occurrence of persistence rather than lack of persistence, which requires a finer analysis at times. In the `infinite lifetime' regime we expand and generalise ideas developed by Iyer in~\cite{Iyer24}. 
	
	\textbf{Infinite lifetime regime. } We first present our result for the PAVD model when vertices can survive forever with positive probability.  To this end, we introduce the following assumptions. Recall $\widehat \mu$ from~\eqref{eq:muhat}. We assume that 
	\be \tag{F-L}\label{ass:mufin}
	\text{There exists $\lambda>0$ such that }\widehat \mu(\lambda)<\infty. 
	\ee 
	Furthermore, we have the assumptions
	\begin{align} 
		\lim_{k\to\infty}\varphi_1(k)&=\sum_{i=0}^\infty \frac{1}{b(i)+d(i)}=\infty,\tag{N-E}\label{ass:A1} \\
		\lim_{k\to\infty}\varphi_2(k)&=\sum_{i=0}^\infty \frac{1}{(b(i)+d(i))^2}<\infty, \tag{C-V}\label{ass:varphi2}\\
		\lim_{k\to\infty}\rho_1(k)&=\sum_{i=0}^\infty \frac{d(i)}{b(i)+d(i)}=\infty.\tag{F-D}\label{ass:A2}
	\end{align}
	Throughout the paper we assume that Assumption~\eqref{ass:mufin} ``Finite-Laplace" holds. It ensures that the branching process used in the analysis does not grow infinitely large in finite time.   Assumption~\eqref{ass:A1} ``Non-Explosion''  is of a similar nature and implies that there does not exist a unique vertex in the tree that obtains an infinite degree (as $n\to\infty$).  Assumption~\eqref{ass:varphi2} ``Converging-Variance" implies that the variance of $S_k$ converges with $k$, and is a \emph{crucial} assumption under which we prove that  a persistent $m$-hub exists (as in~\eqref{eq:pershub}) or that persistence occurs (as in~\eqref{eq:pers}). Finally, Assumption~\eqref{ass:A2} ``Finite-Degree" implies that the random variable $D$, as in~\eqref{eq:D}, is finite almost surely (see Lemma~\ref{lemma:Dtail}) and governs whether we are in the `infinite lifetime' regime or `finite lifetime' regime. We hence may or may not assume that~\eqref{ass:A2} is satisfied. 
	
	Deijfen~\cite{Dei10} shows that Assumption~\eqref{ass:mufin} implies that either Assumption~\eqref{ass:A1} or Assumption~\eqref{ass:A2} must be satisfied. We do not include an analysis of the case in which Assumption~\eqref{ass:A1} is not satisfied but Assumption~\eqref{ass:A2} is, which is on-going work.
	
	We then have the following result, where we show under which conditions  a persistent $m$-hub does or  does not exist. 
	
	\begin{theorem}[Infinite lifetime regime]\label{thrm:inflife}
		Consider the PAVD model in Definition~\ref{def:pavd}. Suppose that $b$ and $d$ are such that Assumption~\eqref{ass:A1} is satisfied and let $O\sim \mathrm{Geo}(\P{D=\infty})$. Then, 
		\be \label{eq:Otasconv}
		O_n\longrightarrow O\qquad \mathbb P_\cS\mathrm{-a.s.}
		\ee 
		In particular, a persistent elder exists, in the sense of~\eqref{eq:persOn}, if and only if Assumption~\eqref{ass:A2} is not satisfied.
		
		Suppose Assumption~\eqref{ass:A1} is satisfied and Assumptions~\eqref{ass:A2} and~\eqref{ass:varphi2} are not satisfied. Then, for any $m\in\N$,   a persistent $m$-hub does not exist $ \mathbb P_\cS$-almost surely. 
		
		If  Assumptions~\eqref{ass:A1}, \eqref{ass:varphi2}, and~\eqref{ass:mufin} are satisfied and Assumption~\eqref{ass:A2} is not, then for any $m\in\N$ there exists an almost surely finite random variable $I^{(m)}$ such that 
		\be \label{eq:Inas}
		I_n^{(m)}\longrightarrow I^{(m)}\qquad \mathbb P_\cS\mathrm{-a.s.} 
		\ee 
		That is, a \emph{persistent $m$-hub} exists $ \mathbb P_\cS$-almost surely, in the sense of~\eqref{eq:pershub}. Moreover, for any $m\geq 2$ there exists $n_0\in\N$ such that for all $n\geq n_0$, 
		\be \label{eq:degdistinct}
		\deg_n(I_n^{(1)})>\deg_n(I_n^{(2)})>\cdots>\deg_n(I_n^{(m)})\quad\mathbb P_\cS\mathrm{-a.s}.
		\ee 
	\end{theorem} 
	
	We summarise the results of Theorem~\ref{thrm:inflife} in Table~\ref{table:summary} and provide the proof in Section~\ref{sec:inflife}.
	
	\begin{remark}
		Theorem~\ref{thrm:inflife} \emph{generalises} Theorem $4.9$ in~\cite{BanBha21} and part of Theorem $2.3$ in~\cite{Iyer24}, which concern themselves with persistent $m$-hubs and the persistent $1$-hub in the case $d\equiv 0$ (i.e.\ preferential attachment trees without death), respectively. It also extends Theorem $2.3$ in~\cite{HeyLod25}, which proves the non-existence of a persistent $1$-hub in the infinite lifetime regime under additional assumptions, such as that $b(i)$ needs to diverge to infinity with $i$.  \ensymboldremark
	\end{remark}
	
	\begin{remark}
		Assumption~\eqref{ass:mufin} is a sufficient condition for the results in~\eqref{eq:Inas} and~\eqref{eq:degdistinct}. It can be readily verified that this assumption is satisfied for $b$ such that 
		\be 
		\limsup_{i\to\infty}\frac{b(i)}{i}<\infty,
		\ee 
		irrespective of $d$. We additionally need that $d$ is such that Assumption~\eqref{ass:A2} is not satisfied to be in the infinite lifetime regime, so that either $d$ converges to zero sufficiently fast or is non-zero only at a sufficiently sparse sub-sequence.
		
		Though crucial in our proofs, we do not believe this assumption to be necessary, so that the results in~\eqref{eq:Inas} and~\eqref{eq:degdistinct} should hold for a wider class of models. Indeed, Iyer provides such results in the case $m=1$ and $d\equiv 0$ in~\cite{Iyer24}, though a different proof idea is used.\ensymboldremark
	\end{remark}
	
	\begin{table}[h]
		\centering 
		\begin{tabular}{|l|l|}
			\hline 
			\textbf{Result} & \textbf{Assumptions}\\ 
			\hline 
			Persistent elder (\eqref{eq:Otasconv}) & \eqref{ass:A1} and \emph{not}~\eqref{ass:A2}\\ 
			\hline No persistent $m$-hub for any $m\in\N$ & \eqref{ass:A1}, \emph{not}~\eqref{ass:A2}, and \emph{not}~\eqref{ass:varphi2}\\
			\hline 
			Persistent $m$-hub for any $m\in\N$ (\eqref{eq:Inas} \& \eqref{eq:degdistinct}) & \eqref{ass:mufin}, \eqref{ass:A1}, \emph{not}~\eqref{ass:A2}, and~\eqref{ass:varphi2}\\
			\hline
		\end{tabular} 
		\caption{A summary of the results in Theorem~\ref{thrm:inflife} and the required assumptions.}\label{table:summary}
	\end{table}
	
	\textbf{Finite lifetime. } We now present our result for the PAVD model when vertices die after a finite number of steps almost surely. The results in Theorem~\ref{thrm:inflife} are qualitative, as they state whether certain random variables converge or diverge. Our result in finite lifetime, on the other hand, is quantitative. It shows that certain random variables grow at the same rate (in the sense of~\eqref{eq:pers}).  To this end, we need some additional definitions and assumptions to more finely control the behaviour of the model. First, we define
	\be \label{eq:R}
	R\coloneq \inf_{i\in\N_0}(b(i)+d(i)), 
	\ee
	and assume that there exists $d^*\in[0,R)$, such that
	\be \label{eq:dconv}
	\lim_{i\to\infty} d(i)=d^*, 
	\ee 
	and such that Assumption~\eqref{ass:A2} is satisfied. Here, Assumption~\eqref{ass:A2} ensures that vertices die after a finite number of steps almost surely. Furthermore, it is crucial that the limit $d^*$ is strictly smaller than $R$. Indeed, the `finite lifetime' regime can be divided into two further sub-regimes, called the `rich are old' `rich die young' and  regime, based on whether $d^*<R$ or $d^*>R$ (when assuming that such a $d^*$ exists), respectively, see Table~\ref{table:regimes}. The recent work of Heydenreich and the author~\cite{HeyLod25} studies lack of persistence in these regime. They show that, under mild conditions, persistence \emph{never} occurs in the `rich die young' regime, see~\cite[Theorem $2.10$]{HeyLod25}, and that lack of persistence occurs in the `rich are old' regime when Assumption~\eqref{ass:varphi2} is not satisfied. Our focus here thus lies on the `rich are old' regime, where we show that persistence does occur when Assumption~\eqref{ass:varphi2} is satisfied (and additional technical assumptions).

	\begin{table}[h]
		\centering
		\begin{tabular}{|c|c|c|}
			\hline 
			\textbf{Infinite lifetime} & \multicolumn{2}{|c|}{\textbf{Finite lifetime} }\\ 
			\hline 
			& \textbf{Rich are old} & \textbf{Rich die young} \\ 
			\hline 
			\eqref{ass:A1} and \emph{not}~\eqref{ass:A2} & \eqref{ass:A1}, \eqref{ass:A2}, and~\eqref{eq:dconv}, $d^*<R$ & \eqref{ass:A1}, \eqref{ass:A2}, and~\eqref{eq:dconv}, $d^*>R$\\ 
			\hline 
		\end{tabular}
		\caption{The different regimes that can be observed in the PAVD model and the main assumptions that are required in each regime. The `infinite lifetime' regime in studied in Theorem~\ref{thrm:inflife}, and `finite lifetime' regime, and in particular the `rich are old' regime is studied in Theorem~\ref{thrm:conv}. The `rich die young' regime is studied in~\cite{HeyLod25}.}\label{table:regimes}
	\end{table}
	
	We also define
	\be \label{eq:alpha}
	\alpha(k)\coloneq \rho_1(k)-d^*\varphi_1(k)=\sum_{i=0}^{k-1}\frac{d(i)-d^*}{b(i)+d(i)} \qquad\text{for } k\in \N. 
	\ee
	Note that $\alpha\equiv \rho_1$ when $d^*=0$, that is, when $d$ converges to zero. We extend the domain of $\varphi_1,\varphi_2,\rho_1$, as in~\eqref{eq:seqs1}, and $\alpha$ to $\R_+$ by linear interpolation so that, for example, $\varphi_1(t)=\int_0^t1/(b(\lfloor x\rfloor)+d(\lfloor x\rfloor))\,\dd x$ for $t\geq 0$. In particular, this implies that $\varphi_1$ and $\varphi_2$ are strictly increasing and thus invertible.   We then define 
	\be \label{eq:Ks}
	\cK_\alpha(t)\coloneq \alpha\big(\varphi_1^{-1}(t)\big) \qquad \text{for }t\geq 0. 
	\ee  
	Heuristically, $\cK_\alpha(t)$ quantifies the difference between $\rho_1(\varphi_1^{-1}(t))$ and $d^*t$, and allows us to precisely quantify probabilities of the form $\P{D\geq \varphi_1^{-1}(t)}$ (see Lemma~\ref{lemma:Dtail}). We then recall $\widehat \mu$ from~\eqref{eq:muhat}. Assumption~\eqref{ass:mufin} is not sufficiently strong for the proof techniques used in the `finite lifetime' regime. Rather, we introduce
	\be \label{eq:undlambda}
	\underline \lambda\coloneq \inf\{\lambda>0: \widehat \mu(\lambda)<\infty\}, 
	\ee 
	and assume that 
	\be \tag{Ma}\label{ass:C1}
	\widehat \mu(\lambda^*)=1\text{ has a solution }\lambda^*\in(0,\infty)\qquad\text{and}\qquad\  
	\lambda^* >\underline \lambda.
	\ee
	The parameter $\lambda^*$ is known as the \emph{Malthusian parameter}, and clearly Assumption~\eqref{ass:C1} is stronger compared to Assumption~\eqref{ass:mufin}. 
	
	Finally, we have the following technical assumption for the function $\cK_\alpha$, as in~\eqref{eq:Ks}.
	\begin{assKalpha}\label{ass:Kalpha}
		For $\lambda^*>0,d^*\geq 0$, there exists a function $r\colon(0,\infty)\to\R_+$ such that
		\be\label{ass:Kalphas}
		\cK_\alpha\big(\tfrac{\lambda^*}{\lambda^*+d^*}t-\tfrac{1}{\lambda^*+d^*}\cK_\alpha(r(t))\big)-\cK_\alpha(r(t))=\cO(1).
		\ee 
	\end{assKalpha}
	
	\begin{remark} 
		Assumption~\hyperref[ass:Kalpha]{$\cK_\alpha$}  can be interpreted as an assumption for $\rho_1\circ \varphi_1^{-1}$ in the case that $d$ converges to zero (as $\alpha\equiv \rho_1$ and thus $\cK_\alpha\equiv \rho_1\circ \varphi_1^{-1}$ in that case).\ensymboldremark
	\end{remark}
	
	\begin{remark}\label{rem:Kalpha}
		Assumption~\hyperref[ass:Kalpha]{$\cK_\alpha$} is satisfied when $r(t)=\frac{\lambda^*}{\lambda^*+d^*}t-x(t)$, where $x(t)$ satisfies $x(t)=\frac{1}{\lambda^*+d^*}\cK_\alpha(\frac{\lambda^*}{\lambda^*+d^*}t-x(t))+\cO(1)$.\ensymboldremark
	\end{remark}
	
	We then have the following result, where we provide conditions under which persistence occurs. 
	
	\begin{theorem}[Finite lifetime regime]\label{thrm:conv}
		Consider the PAVD model, as in Definition~\ref{def:pavd}.  Suppose that $b$ and $d$ are such that Assumptions~\eqref{ass:A1}, \eqref{ass:C1}, \eqref{ass:A2}, and~\eqref{ass:varphi2} are satisfied. Recall $R$ from~\eqref{eq:R} and suppose that $d$ satisfies~\eqref{eq:dconv} with $d^*\in[0,R)$, and let Assumption~\hyperref[ass:Kalpha]{$\cK_\alpha$} be satisfied. Then,
		\begin{align}
			\Big(\log O_n-\frac{d^*}{\lambda^*+d^*}\log n-\frac{\lambda^*}{\lambda^*+d^*}\cK_\alpha\big(r\big(\tfrac{1}{\lambda^*}\log n\big)\big)\Big)_{n\in\N},\label{eq:Onconv}
			\intertext{and} 
			\Big(\log I_n^{(1)}-\frac{d^*}{\lambda^*+d^*}\log  n-\frac{\lambda^*}{\lambda^*+d^*}\cK_\alpha \big(r\big(\tfrac{1}{\lambda^*}\log n\big)\big)\Big)_{n\in\N}, \label{eq:Inconv}
			\intertext{and} 
			\Big(\varphi_1(\max_{v\in \cA_n}\deg_n(v))-\frac{1}{\lambda^*+d^*}\log n+\frac{1}{\lambda^*+d^*}\cK_\alpha\big(r\big(\tfrac{1}{\lambda^*}\log n\big)\big)\Big)_{n\in\N}.\label{eq:maxdegconv}
		\end{align} 
		are tight sequences of random variables with respect to $\mathbb P_\cS$. In particular, persistence occurs in the sense of~\eqref{eq:pers}. 
	\end{theorem}
	
	\begin{remark}
		The results in~\eqref{eq:Inconv} and~\eqref{eq:maxdegconv} hold for \emph{any} $d^*\geq 0$. The restriction $d^*<R$ is necessary only for~\eqref{eq:Onconv}, caused by the `rich are old' and `rich die young' sub-regimes (as in Table~\ref{table:regimes}).\ensymboldremark
	\end{remark}
	
	\begin{remark}\label{rem:incldead}
		The results presented in~\eqref{eq:Inconv} and~\eqref{eq:maxdegconv} remain true if we consider the label of the vertex with the largest degree among \emph{all} vertices (i.e.\ both alive and dead). This can be verified by modifying parts of the proofs in later sections. It shows that vertices can obtain a large degree only if they also manage to stay alive for a long time, and do not get lucky due to another vertex with a possibly even larger degree dying.\ensymboldremark
	\end{remark} 
	
	The crucial assumption in both Theorem~\ref{thrm:inflife} and~\ref{thrm:conv} is Assumption~\eqref{ass:varphi2}. Indeed, it follows from the first part of Theorem~\ref{thrm:inflife} that no persistent $m$-hub exists $\mathbb P_\cS$-almost surely when Assumption~\eqref{ass:varphi2} is not satisfied. Also, Theorem $2.10$ in~\cite{HeyLod25} shows that \emph{lack of persistence} in the sense of~\eqref{eq:nopers} occurs when Assumption~\eqref{ass:varphi2} is \emph{not} satisfied. We note that this condition for persistence and the existence of persistent $m$-hub is similar to that for preferential attachment without vertex death, as in~\eqref{eq:sumcond}.
	
	\textbf{Discussion,  Examples, and Related work. } The results presented in this section identify the behaviour of the PAVD model in the `infinite lifetime' and `finite lifetime' with regards to persistence of the maximum degree (or $m^{\mathrm{th}}$-largest degree, in general). In the former regime, we require fairly mild conditions only. In the `finite lifetime' regime, stronger assumptions are required compared to the `infinite lifetime' regime. This is due to the fact that we need very precise asymptotic expansions of the random variables $I_n^{(1)}$ and $O_n$, which in this regime no longer converge without rescaling. In particular, we require control over the growth rate of a continuous-time branching process that we analyse to prove the results in Theorem~\ref{thrm:conv} (see Section~\ref{sec:embed} for more details). In Section~\ref{sec:disc} we discuss some open problems related to this regime. 
	
	\emph{Examples.} In the `finite lifetime' regime, we can further distinguish between the `rich are old' and `rich die young' regimes. Whether the PAVD model is in the `rich are old' or `rich die young' regime subtly depends on the sequences $d$ and $b$. By changing only a small number of values in either sequences, it is  possible to go from one regime to the other. For example, consider the three models
	\be \ba \label{eq:ex}
	({}&\text{RaO}) & b&=(1,2,3,\ldots)&&\text{and } d=(1,2,\tfrac32,\tfrac32,\ldots),\\ 	({}&\text{RdY}1)& b&=(1,2,3,\ldots)&&\text{and } d=(\tfrac14,2,\tfrac32,\tfrac32,\ldots),\\
	({}&\text{RdY}2) & b&=(\tfrac14,2,3,\ldots)&&\text{and } d=(1,2,\tfrac32,\tfrac32,\ldots).
	\ea\ee 
	In the first example we have $R=2$ and $d^*=\frac32$, so that this case belongs to the `rich are old' regime. In the second and third line we have changed one instance of $d$ and $b$, respectively, so that $R=\frac54$ and $d^*=\frac32$. These cases thus belong to the `rich die young' regime.	
	
	\emph{Assumptions.} Let us then provide some examples that satisfy the assumptions used in the two main theorems. First, we discuss Assumption~\hyperref[ass:Kalpha]{$\cK_\alpha$}. Here, with $c\in(0,\lambda^*/(\lambda^*+d^*))$, set 
	\be 
	\eps_t\coloneq \sup_{i\geq ct}|d(i)-d^*|. 
	\ee 
	Assumption~\hyperref[ass:Kalpha]{$\cK_\alpha$} is then satisfied if there exists $k\in\N$ such that 
	\be\label{eq:eps} 
	\eps_t^k=\cO\Big(\frac{1}{\cK_\alpha(t)}\Big).
	\ee 
	Namely, using Remark~\ref{rem:Kalpha}, set 
	\be 
	x_1(t)\coloneq \frac{1}{\lambda^*+d^*}\cK_\alpha\big(\tfrac{\lambda^*}{\lambda^*+d^*}t\big)\qquad\text{and}\qquad x_i(t)\coloneq \frac{1}{\lambda^*+d^*}\cK\big(\tfrac{\lambda^*}{\lambda^*+d^*}t-x_{i-1}(t)\big)\qquad \text{for }i>1. 
	\ee 
	Then, take $x(t)\coloneq x_k(t)$ and $r(t)\coloneq \frac{\lambda^*}{\lambda^*+d^*}t-x(t)$. It can be shown, using (the proof of) Lemma~\ref{lemma:func}$(c)$ that this choice of $r$ satisfies Assumption~\hyperref[ass:Kalpha]{$\cK_\alpha$}. We observe that $\cK_\alpha(t)=o(t)$ by Lemma~\ref{lemma:func}$(c)$, so that $\eps_t=t^{-\beta}$ for some $\beta>0$ satisfies~\eqref{eq:eps} for any $k>\beta^{-1}$. That is, as long as $d$ converges to $d^*$ at a polynomial rate, Assumption~\hyperref[ass:Kalpha]{$\cK_\alpha$} is satisfied. We expect that much weaker conditions on $d$ are sufficient as well, when making more assumptions about the regularity of $b$.
	
	It can be checked with relative ease that Assumption~\eqref{ass:C1} is satisfied when 
	\begin{itemize} 
		\item $\E{D}\in(1,\infty)$, or
		\item $\E{D}=\infty$, $b$ tends to infinity,   $C_1\coloneq \limsup_{i\to\infty}\frac{b(i)}{i}<\infty$, and $\widehat \mu(C_1)>1$.
	\end{itemize} 
	Finally, Assumption~\eqref{ass:mufin} is implied by Assumption~\eqref{ass:C1}. It thus requires fewer conditions on the sequences $b$ and $d$. Indeed, it is satisfied when 
	\begin{itemize} 
		\item $\E{D}\in(1,\infty)$, or
		\item $\E{D}=\infty$,   $\limsup_{i\to\infty}\frac{b(i)}{i}<\infty$.
	\end{itemize}
	
	\emph{Related work. } The results in Theorem~\ref{thrm:inflife} expand and generalise recent work of Iyer in~\cite{Iyer24} in which he developed a novel methodology to study the (non-)existence of persistent hubs in evolving random trees, based on the summability of infinite sums of symmetric random variables. The proof of Theorem~\ref{thrm:inflife} is inspired by this methodology, but generalises it to include vertex death, expands the results to persistent $m$-hubs for $m>1$, and includes the result regarding the distinct degrees of $m$-hubs in~\eqref{eq:degdistinct}. The results on persistent $m$-hubs with $m>1$, as well as the result in~\eqref{eq:degdistinct} are novel and also generalise earlier work of Banerjee and Bhamidi for preferential attachment trees~\cite[Theorem $4.9$ and Remark $4.11$]{BanBha21}. 
	
	Recent work of Heydenreich and the author in~\cite{HeyLod25} also studies the PAVD model, but focusses on lack of persistence, in the sense of~\eqref{eq:nopers}. This complements the results presented here and combined these two works provide a clear understanding of the different regimes that can be observed, see Table~\ref{table:regimes}. Further extensions can still be made, however, in particular by not requiring that the death sequence $d$ converges to a constant $d^*\geq 0$. We discuss this in more detail in Section~\ref{sec:disc}. 
	
	\section{Methodology: Embedding in CMJ branching process}\label{sec:embed}
	
	To prove the results presented in Section~\ref{sec:results}, we make use of a technique where one `embeds' a random discrete process into a Crump--Mode-Jagers (CMJ) continuous-time branching process. The  random discrete structure is then equal in distribution to the CMJ branching process, when viewed at certain stopping times, whilst the CMJ branching process in continuous time provides more analytical advantages compared to the discrete process. 
	
	\textbf{General CMJ branching processes. } We first introduce some definitions and notation to facilitate the introduction of the CMJ branching process that relates to the discrete PAVD model. We let $\cU_\infty$ denote the Ulam-Harris tree. That is,
	\be\label{eq:UH}
	\cU_\infty\coloneq \{\varnothing\}\cup\bigcup_{k=1}^\infty \N^k.
	\ee 
	For $v=v_1v_2\cdots v_k\in \cU_\infty$, we view $v$ as the $v_k^{\th}$ child of $v_1\cdots v_{k-1}$ (where $v_1\cdots v_{k-1}$ denotes $\varnothing$ when $k=1$). In particular, we view $j\in \N$ as the $j^{\th}$ child of $\varnothing$. We assign to each individual $v\in\cU_\infty$  a sequence of non-negative \emph{inter-birth time} random variables $(X(vi))_{i\in \N}$ and an \emph{offspring} random variable $D^{(v)}$, which for each $v\in \cU_\infty$ are i.i.d.\ copies of some random variables $(X(i))_{i\in\N}$ and $D$, respectively. We assume throughout that $X(i)<\infty$ for all $i\in\N$, almost surely. 
	
	Given the inter-birth times of an individual $v$, we define its \emph{lifetime} as 
	\be \label{eq:lifetime}
	L^{(v)}\coloneq \sum_{i=1}^{D^{(v)}+1}X(vi).
	\ee 
	Furthermore, we can recursively construct all \emph{birth times} $\cB(v)$ by setting $\cB(\varnothing)\coloneq 0$ and
	\be 
	\cB(vj)\coloneq \begin{cases} 
		\cB(v)+\sum_{i=1}^j X(vi)&\mbox{for }v\in \cU_\infty, j\in\{1,\ldots, D^{(v)}\},\\ 
		\infty &\mbox{for }v\in \cU_\infty, j\in\{D^{(v)}+1,\ldots\}.
	\end{cases}
	\ee  
	Given that $v$ is born in finite time, each child $v1,\ldots, vD^{(v)}$ is born finite time and children $vj$ for $j>D^{(v)}$ are never born (as their birth-time is infinite). The individual $v$ then dies after $L^{(v)}$ time. If $v$ is never born (i.e.\ $\cB(v)=\infty$) then it is clear that all descendants of $v$ are never born either.
	
	We can then construct a branching process $(\bp(t))_{t\geq 0}$ by setting \be \label{eq:bpdef}
	\bp(t)\coloneq \{u\in \cU_\infty: \cB(u)\leq t\}\qquad \text{for }t\geq 0. 
	\ee 
	It is clear that individuals $v=v_1\cdots v_m$ with $v_1,\ldots, v_m\in\N$ and $m\in\N_0$ are born in the process $\bp$ if and only if $v_i\leq D^{(v_1\cdots v_{i-1})}$ for all $i\in[m]$. All other individuals are never born. We further define 
	\be \label{eq:At}
	\cA^\cont_t\coloneq \{u\in \cU_\infty: \cB(u) \leq t<\cB(u)+L^{(u)}\}, \qquad \text{ for }t\geq 0, 
	\ee 
	as the set of alive individuals at time $t$. We can view this as the continuous-time counterpart of $\cA_n$, as defined in Definition~\ref{def:pavd}. We abuse notation to write 
	\be 
	\cS=\{\forall t\geq 0: \cA_t^\cont\neq\emptyset\}
	\ee 
	as the event that the branching process $\bp$ survives.

	\textbf{Continuous-time embedding of PAVD. }  We now provide a specific case for the general definition above that we coin the continuous-time preferential attachment with vertex death branching process (CTPAVD), which is directly related to the PAVD model. Recall the random variables $D$, $(E_i)_{i\in\N_0}$, and $(S_k)_{k\in\N}$ from~\eqref{eq:D}. For each individual $v\in \cU_\infty$, we let $D^{(v)}$ be an i.i.d.\ copy of $D$, and set $X(vi)=E_{i-1}^{(v)}$, where $E^{(v)}_{i-1}$ is an i.i.d.\ copy of $E_{i-1}$ for each $i\in\N$. With $S_k^{(v)}\coloneq \sum_{i=0}^{k-1} E^{(v)}_i$, we can write for each $v\in \cU_\infty$ and $k\in\N$,
	\be \label{eq:birthlifetime}
	\sum_{i=1}^k X(vi)=S_k^{(v)}\qquad\text{and, in particular,}\qquad L^{(v)}\overset{\mathrm{d}}{=} S_{D+1}^{(v)}.
	\ee 
	Let us introduce for $v\in\cU_\infty$ the characteristics
	\be 
	\chi^{(v)}_{\mathrm b}(t)\coloneq \ind_{\{t\geq 0\}}, \qquad\text{and}\qquad  \chi^{(v)}_{\mathrm a}(t)\coloneq \ind_{\{0\leq t<L^{(v)}\}},
	\ee 
	and define, for $\square\in\{\mathrm a,\mathrm b\}$,
	\be 
	Z_{\square}(t)\coloneq \sum_{v\in\cU_\infty}\chi^{(v)}_{\square}(t-\cB(v)).
	\ee 
	$Z_\mathrm{a}$ and $Z_{\mathrm b}$ `count' the branching process $\bp$ according to the characteristics $\chi_{\mathrm a}$ and $\chi_{\mathrm b}$, respectively. Here, $Z_{\mathrm a}(t)$ denotes the number of alive individuals in $\bp(t)$ and $Z_{\mathrm b}(t)$ denotes the total number of individuals born in $\bp(t)$. We then set
	\be \label{eq:Nt}
	N(t)\coloneq Z^{\mathrm b}_t+(Z^{\mathrm b}_t-Z^{\mathrm a}_t) \qquad\text{for }t\geq0,
	\ee 
	as the number of births and deaths that have occurred in the branching process up to time $t$. We  let $(\tau_n)_{n\in\N}$ be stopping times, defined  as
	\be\label{eq:taun} 
	\tau_n\coloneq \inf\{t\geq 0: N(t)=n\}, \qquad n\in\N. 
	\ee 
	That is, $\tau_n$ denotes the time at which the $n^{\th}$ birth or death event occurs. We then have the following correspondence between the discrete tree process $(T_n, \cA_n)_{n\in\N}$ and the CMJ branching process $(\bp(\tau_n), \cA_{\tau_n}^{\cont})_{n\in\N}$ viewed at the sequence of stopping times.
	
	\begin{proposition}[Embedding of PAVD in CTPAVD]\label{prop:embed}
		Let $(T_n,\cA_n)_{n\in\N}$ be the sequence of trees and sets of alive vertices of a PAVD model and $(\bp(t),\cA_t^\cont)_{t\geq 0}$ be a CTPAVD model and the set of alive individuals, both constructed with the same birth and death sequences $b$ and $d$, respectively. Then, 
		\be 
		(T_n,\cA_n)_{n\in\N}\overset{\mathrm{d}}{=} (\bp(\tau_n),\cA^\cont_{\tau_n})_{n\in\N}. 
		\ee  
	\end{proposition}
	
	The proposition follows from the memoryless property of exponential random variables and properties of minima of exponential random variables. The technique of embedding evolving discrete structures was originally pioneered by Athreya and Karlin~\cite{ArthKar68} and has found fruitful applications in a wide variety of discrete (combinatorial) models, such as P\'olya urns, discrete randomly growing trees such as preferential attachment trees and uniform attachment trees (as first observed by Pittel~\cite{Pit94}), evolving simplicial complexes, and many more. 
	
	The description of the branching process $\bp$ in which we embed the PAVD tree is used by Deijfen~\cite{Dei10} to determine an explicit description for the limiting offspring distribution of alive individuals. The description of the branching process used here is somewhat different (though equivalent). Namely, the explicit construction in terms of the sum of inter-birth times $S_k$, the offspring random variable $D$, and the lifetime $L$, was not used by Deijfen. This viewpoint provides an analytical advantage. In particular, it allows us to analyse the lifetime distribution of individuals and the number of children an individual produces over time to a greater extent.

	\newpage
	\textbf{Persistence in the CTPAVD branching process. } 
	
	We can view the concepts of persistent $m$-hubs, and persistence and the lack thereof from a continuous-time perspective as well. To this end, we define 
	\be \label{eq:offspring}
	\deg_t(u)\coloneq \min\big\{\sup\big\{k\in\N_0: S_k^{(u)}\leq t-\cB(u)\big\},D^{(u)}\big\},\qquad u\in \cU_\infty, t\geq 0,
	\ee 
	as the number of children $u$ has produced by time $t$. We use the convention that $S_0^{(v)}=0$ for all $v$ and that we set $\deg_t(u)=-\infty$ when $t<\cB(u)$, i.e.\ when $u$ has not (yet) been born. We also define
	\be\label{eq:contdef}
	O_t^\cont\coloneq \!\!\min_{v\in \cA_t^\cont}\!\cB(v)
	\ee
	as the birth-time of the oldest alive individual at time $t$, which is the continuous-time counterpart of $O_n$, as in~\eqref{eq:On}. Similarly, we define 
	\be 
	I_t^{(1),\cont}\coloneq \min\big\{\cB(v):v\in \cA^\cont_t, \deg_t(v)\geq\deg_t(u)\text{ for all }u\in \cA^\cont_t\big\},
	\ee 
	and for $m\geq 2$ fixed, 
	\be 
	I_t^{(m),\cont}\coloneq \min\big\{\cB(v):v\in \cA^\cont_t\setminus\{I_t^{(i),\cont}\}_{i\in[m-1]}, \deg_t(v)\geq\deg_t(u)\text{ for all }u\in \cA^\cont_t\setminus\{I_t^{(i),\cont}\}_{i\in[m-1]}\big\}.
	\ee 
	Here $I_t^{(m),\cont}$ denotes the birth-time of the oldest alive individual that has the $m^{\mathrm{th}}$ largest number of children at time $t$ and is the continuous-time counterpart of $I_n^{(m)}$. We say a branching process $(\bp(t))_{t\geq 0}$ contains a persistent elder  when there exists a random variable $O^{\cont}$ that is finite almost surely, such that 
	\be
	O^\cont_t\to O^\cont \qquad \mathbb P_\cS\text{-almost surely as $t\to\infty$}.
	\ee
	We say that $(\bp(t))_{t\geq 0}$ contains a persistent $m$-hub when there exists an almost surely finite random variable $I^{(m),\cont}$ such that
	\be \label{eq:hubdefcont}
	I_t^{(m),\cont} \to I^{(m),\cont} \qquad \mathbb P_\cS\text{-almost surely as $t\to\infty$}.
	\ee 
	When $(\bp(t))_{\geq 0}$ does not contain a persistent $m$-hub for any $m\in\N$, we say that persistence occurs in $(\bp(t))_{t\geq 0}$  when 
	\be\label{eq:persnoperscont} 
	(I^{(1),\cont}_t-O^\cont_t)_{t\geq 0} \text{ is a tight sequence of random variables with respect to $\mathbb P_\cS$,} 
	\ee 
	and that lack of persistence occurs when 
	\be 
	I^{(1),\cont}_t-O^\cont_t\toinps \infty. 
	\ee 
	As in the discrete setting, $(\bp(t))_{t\geq 0}$ containing a persistent $1$-hub implies that it contains a persistent elder, and it implies persistence in $(\bp(t))_{t\geq 0}$. On the other, hand, lack of persistence implies that $(\bp(t))_{t\geq 0}$ does not contain a persistent $1$-hub.

	Proposition~\ref{prop:embed} also implies that 
	\be \label{eq:OtItequiv}
	\{(O_n,I_n): n\in\N\}\overset{\mathrm{d}}{=} \{(N(O^\cont_{\tau_n}),N(I^\cont_{\tau_n})): n\in\N\}.
	\ee 
	Under Assumption~\eqref{ass:C1} it follows that $N(t)$, as defined in~\eqref{eq:Nt}, satisfies 
	\be \label{eq:Nconv}
	N(t)\e^{-\lambda^*\!t}\toas W, 
	\ee 
	where $W$ is a non-negative random variable, such that $\P{W>0\,|\,\cS}=1$ and hence that $\tau_n-\frac{1}{\lambda^*}\log n$ converges almost surely (see also Proposition~\ref{prop:growth}). As a result, we can, approximately, relate the quantities in~\eqref{eq:OtItequiv} via 
	\be \label{eq:corr}
	O_n\approx \exp\big( \lambda^* O^\cont_{\log(n)/\lambda^*}\big), \qquad\text{and},\qquad \ I_n\approx \exp\big( \lambda^* I^\cont_{\log(n)/\lambda^*}\big).
	\ee 
	In particular, this relates the definition of a persistent hub and (lack of ) persistence in $(\bp(t))_{t\geq 0}$ to the same concepts for the PAVD discrete tree process $(T_n)_{n\in\N}$. 
	
	In Sections~\ref{sec:inflife} and~\ref{sec:finlife} we analyse the CMJ branching process  defined in this section and, in particular, prove results under which conditions the behaviour in~\eqref{eq:hubdefcont} and \eqref{eq:persnoperscont} occurs. In each of these sections, we use these results when proving  the main theorems by leveraging Proposition~\ref{prop:embed}.

	\section{(Non)-existence of persistent hubs} \label{sec:inflife}
	
	This section is devoted to proving results regarding the (non-)existence of persistent hubs for the CTPAVD branching process introduced in Section~\ref{sec:embed}. Via Proposition~\ref{prop:embed} we then obtain Theorem~\ref{thrm:inflife}. 

	We first state the following result classifying when $(\bp(t))_{t\geq 0}$ does (not) contain a persistent $m$-hub for any $m\in\N$

	\begin{theorem}[(Non-)existence of $m$-hubs]\label{thrm:genpers}
		Recall the CTPAVD branching process $(\bp(t))_{t\geq 0}$ defined in~\eqref{eq:bpdef}, with offspring, inter-birth times, and lifetimes as defined in (the paragraph before)~\eqref{eq:birthlifetime}. Assume that Assumption~\eqref{ass:A1} is satisfied and Assumption~\eqref{ass:A2} is not satisfied. 
		\begin{enumerate}[label=(\alph*)]
			\item If Assumption~\eqref{ass:varphi2} is not satisfied, then for any $m\in\N$, $\mathbb P_\cS$-almost surely $(\bp(t))_{t\geq 0}$ does not contain a persistent $m$-hub.
			\item If Assumption~\eqref{ass:varphi2} is satisfied and Assumption~\eqref{ass:mufin} is satisfied, then for any $m\in \N$,  $\mathbb P_\cS$-almost surely $(\bp(t))_{t\geq 0}$ contains a unique persistent $m$-hub, in the sense of~\eqref{eq:hubdefcont}.  For any $m\geq 2$ there exists $t'>0$ such that  $\mathbb P_\cS$-almost surely for all $t>t'$,
			\be \label{eq:distdeg} \deg_t(I^{(1),\cont})>\deg_t(I^{(2),\cont})>\cdots>\deg_t(I^{(m),\cont}).
			\ee  
		\end{enumerate}
	\end{theorem}

	\textbf{Related work.} The results presented in Theorem~\ref{thrm:genpers} are related to recent work of Iyer on persistent hubs in general CMJ branching processes~\cite{Iyer24}. The main difference between the results of Iyer and the results presented here are:
	\begin{itemize}
		\item Iyer focusses on general branching processes, where  the inter-birth times $X(i)$ are not necessarily exponentially distributed.
		\item Iyer focusses the case where the offspring $D^{(v)}$ is infinite for each individual $v$.
		\item The results of Iyer are restricted to the case $m=1$ and do not include~\eqref{eq:distdeg}.
	\end{itemize} 
	
	In short, based on proof techniques developed by Iyer, we extend his results to the case $m>1$ and to the case that $\P{D=\infty}\in(0,1]$ (which is the case when  Assumption~\eqref{ass:A2} is not satisfied) when the inter-birth times are exponentially distributed. These extensions can be proved for more general CMJ branching processes, where we do not assume the inter-birth times to be exponentially distributed, as well. As the focus of this work lies on the PAVD model, we do not prove this explicitly, but rather provide a discussion on how to  extend (part of) the results in Theorem~\ref{thrm:genpers} to general CMJ branching processes can be found Appendix~\ref{app:genproof}.

	\subsection{Proof of Theorem~\ref{thrm:inflife}} 
	
	To translate the results in Theorem~\ref{thrm:genpers} to the results of Theorem~\ref{thrm:inflife}, we first state the following lemma, which is a slight generalisation of Lemma $3.5$ in~\cite{Iyer24}. 
	
	\begin{lemma}\label{lemma:finbp}
		Recall the Laplace transform  $\widehat \mu$ from~\eqref{eq:muhat} and suppose that Assumption~\eqref{ass:mufin} is satisfied. That is, suppose that there exists $\lambda>0$ such that $\widehat\mu(\lambda)<\infty$. Then, $|\bp(t)|<\infty$ for all $t\geq 0$ almost surely.
	\end{lemma}
	
	\begin{proof}
		We start by observing that Assumption~\eqref{ass:mufin} is equivalent to there existing $\lambda'\geq \lambda$ such that $\widehat\mu(\lambda')<1$. Indeed, since $\widehat\mu(\lambda)<\infty$ and $\widehat\mu(\alpha)$ is decreasing in $\alpha$, it follows from the dominated convergence theorem that 
		\be\label{eq:muless1}
		\lim_{\alpha \to\infty}\widehat\mu(\alpha)=\mu(\{0\})=0\qquad \text{and so there exists $\lambda'\in[\lambda,\infty)$ such that }\widehat\mu(\lambda')<1.  
		\ee 	
		We then prove the result.  For any $v=v_1\cdots v_m\in \cU_\infty$, with $m\in\N$ and $v_1,\ldots, v_m\in\N$,
		\be 
		\P{\cB(v)\leq t}=\P{D^{(v_1\cdots v_{\ell-1})}\geq v_\ell\text{ for all }\ell\in[m]\text{, and }\sum_{\ell=1}^m \sum_{i=0}^{v_\ell-1}E^{(v_1\cdots v_{\ell-1})}_i\leq t}.
		\ee  
		We use the independence of the offspring random variables and the inter-birth times, together with a Chernoff bound with $\lambda'$, to bound the right-hand side from above by 
		\be 
		\e^{\lambda't}\prod_{\ell=1}^m \P{D\geq v_\ell} \E{\exp\bigg(-\lambda' \sum_{i=0}^{v_\ell-1}E_i\bigg)}.
		\ee 
		By summing over all individuals $v\in \cU_\infty$ (where $\P{\cB(\varnothing)\leq t}=1$ as $\cB(\varnothing)=0$), we obtain 
		\be \ba 
		\E{|\bp(t)|}&=\sum_{v\in \cU_\infty}\P{\cB(v)\leq t}\\
		&\leq 1+\e^{\lambda' t}\sum_{m=1}^\infty \sum_{v_1=1}^\infty\cdots \sum_{v_m=1}^\infty \prod_{\ell=1}^m \P{D\geq v_\ell} \E{\exp\bigg(-\lambda' \sum_{i=0}^{v_\ell-1}E_i\bigg)}.
		\ea \ee 
		As the terms in the product depend on distinct summation indices, we can interchange summation and product to yield
		\be 
		\sum_{m=0}^\infty \Bigg(\sum_{j=1}^\infty \P{D\geq j}\E{\exp\bigg(-\lambda' \sum_{i=0}^{j-1}E_i\bigg)}\Bigg)^m.
		\ee 
		By rewriting the inner sum and using that $D$ is independent of $(E_i)_{i\in\N}$, we obtain 
		\be 
		\sum_{j=1}^\infty \P{D\geq j}\E{\exp\bigg(-\lambda' \sum_{i=0}^{j-1}E_i\bigg)}=\E{\sum_{j=1}^D\exp\bigg(-\lambda'\sum_{i=0}^{j-1}E_i\bigg)}=\widehat \mu(\lambda'),
		\ee 
		where we recall $\widehat \mu$ from~\eqref{eq:muhat}. We finally obtain that $\E{|\bp(t)|}\leq \e^{\lambda' t}\sum_{m=0}^\infty \widehat\mu(\lambda')^m$, which is finite since   $\widehat\mu(\lambda')<1$. Hence, $|\bp(t)|<\infty$ almost surely, as desired.
	\end{proof}
	
	Theorem~\ref{thrm:inflife} then readily follows.
	
	\begin{proof}[Proof of Theorem~\ref{thrm:inflife}]
		Recall the stopping times $(\tau_n)_{n\in\N}$ from~\eqref{eq:taun}. By using Lemma~\ref{lemma:finbp} we observe that $\tau_n\to \infty$ almost surely conditionally on survival. As a result, by using Proposition~\ref{prop:embed} we obtain that Theorem~\ref{thrm:genpers} directly implies Theorem~\ref{thrm:inflife}. 
	\end{proof}

	The remainder of this section is devoted to proving Theorem~\ref{thrm:genpers}.
	
	\subsection{Preliminary results} 
	
	We first collect a few preliminary results that are used throughout the proof of Theorem~\ref{thrm:genpers}.
	
	\begin{lemma}[Lemma $3.4$ in~\cite{Iyer24}] \label{lemma:catch}
		Let $(X(i))_{i\in\N}$, $(X(i)')_{i\in\N}$, and $Y$ be mutually independent random variables, with $(X(i))\overset{\mathrm{d}}{=}(X(i)')_{i\in\N}$. Suppose that there exist $\lambda>0$ and $k\in\N$ such that 
		\be \label{eq:infprodexp}
		\prod_{i=k}^\infty \E{\e^{\lambda(X(i)'-X(i))}}<\infty. 
		\ee 
		Then, 
		\be 
		\mathbb P\bigg(\exists j\in\N_0\colon  Y+\sum_{i=1}^{k+j}X(i)\leq \sum_{i=k}^{k+j}X(i)'\bigg)\leq \bigg(\prod_{i=k}^\infty\mathbb E\Big[\e^{\lambda(X(i)'-X(i))}\Big]\bigg)\mathbb E\bigg[\exp\bigg(-\lambda\bigg(Y+\sum_{i=1}^{k-1}X(i)\bigg)\bigg)\bigg].
		\ee 
	\end{lemma} 
	
	In short, the proof of Lemma~\ref{lemma:catch} is based on the fact that the random variable $\exp(\lambda \sum_{i=k}^{k+j}( X(i)-X(i)'))$ is a sub-martingale combined with Chernoff's inequality.
	
	\begin{lemma}[Proposition $1.15$ in~\cite{Iyermonopoly24}] \label{lemma:divsum} 
		Suppose that $(S_i)_{j\in\N}$ is a sequence of mutually independent, symmetric random variables. If $\sum_{i=1}^\infty S_i$ diverges almost surely, then 
		\be 
		\limsup_{n\to\infty}\sum_{i=1}^n S_i=\infty\qquad \text{and}\qquad \liminf_{n\to\infty} \sum_{i=1}^n S_i=-\infty,\qquad\text{almost surely.} 
		\ee 
	\end{lemma} 
	
	The proof of Lemma~\ref{lemma:divsum} is based on the fact that $\sum_{i=1}^n S_i$ is a martingale, combined with the fact that divergent martingales are not bounded almost surely. 
	
	\begin{lemma}\label{lemma:sumexp}
		Let $(E_i)_{i\in\N_0}$ and $(E_i')_{i\in\N_0}$ be independent exponential random variables, where $E_i$ and $E_i'$ have rate $\lambda_i>0$ for each $i\in\N_0$. Then, 
		\be 
		\sum_{i=0}^\infty (E_i-E_i')\text{ converges almost surely} \quad \Leftrightarrow \quad \sum_{i=0}^\infty \frac{1}{\lambda_i^2}<\infty. 
		\ee 
	\end{lemma}
	
	\begin{remark}\label{rem:sumconvass}
		$(i)$ If we set $X(i)=E_i$ and $X(i)'=E_i'$ in Lemma~\ref{lemma:catch},  it also follows that for any  $\lambda>0$ there exist $k\in\N$ such that~\eqref{eq:infprodexp} is satisfied when 
		\be \label{eq:sum}
		\sum_{i=0}^\infty \frac{1}{\lambda_i^2}<\infty.
		\ee
		Indeed, as the sum is finite it follows that $\lambda_i\to\infty$ with $i$, so that for any $\lambda>0$ we can take $k$ large enough such that $\lambda_i>\lambda$ for all $i\geq k$, so that 
		\be 
		\E{\e^{\lambda(E_i'-E_i)}}=\frac{\lambda_i^2}{(\lambda_i-\lambda)(\lambda_i+\lambda)}=\frac{1}{1-\lambda^2/\lambda_i^2}.
		\ee 
		The product in~\eqref{eq:infprodexp} hence is finite when the sum in~\eqref{eq:sum} is finite.
		
		$(ii)$ If we apply Lemma~\ref{lemma:sumexp} with rates $\lambda_i=b(i)+d(i)$, the condition in~\eqref{eq:sum} is  Assumption~\eqref{ass:varphi2}. \ensymboldremark
	\end{remark} 
	
	\begin{proof}
		To prove the convergence of the series of exponential random variables, we check the conditions of Kolmogorov's three series theorem (see e.g.\ \cite[Theorem $2.5.4$]{Durr19}). This theorem states that a random series $\sum_{i=0}^\infty X_i$ of independent random variables $(X_i)_{i\in\N_0}$ converges almost surely if the following conditions holds for some $A>0$, and only if the following conditions hold for any $A>0$:
		\be 
		(i)\  \sum_{i=0}^\infty \P{|X_i|\geq A}<\infty\quad (ii)\ \sum_{i=0}^\infty \E{X_i\ind_{\{|X_i|\leq A\}}}<\infty \quad (iii)\ \sum_{i=0}^\infty \Var(X_i\ind_{\{|X_i|\leq A\}})<\infty. 
		\ee 
		We notice that for symmetric random variables, the second condition is always satisfied. It thus suffices to focus on conditions $(i)$ and $(iii)$, for $X_i=E_i-E_i'$. 
		
		Let us first prove the `if' direction. We suppose that $\sum_{i=0}^\infty \lambda_i^{-2}<\infty$. Then, 
		\be 
		\sum_{i=0}^\infty \Var(X_i\ind_{\{|X_i|\leq A\}})=\sum_{i=0}^\infty \E{X_i^2\ind_{\{|X_i|\leq A\}}}\leq  \sum_{i=0}^\infty \E{X_i^2}=\sum_{i=0}^\infty2\Var(E_i)= \sum_{i=0}^\infty \frac{2}{\lambda_i^2},
		\ee  
		which is finite by assumption. Condition $(iii)$ is hence satisfied. Also, 
		\be \label{eq:PAbound}
		\sum_{i=0}^\infty \P{|X_i|\geq A} \leq \sum_{i=0}^\infty \P{E_i+E_i'\geq A} \leq 2\sum_{i=0}^\infty \P{E_i \geq\frac A2}= 2\sum_{i=0}^\infty\e^{-\lambda_i A/2}
		\ee 
		By assuming that $\sum_{i=0}^\infty \lambda_i^{-2}<\infty$, it follows that $\lambda_i\to\infty$ with $i$. Hence, there exists $I\in\N$ such that $\e^{-\lambda_iA/2}\leq 1/\lambda_i^2$, from which Condition $(i)$ follows.
		
		For the `only if' direction, we argue by contradiction. We assume that there exists $A>0$ such that either Condition $(i)$ or $(iii)$ does not hold, and additionally assume that  $\sum_{i=0}^\infty \lambda_i^{-2}<\infty$. We show that this leads to a contradiction. 
		
		Suppose first that Condition $(i)$ is not satisfied for some $A>0$. By the bound in~\eqref{eq:PAbound} and the assumption that $\sum_{i=0}^\infty \lambda_i^{-2}<\infty$, this directly leads to a contradiction. Hence, now assume that Condition $(i)$ is satisfied for all $A>0$, but that Condition $(iii)$ is not satisfied for some $A>0$. Then, by the second Borel-Cantelli lemma, $|X_i|\geq A$ only finitely often almost surely, so that 
		\be 
		\sum_{i=0}^\infty \Var(X_i\ind_{\{|X_i|\leq A\}})=\infty \quad \Leftrightarrow \quad \sum_{i=0}^\infty \Var(X_i)=\infty.
		\ee 
		However, the equality on the right-hand side contradicts the assumption that $\sum_{i=0}^\infty \lambda_i^{-2}<\infty$, since $\Var(X_i)=\E{X_i^2}=2\Var{E_i}=2\lambda_i^{-2}$. 
	\end{proof}
	
	\subsection{Proving Theorem~\ref{thrm:genpers}(a)}

	For $u,v\in\cU_\infty$, define the event 
	\be \label{eq:win}
	\text{Win}(u,v)\coloneq \{\exists T\geq 0: \forall t\geq T\ \deg_t(u)\geq \deg_t(v)\}.
	\ee 
	If an individual $u$ has not been born in $\bp(t)$, we use the convention that $\deg_t(u)=-\infty$. 
	
	We then have the following result, which is based on Lemma $3.6$ in~\cite{Iyer24}.
	
	\begin{lemma}\label{lemma:tauwin}
		Recall the CTPAVD branching process $(\bp(t))_{t\geq 0}$ defined in~\eqref{eq:bpdef}, with offspring, inter-birth times, and lifetimes as defined in (the paragraph before)~\eqref{eq:birthlifetime}. Assume that Assumption~\eqref{ass:A1} is satisfied and Assumption~\eqref{ass:A2} is not satisfied. The following statements are true:
		\begin{enumerate}
			\item  $\lim_{t\to\infty}\max_{v\in \bp(t)}\deg_t(v)=\infty$ $\mathbb P_\cS$-almost surely. 
			\item If Assumption~\eqref{ass:varphi2} is satisfied, then for all $u,v\in\cU_\infty$,
			\be 
			\Ps{\mathrm{Win}(u,v)\cup\mathrm{Win}(v,u)}=1.
			\ee  
			\item If Assumption~\eqref{ass:varphi2} is not satisfied, there exists a function $\phi:\N\to\N$ such that $\phi(j)>j$ for all $j\in\N$ such that for any $u\in\cU_\infty$ and $j\in\N$, 
			\be 
			\P{\exists k\in\{j+1,\ldots, \phi(j)\}: \sum_{i=j}^{k-1} E^{(u)}_i>\sum_{i=0}^{k-1} E^{(uj)}_i}\geq \frac12.
			\ee  
		\end{enumerate}
	\end{lemma}
	
	\begin{proof}
		$(a)$ As Assumption~\eqref{ass:A2} is not satisfied, it follows from~\cite[Lemma 5.3 and Corollary 6.4]{HeyLod25} that offspring random variable $D$ is such that $\P{D=\infty}>0$. We can hence define $V$ be the first individual born in $(\bp(t))_{t\geq 0}$ such that $D^{(V)}=\infty$. Conditionally on survival such an individual $V$ exists almost surely. Furthermore, the birth-time of $V$ is finite almost surely. As a result, for all $t\geq \cB(V)$, 
		\be 
		\max_{v\in \bp(t)}\deg_t(v)\geq \deg_t(V), 
		\ee 
		and the right-hand side tends to infinity with $t$ since $D^{(V)}=\infty$.
		
		$(b)$ The event $\text{Win}(u,v)\cup \text{Win}(v,u)$ is trivially satisfied if at most one of the two individuals $u$ and $v$ is born. Given that both vertices are born, again the event is trivially satisfied if at most one individual produces infinite offspring. As a result, we define 
		\be
		\cE(u,v)\coloneq  \{\cB(u),\cB(v)<\infty\}\cap \{D^{(u)},D_v=\infty\},
		\ee 
		and we can thus write 
		\be
		\P{\mathrm{Win}(u,v)\cup\mathrm{Win}(v,u)}=	\P{(\mathrm{Win}(u,v)\cup\mathrm{Win}(v,u))\cap \cE(u,v)}+\P{\cE(u,v)^c}.
		\ee 
		Our aim is to show that 
		\be \label{eq:goal}
		\P{(\mathrm{Win}(u,v)\cup\mathrm{Win}(v,u))\cap\cE(u,v)}=\P{\cE(u,v)}.
		\ee 
		Lemma~\ref{lemma:sumexp} tells us that  Assumption~\eqref{ass:varphi2} being satisfied is equivalent to the infinite series $\sum_{i=0}^{\infty} (E_i - E_i')$ converging almost surely. Then, for $u, v \in \cU_\infty$, on the event $\cE(u,v)$ we have
		\be
		\left| \cB(u) - \cB(v) + \sum_{i=0}^{\infty} (E^{(u)}_i - E^{(v)}_i) \right| < \infty \quad \text{almost surely,}
		\ee
		since $E^{(u)}_i\overset \dd=E_i$ and $E^{(v)}_i\overset \dd= E_i'$. 	Assume, without loss of generality, that $ |u| \leq |v| $. Let $ N $ be chosen such that $ \cB(v) $ is independent of $ (E^{(u)}_n)_{n \geq N} $. For example, if $ u $ is a non-parent ancestor of $ v $, we may choose the value $ N $ such that $ u(N-1) $ is the ancestor of $ v $, whilst if $ v = uj $ for $ j \in \N $, we can choose $ N > j $. Then, the sequences $ (E^{(u)}_n)_{n \geq N} $ and $ (E^{(v)}_n)_{n \geq N} $ are independent. Therefore, on the event $\cE(u,v)$ there exists $ K_0 \in \N $ such that  either
		\begin{align}
			\cB(u) - \cB(v) + \sum_{j=0}^{k-1} (E^{(u)}_j - E^{(v)}_j) > 0\qquad\text{for all $k\geq K_0$,}\label{eq:winv}
			\intertext{or}
			\cB(u) - \cB(v) + \sum_{j=0}^{k-1} (E^{(u)}_j - E^{(v)}_j) < 0 \qquad\text{for all $k\geq K_0$,}\label{eq:winu}
		\end{align}
		is satisfied $\mathbb P_\cS$-almost surely. In other words, we either have $ \cB(uk) > \cB(vk) $ for all $k\geq K_0$, or we have $ \cB(uk) < \cB(vk) $ for all $k\geq K_0$ $\mathbb P_\cS$-almost surely on the event $\cE(u,v)$. This in turn implies that for all $ k $ sufficiently large $ u $ has offspring $k$  before $ v $ has offspring $k$ or vice versa. Thus, we have that~\eqref{eq:goal} is satisfied, which concludes the proof of part $(b)$.
		
		$(c)$ We again use Lemma~\ref{lemma:sumexp} to conclude that Assumption~\eqref{ass:varphi2} not being satisfied is equivalent to the infinite series $\sum_{i=0}^\infty (E_i-E_i')$ diverging almost surely, and use again that $E^{(u)}_i\overset \dd= E_i$ and $E^{(uj)}_i\overset \dd = E_i'$ for any $j\in\N$ and any $u\in\cU_\infty$. We then use Lemma~\ref{lemma:divsum} to obtain that for any $j\in\N$ we have
		\be 
		\limsup_{n\to\infty}\sum_{i=j}^n (E^{(u)}_i-E^{(uj)}_i)=\infty \quad\text{almost surely}.
		\ee
		Since also $\sum_{i=0}^{j-1} E^{(u)}_i<\infty$ almost surely, we thus obtain that 
		\be 
		\P{\exists k>j: \sum_{i=j}^k E^{(u)}_i>\sum_{i=0}^k E^{(uj)}_i}=1.
		\ee 
		Then,  by the monotone convergence theorem, 
		\be 
		\lim_{n\to\infty} \P{\exists k\in\{j+1,\ldots,n\}: \sum_{i=j}^{k-1} E^{(u)}_i>\sum_{i=0}^{k-1} E^{(uj)}_i}=1.
		\ee 
		As a result, we can define $\phi(j)$ as 
		\be 
		\phi(j)\coloneq \inf\bigg\{n>j: \mathbb P\bigg(\exists k\in\{j+1,\ldots,n\}:\sum_{i=j}^{k-1} E^{(u)}_i>\sum_{i=0}^{k-1} E^{(uj)}_i\bigg)\geq \frac12\bigg\},
		\ee 
		which concludes the proof.
	\end{proof}
	
	We are now ready to prove part $(a)$ of Theorem~\ref{thrm:genpers}, i.e.\ that almost surely no persistent $m$-hub emerges $\mathbb P_\cS$-almost surely when Assumption~\eqref{ass:varphi2} is not satisfied (i.e.\ when the series $\sum_{i=0}^\infty (E_i-E_i')$ diverges almost surely. 
	
	\begin{proof}[Proof of Theorem~\ref{thrm:genpers}$(a)$]
		Fix $m\in\N$ and $u\in\cU_\infty$. We observe that, by $(a)$ of Lemma~\ref{lemma:tauwin}, conditionally on survival $u$ can be a persistent $m$-hub only when $\cB(u)<\infty$ and $D^{(u)}=\infty$. That is, $u$ needs to be born in the branching process in finite time and produce an infinite number of children. Hence, we can write
		\be \ba 
		\Ps{u\text{ is a persistent $m$-hub}}&=\mathbb P_\cS\big(\{u\text{ is a persistent $m$-hub}\}\cap\{\cB(u)<\infty, D^{(u)}=\infty\}\big)\\ 
		&= \frac{1}{\P{\cS}}\mathbb P\big(\{u\text{ is a persistent $m$-hub}\}\cap\{\cB(u)<\infty, D^{(u)}=\infty\}\big),
		\ea \ee 
		where the last step follows from the fact that the event $\{\cB(u)<\infty, D^{(u)}=\infty\}$ implies survival. We then use that, upon the event $\{\cB(u)<\infty, D^{(u)}=\infty\}$, $u$ can be a persistent $m$-hub only when for some $T\geq \cB(u)$  it has a larger offspring than all but at most $m-1$ many of its children for all time $t\geq T$. That is, 
		\be\ba  
		\{u\text{ is a persistent $m$-hub}\}&\subseteq \bigcup_{T\geq \cB(u)}\bigcup_{\substack{C\subset \N\\ |C|<m}}\{\forall j\in\N\setminus C \ \forall t\geq T: \deg_t(u)\geq \deg_t(uj)\}\\ 
		&=\bigcup_{K\in\N}\bigcup_{\substack{C\subset \N\\ |C|<m}}\{\forall j\in \N\setminus C\ \forall k\geq K: \cB(uk)\leq \cB(ujk)\}.
		\ea\ee 
		Here, the second step follows from the fact that any individual almost surely takes an infinite amount of time to produce infinite offspring by Assumption~\eqref{ass:A1}. We thus obtain 
		\be \ba \label{eq:persmub}
		\mathbb P_\cS{}&(u\text{ is a persistent hub})\\
		\leq{}& \frac{1}{\P{\cS}}\sum_{K=1}^\infty \sum_{\substack{C\subset \N\\ |C|<m}}\mathbb P\big(\{\forall j\in \N\setminus  C\ \forall k\geq K: \cB(uk)\leq \cB(ujk)\}\cap \{\cB(u)<\infty, D^{(u)}=\infty\}\big).
		\ea \ee 
		We now show that each probability on the right-hand side equals zero, by showing that 
		\be \label{eq:omitprob}
		\mathbb P\big(\{\exists j\in \N\setminus C\, \exists k\geq K: \cB(uk)> \cB(ujk)\}\cap \{\cB(u)<\infty, D^{(u)}=\infty\}\big)=\mathbb P\big(\cB(u)<\infty, D^{(u)}=\infty\big).
		\ee 
		On the event $\{\cB(u)<\infty\}$, we can write
		\be \label{eq:birthineq}
		\cB(uk)>\cB(ujk) \ \ \Leftrightarrow \ \ \sum_{i=0}^{k-1} E^{(u)}_i>\sum_{i=0}^{j-1} E^{(u)}_i +\sum_{i=0}^{k-1} E^{(uj)}_i  \ \ \Leftrightarrow \ \  \sum_{i=j}^{k-1} E^{(u)}_i>\sum_{i=0}^{k-1} E^{(uj)}_i.
		\ee 
		We now recall the function $\phi$ in part $(c)$ of Lemma~\ref{lemma:tauwin}. We let $(j_\ell)_{\ell\in\N}=j_\ell(C,K)$ be the subsequence $j_1=1+\max\{\max C,K\}$ and $j_{\ell+1}\coloneq \phi(j_\ell)$ for all $\ell\in\N$. As $\phi(j)>j$ for all $j\in\N$, it follows that the subsequence $(j_\ell)_{\ell\in\N}$ is strictly increasing and thus $j_\ell>\max C$ and $j_\ell>K$ for all $\ell\in\N$. Then,
		\be\ba 
		\mathbb P{}&\big(\{\exists j\in \N\setminus C\, \exists k\geq K: \cB(uk)> \cB(ujk)\}\cap \{\cB(u)<\infty, D^{(u)}=\infty\}\big)\\ 
		&=\mathbb P\bigg(\bigg\{\exists j\in \N\setminus C\, \exists k\geq K:\sum_{i=j}^{k-1} E^{(u)}_i>\sum_{i=0}^{k-1} E^{(uj)}_i \bigg\}\cap \{\cB(u)<\infty, D^{(u)}=\infty\}\bigg)\\ 
		&\geq \mathbb P\bigg(\bigg\{\exists \ell\in \N\, \exists k\in\{j_\ell+1,\ldots, \phi(j_\ell)\}\colon \sum_{i=j_\ell}^{k-1} E^{(u)}_i>\sum_{i=0}^{k-1}E^{(uj_\ell)}_i \bigg\}\cap \{\cB(u)<\infty, D^{(u)}=\infty\}\bigg).
		\ea\ee 
		For ease of writing, we define
		\be 
		\cE(u,\ell)\coloneq \bigg\{\exists k\in\{j_\ell+1,\ldots, \phi(j_\ell)\}\colon\sum_{i=j_\ell}^{k-1} E^{(u)}_i>\sum_{i=0}^{k-1} E^{(uj_\ell)}_i \bigg\}\qquad\text{for }\ell\in\N.
		\ee 
		We then arrive at the lower bound
		\be 
		\mathbb P\big(\{\cE(u,\ell) \text{ occurs for infinitely many }\ell\}\cap \{\cB(u)<\infty, D^{(u)}=\infty\}\big).
		\ee 
		By the choice of the sub-sequence $(j_\ell)_{\ell\in\N}$, the events $(\cE(u,\ell))_{\ell\in\N}$ are mutually independent, as they depend on distinct (and thus independent) inter-birth times. They are  also independent of  $\cB(u)$ and $D^{(u)}$, so that we can factorize the probability into
		\be 
		\mathbb P\big(\{\cE(u,\ell) \text{ occurs for infinitely many }\ell\}\big)\mathbb P\big(\cB(u)<\infty , D^{(u)}=\infty\big).
		\ee 
		By $(c)$ of Lemma~\ref{lemma:tauwin} we have that $\P{\cE(u,\ell)}\geq \frac12$ for all $\ell\in\N$. By their mutual independence and using the reverse Borel-Cantelli lemma, we thus obtain that infinitely many event $\cE(u,\ell)$ occur almost surely. This yields~\eqref{eq:omitprob}. Using this in~\eqref{eq:persmub} implies that 
		\be 
		\Ps{u\text{ is a persistent $m$-hub}}=0, 
		\ee 
		so that a countable union bound yields
		\be 
		\Ps{\exists u\in\cU_\infty: u\text{ is a persistent $m$-hub}}=0,
		\ee 
		which concludes the proof.
	\end{proof}
	
	\subsection{Proving Theorem~\ref{thrm:genpers}(b)} 
	
	We then move to part $(b)$ of Theorem~\ref{thrm:genpers}. To this end, we define for $u,v\in\cU_\infty$ such that $\cB(u)\leq \cB(v)<\infty$, the event
	\be 
	\{v\text{ catches up to }u\}=\{\exists j\in \N: \cB(vj)\leq \cB(uj)\} .
	\ee 
	That is, for an individual $v$ that is born at the same time as or after $u$, we say that $v$ catches up to $u$ if at some point its offspring is at least as large as that of $u$ (not including having zero children). Heuristically, one would expect that individuals that are born `late' in the branching process are unlikely to catch up to the root (if the root were to produce infinite offspring). The following result makes this heuristic precise.
	
	\begin{proposition}\label{prop:catch}
		Recall the CTPAVD branching process $(\bp(t))_{t\geq 0}$ defined in~\eqref{eq:bpdef}, with offspring, inter-birth times, and lifetimes as defined in (the paragraph before)~\eqref{eq:birthlifetime}. Assume that Assumptions~\eqref{ass:A1}, \eqref{ass:mufin}, and~\eqref{ass:varphi2} are satisfied and Assumption~\eqref{ass:A2} is not satisfied. Then, there exists $K\in\N$ such that
		\be 
		\Es{|\{v\in \cU_\infty: v_1\geq K, D^{(v)}=\infty,v\text{ catches up to }\varnothing\}|\ind_{\{D_\varnothing=\infty\}}}<\infty.
		\ee 
	\end{proposition}
	
	\begin{proof}
		We write the size of the set in the expected value as a sum of indicators. That is, we fix $v=v_1v_2\cdots v_m\in\cU_\infty$ for some $m\in\N$, $v_1\geq K$, and $v_2,\ldots v_m\in\N$,  and consider 
		\be \ba 
		\mathbb P_\cS{}&(v\text{ catches up to }\varnothing\text{ and }D^{(v)}=D^{(\varnothing)}=\infty)\\ 
		&=\P{\{D^{(v)}=D^{(\varnothing)}=\infty\}\cap\{\cB(v)<\infty\}\cap\{\exists j\in \N_0: \cB(v(v_1+j))\leq \cB(v_1+j)\}}.
		\ea\ee 
		We note that we can omit the conditioning on survival, since the event that $\varnothing$ produces infinite offspring guarantees survival. We also need only consider children numbered $v_1+j$ with $j\geq 0$ in the second line, since $\varnothing$ gives birth to $v_1$ many children before $v$ is born.  The event $\cB(v)<\infty$ is equivalent to $D^{(v_1\cdots v_\ell)}\geq v_{\ell+1}$ for all $\ell\in[m-1]$ (the event $\{D^{(\varnothing)}=\infty\}$ already implies that $v_1$ is born). When $D^{(v)}=\infty$, for $j\in\N_0$ the event that $v$ gives birth to its $(v_1+1+j)^{\mathrm{th}}$ child before $\varnothing$ does is equivalent to
		\be 
		\sum_{\ell=1}^{m-1}\sum_{i=0}^{v_{\ell+1}-1} E^{(v_1\cdots v_\ell)}_i+\sum_{i=0}^{v_1+j}E^{(v)}_i\leq \sum_{i=v_1}^{v_1+j}E^{(\varnothing)}_i. 
		\ee 
		This yields, by the independence of the offspring and inter-birth time random variables, 
		\be \ba 
		\mathbb P_\cS{}&(v\text{ catches up to }\varnothing\text{ and }D^{(v)}=D^{(\varnothing)}=\infty)\\ 
		={}&\P{ \exists j\in\N_0\colon  \sum_{\ell=1}^{m-1}\sum_{i=0}^{v_{\ell+1}-1} \!\!\!E^{(v_1\cdots v_\ell)}_i+\!\sum_{i=0}^{v_1+j}\!E^{(v)}_i\leq \sum_{i=v_1}^{v_1+j}\!E^{(\varnothing)}_i}\P{D=\infty}^2\prod_{\ell=1}^{m-1}\P{ D\geq v_{\ell+1}}.
		\ea\ee 
		We bound $\P{D=\infty}\leq \P{D\geq v_1}$ and apply Lemma~\ref{lemma:catch} (the condition in~\eqref{eq:infprodexp} is satisfied for any $\lambda>0$ by choosing $K=K(\lambda)$ large, see  Remark~\ref{rem:sumconvass}$(i)$ and $(ii)$) with 
		\be 
		Y=\sum_{\ell=1}^{m-1}\sum_{i=0}^{v_{\ell+1}-1} \!\!\!E^{(v_1\cdots v_\ell)}_i, \qquad X(i)=E^{(v)}_{i-1}, \qquad X(i)'=E^{(\varnothing)}_{i-1}, \qquad k=v_1+1,
		\ee 
		to arrive at the upper bound
		\be \ba 
		\sum_{\substack{v\in\cU_\infty\\ v_1\geq K}}{}&\Ps{v\text{ catches up to }\varnothing\text{ and }D_\varnothing=\infty}\\ 
		\leq{}& \sum_{m=1}^\infty \sum_{v_1=K}^\infty \sum_{v_2,\ldots v_m=1}^\infty \prod_{i=v_1}^\infty \E{\e^{\lambda(E_i'-E_i)}}\prod_{\ell=1}^m \P{D\geq v_\ell}\E{\exp\bigg(-\lambda \sum_{i=0}^{v_\ell-1}E_i\bigg)}.
		\ea\ee 
		We observe that $\E{\exp(\lambda(E_i'-E_i))}\geq 1$ by Jensen's inequality). Hence, as $v_1\geq K$ we can extend the range $i\geq v_1$ in the first product to $i\geq K$, and also  extend the range $v_1\geq K$ in the second sum to $v_1\geq 1$, to obtain an upper bound. As the different terms in the second product depend on distinct summation indices $v_1,\ldots, v_m$, we can switch the summation and product to obtain
		\be\ba 
		\prod_{i=K}^\infty{}& \E{\e^{\lambda(E_i'-E_i)}}\sum_{m=1}^\infty \bigg(\sum_{j=1}^\infty \P{D\geq j} \E{\exp\bigg(-\lambda \sum_{i=0}^{j-1} E_i\bigg)}\bigg)^m\\ 
		={}&\prod_{i=K}^\infty \E{\e^{\lambda(E_i'-E_i)}} \sum_{m=1}^\infty  \E{\sum_{j=1}^D\exp\bigg(-\lambda \sum_{i=1}^j E_i\bigg)}^m=\prod_{i=K}^\infty \E{\e^{\lambda(E_i'-E_i)}} \sum_{m=1}^\infty\widehat\mu(\lambda)^m,
		\ea \ee 
		where we use the definition of $\widehat\mu$, as in~\eqref{eq:muhat}, in the final step. We now choose $\lambda$ large enough so that $\widehat\mu(\lambda)<1$, which is possible by Assumption~\eqref{ass:mufin} (as follows from~\eqref{eq:muless1}). It is then clear that the upper bound is finite, which concludes the proof.
	\end{proof}
	
	We define
	\be 
	\pre(u)\coloneq \{v\in \cU_\infty: \cB(v)\leq \cB(u)\}\qquad\text{for }u\in\cU_\infty 
	\ee 
	as the set of \emph{predecessors} of $u$ in the branching process $\bp$. Fix $m\in\N$. We then define 
	\be \ba \label{eq:P}
	\cP_m\coloneq {}&\bigcup_{\substack{C\subseteq \pre(u)\\ |C|<m}}\{u\in\cU_\infty: D^{(u)}=\infty\text{, } u \text{ catches up to all }v\in \pre(u)\setminus C\}\\ 
	={}&\bigcup_{j=0}^{m-1} \{u\in\cU_\infty:D^{(u)}=\infty\text{, } u \text{ catches up to all but $j$ many }v\in \pre(u)\}\\ 
	={}&\{u\in\cU_\infty:D^{(u)}=\infty\text{, } u \text{ catches up to all but at most $m-1$ many }v\in \pre(u)\},
	\ea \ee 
	to be the set of all individuals $u$ whose offspring diverges to infinity and \emph{at some point in time} obtain a larger offspring than all but at most $m-1$ many of its predecessors. Note that this point in time need not be the same for all these predecessors. That is, it can catch up to a predecessor $v_1$ at time $t_1$ and catch up to a predecessor $v_2$ at time $t_2\neq t_1$.  A (potential) persistent $m$-hub $u$ needs to attain the $m^{\mathrm{th}}$ largest offspring for all but a finite amount of time. As a result, there can be at most $m-1$ many predecessors of $u$ that it does not catch up to. The main ingredient for the proof of Theorem~\ref{thrm:genpers}$(b)$ is to `extend' the result in Proposition~\ref{prop:catch} to the set $\cP_m$, as follows. 
	
	\begin{proposition}[$\cP_m$ is finite almost surely]\label{prop:P}
		Recall the CTPAVD branching process $(\bp(t))_{t\geq 0}$ defined in~\eqref{eq:bpdef}, with offspring, inter-birth times, and lifetimes as defined in (the paragraph before)~\eqref{eq:birthlifetime}. Assume that Assumptions~\eqref{ass:A1}, \eqref{ass:mufin}, and~\eqref{ass:varphi2} are satisfied and Assumption~\eqref{ass:A2} is not satisfied. Fix $m\in\N$ and recall $\cP_m$ from~\eqref{eq:P}. Then, $|\cP_m|$ is finite $\mathbb P_\cS$-almost surely.
	\end{proposition}
	
	As only finitely many individuals can catch up to all but at most $m-1$ of their predecessors, and out of all pairs of these vertices one must `win', as in the sense of~\eqref{eq:win}, Theorem~\ref{thrm:genpers}$(b)$ readily follows:
	
	\begin{proof}[Proof of Theorem~\ref{thrm:genpers}$(b)$  subject to Proposition~\ref{prop:P}]
		Fix $m\in\N$. By Lemma~\ref{lemma:tauwin}$(b)$, 
		\be \label{eq:capsetswin}
		\Ps{\bigcap_{\substack{S\subset \cU_\infty \\ |S|<\infty}}\bigcap_{u,v\in S}\text{Win}(u,v)\cup\text{Win}(v,u)}=1,
		\ee 
		as the intersections are over countable sets. For any finite set $S$, 
		\be 
		\bigcap_{u,v\in S}\text{Win}(u,v)\cup\text{Win}(v,u)\subseteq\{\exists w\in S\ \exists T\geq 0: \deg_t(w)=\max_{u\in S}\deg_t(u)\text{ for all }t\geq T\}.
		\ee 
		As a result, with $S=\cP_1$ being $\mathbb P_\cS$-almost surely finite by Proposition~\ref{prop:P}, it follows that 
		\be \label{eq:winner}
		\Ps{\exists w_1\in \cP_1\ \exists T_1\geq 0: \deg_t(w_1)=\max_{u\in \cP_1}\deg_t(u)\text{ for all }t\geq T_1}=1.
		\ee 
		Let us call   this individual the `winner' of $\cP_1$. Note that	this individual $w_1$ is unique, since otherwise, if there were winners $w_1$ and $w_1'$, then $\deg_t(w_1)=\deg_t(w_1')$ for all $t\geq \wt T$ and some time $\wt T$. But this would imply that infinitely many of their inter-birth times are of the same length, which occurs with probability zero.  We can then also recursively obtain that, for any $m\geq 2$ and with $\cW_m\coloneq \{w_1,\ldots,w_{m-1}\}$, 
		\be
		\Ps{\exists w_m\in \cP_m\setminus \cW_m\ \exists T_m\geq 0: \deg_t(w_m)= \max_{u\in \cP_m\setminus\cW_m} \deg_t(u)\text{ for all }t\geq T_m}=1.
		\ee 
		Now, for any $m\in\N$ and any $v\not\in \cP_m$ there exist at least $m$ predecessors  $v'_1,\ldots, v'_{m}$ of $v$ such that $\deg_t(v'_i)>\deg_t(v)$ for all $t\geq \cB(v'_i1)$ and all $i\in[m]$ (that is, we do not include the time when both $v$ and $v_i'$ have no children). Either $v_i'\in \cP_m$ for all $i\in[m]$ or there exists yet another $m$ predecessors $v''_1,\ldots, v''_{m}$ of $v_I'$ for some $I\in[m]$ such that $\deg_t(v_i'')>\deg_t(v_I')>\deg_t(v)$ for all $t\geq \cB(v''_i1)$ and all $i\in[m]$. This procedure terminates, since the number of predecessors of $v$ is finite, so that there exist  $u_1,\ldots, u_{m}\in \cP_m  $ such that $\deg_t(u_i)>\deg_t(v)$ for al $t\geq \cB(u_i1)$ and for all $i\in[m]$. As a result, for any $v\not\in \cP_m  $ there does not exist a time $t$ such that $v$ has the $m^{\mathrm{th}}$ largest offspring among all individuals in $\bp(t)$. 
		
		We now show that a persistent $m$-hub exists $\mathbb P_\cS$-almost surely for any $m\in\N$ by induction. Fix first $m=1$. The unique winner $w_1$ of $\cP_1=\cP_1\setminus \cW_1$ (with $\cW_1$ the empty set) has the largest offspring among all in $\cP_1$ for all $t\geq T_1$ by~\eqref{eq:winner}. Since any $v\not\in \cP_1$ does not attain the largest offspring for any time $t$, it follows that $w_1$ has the largest offspring among \emph{all} individuals for all time $t\geq T_1$. This implies that a unique persistent $1$-hub $I^{(1),\cont}=w_1$ exists $\mathbb P_\cS$-almost surely. 
		
		Now, suppose that a unique persistent $1$-hub $I^{(1),\cont}=w_2$, $2$-hub $I^{(2),\cont}=w_2$,$\ldots$, $(m-1)$-hub $I^{(m-1),\cont}=w_{m-1}$ exist $\mathbb P_\cS$-almost surely for some $m\geq 2$. By~\eqref{eq:winner} there exists a unique winner $w_m$ in $\cP_m\setminus\cW_m$. As a $1$-hub, $2$-hub,$\ldots$, $(m-1)$-hub exist, $w_m$ can have at most the $m^{\mathrm{th}}$ largest offspring in $\bp(t)$ for $t\geq \max_{i\in [m-1]}T_i$. Since any $v\not\in \cP_m $ never attains the $m^{\mathrm{th}}$ largest offspring in $\bp$ and eventually have fewer offspring than $w_m$, it follows that $w_m$ attains the $m^{\mathrm{th}}$ largest offspring in $\bp(t)$ for all time $t\geq T_m$. This implies that there exists a unique  persistent $m$-hub in $\bp$ $\mathbb P_\cS$-almost surely.
		
		Finally, we prove~\eqref{eq:distdeg}. Fix two individuals $u,v\in\cU_\infty$ such that $\cB(u),\cB(v)<\infty$ and $D^{(u)}=D^{(v)}=\infty$ and consider their number of children $\deg_t(u)$ and $\deg_t(v)$. Suppose that with positive probability there exists an increasing and diverging sequence $(t_k)_{k\in\N}$ of times, such that $\deg_{t_k}(u)=\deg_{t_k}(v)$ for all $k\in\N$. As the number of children of $u$ and $v$ tend to infinity with $k$ almost surely conditionally on survival (and that  both are born in the branching processes), this equality is equivalent to the existence of a sequence $(n_\ell)_{\ell\in\N}$ such that 
		\be 
		\cB(u)+\sum_{i=0}^{n_\ell-1}E^{(u)}_i\leq \cB(v)+\sum_{i=0}^{n_\ell-1}E^{(v)}_i<\cB(u)+\sum_{i=0}^{n_\ell}E^{(u)}_i\qquad \text{for all }\ell\in\N.
		\ee   
		We can rewrite this as 
		\be \label{eq:eqdeg}
		0\leq \cB(v)-\cB(u)+\sum_{i=0}^{n_\ell-1}E^{(v)}_i-E^{(u)}_i\leq E^{(u)}_{n_\ell}\qquad \text{for all }\ell\in\N.
		\ee 
		By Lemma~\ref{lemma:sumexp} and Remark~\ref{rem:sumconvass} the series $\sum_{i=0}^\infty( E^{(v)}_i-E^{(u)}_i)$ converges almost surely. Hence, the distribution of the random variable $\cB(v)-\cB(u)+\sum_{i=0}^\infty (E^{(v)}_i-E^{(u)}_i)$ does not contain an atom on $\R$ by~\cite[Lemma 3.2]{Iyer24}. 
		However, for any $\eps>0$, the inequality $E^{(u)}_{n_\ell}<\eps$ will hold for infinitely many $\ell$ almost surely by the Borel-Cantelli lemma, since 
		\be 
		\sum_{i=0}^\infty \P{E_i\geq \eps}\leq \sum_{i=0}^\infty \e^{-(b(i)+d(i))\eps}\leq \sum_{i=0}^K \e^{-(b(i)+d(i))\eps} +\sum_{i=K+1}^\infty \frac{1}{(b(i)+d(i))^2} <\infty.
		\ee 
		Here, the final sum is finite and $K\in\N$ is such that $\e^{-(b(i)+d(i))\eps}\leq (b(i)+d(i))^{-2}$ for all $i\geq K$, which is possible since $b(i)+d(i)$ tends to infinity, both due to Assumption~\eqref{ass:varphi2}. 
		
		The inequalities in~\eqref{eq:eqdeg} thus lead to a contradiction, so that any pair of vertices satisfies that their degrees are equal for only a finite amount of time. This directly implies~\eqref{eq:distdeg}, which concludes the proof.
	\end{proof}
	
	It remains to prove Proposition~\ref{prop:P}. We define the following (random) subset of $\cU_\infty$. For $K\in\N$ fixed and $u\in\cU_\infty$, set 
	\be \label{eq:Iu}
	I(u)\coloneq \begin{cases}
		K&\mbox{if }D^{(u)}=\infty,\\ 
		D^{(u)}&\mbox{if }D^{(u)}<\infty
	\end{cases}
	\ee
	Then, we define  
	\be \label{eq:Ukinf}
	\cU_K^{<\infty}=\cU_K^{<\infty}((D^{(u)})_{u\in\cU_\infty})\coloneq \{u\in \cU_\infty: u_j\leq I(u_1\cdots u_{j-1})\text{ for all }j\in[|u|]\}\cup\{\varnothing\}. 
	\ee 
	That is, given $(D^{(u)})_{u\in\cU_\infty}$, one constructs the set $\cU_K^{<\infty}$ iteratively as follows. Initialise $\cU^{(0)}\coloneq \{\varnothing\}$ and, given $\cU^{(i)}$ with $i\in\N_0$, we set
	\be 
	\cU^{(i+1)}\coloneq  \cU^{(i)}\cup \bigcup_{u\in \cU^{(i)}}\{u1,\ldots, uI(u)\}.
	\ee 
	Note that, if $u\in\cU^{(i)}$ for some $i$ and $D^{(u)}=0$, then we do not add any children $uj$ to $\cU^{(i+1)}$. This iterative process either never terminates, in which case $\cU^{<\infty}_K=\cup_{i=0}^\infty \cU^{(i)}$ is infinite, or it terminates at some (random) step $M\in\N_0$, so that $\cU^{<\infty}_K=\cU^{(M)}$ is finite. We observe that, conditionally on survival of the process $(\bp(t))_{t\geq 0}$, the set $\cU^{<\infty}_K$ contains at least one individual $u$ such that $D^{(u)}=\infty$, almost surely, when $\P{D=\infty}>0$ (which is the case when Assumption~\eqref{ass:A2} is satisfied). We also define 
	\be 
	\cU^{<\infty}\coloneq \cU_\infty^{<\infty}=\{u\in \cU_\infty\colon u_j\leq D^{(u)}\text{ for all }j\in [|u|]\}\cup \{\varnothing\}
	\ee 
	as the set of all individuals that are eventually born in $\bp$. 
	
	We now prove Proposition~\ref{prop:catch} for $m=2$. The case $m=1$ follows similarly, but requires fewer steps, whilst the difference for $m>2$ mainly lies in more notation and careful bookkeeping. As such, we decided to prove only the case $m=2$ to aid the reader. We provide a brief discussion after the proof where changes are necessary when $m=1$ or $m>2$.
	
	\begin{proof}[Proof of Proposition~\ref{prop:P} for $m=2$]
		We first note that $\cP_2\subseteq \cU^{<\infty}$, since any individual that catches up to all its predecessors clearly needs to be born.  We first construct a superset of $\cU^{<\infty}$, which then informs a superset for $\cP_2$.  Recall $K$ from Proposition~\ref{prop:catch}. We define 
		\be \label{eq:TTdef}
		T=T(K)\coloneq \inf\{t\geq 0 : |\{u\in \cU_\infty: D^{(u)}=\infty\text{ and } \cB(uK)\leq t\}|\geq2\}
		\ee  
		to be the stopping time at which the branching process contains at least two individuals that have produced (at least) $K$ many children and whose offspring $D$ is infinite. Let $V_1$ and $V_2$ denote these two individuals, where we assume without loss of generality that $\cB(V_1K)\leq \cB(V_2K)=T$. Note that, conditionally on survival, such individuals exist and that $T$ is finite $\mathbb P_\cS$-almost surely, since $\P{D=\infty}>0$ by Assumption~\eqref{ass:A2} (see~\cite[Lemma 5.3 and Corollary 6.4]{HeyLod25}). We then define 
		\be 
		\cA_-\coloneq  \cU_K^{<\infty}\cap \bp(T)\quad\text{and}\quad  
		\cA_+\coloneq \bigg(\bigcup_{u\in \cA_-}\{u1,\ldots, uI(u)\}\bigg)\setminus\cA_-, 
		\ee 
		where we recall $I(u)$ from~\eqref{eq:Iu} and $\cU_K^{<\infty}$ from~\eqref{eq:Ukinf}. We think of $\cA_+$ as the `boundary' of $\cA_-$ restricted to $\cU_K^{<\infty}$. That is, $\cA_+$ consists of all individuals in $\cU_K^{<\infty}$ that have not been born by time $T$ but whose parent has been born by time $T$.  We then have
		\be \label{eq:Ufinincl}
		\cU^{<\infty} \subseteq  \cA_-\cup\bigcup_{\substack{u\in \cA_-\\ D^{(u)}=\infty}}\{uv\colon v\in \cU_\infty, v_1>K\} \cup\bigcup_{u\in \cA_+}\{uv\colon v\in \cU_\infty\}.
		\ee 
		Indeed, let $w\in \cU^{<\infty}$. We show that $wj$ is in one of the sets on the right-hand side for each $j\in [D^{(w)}]$ when $D^{(w)}<\infty$ and for each $j\in\N$ when $D^{(w)}=\infty$. Since $\varnothing\in \cA_-$, this recursively implies the inclusion.
		
		First, suppose that $w\in\cA_-$. If $D^{(w)}<\infty$, then for any $j\in[D^{(w)}]$ the individual $wj\in \cA_-$ if $wj$ is born by time $T$, or it is in $\cA_+$ if it is not born by time $T$ (and so it is included in the set $\{wjv\colon v\in \cU_\infty\}$). If $D^{(w)}=\infty$, then for any $j\in[K]$ we can repeat the same argument as in the previous case, and for any $j>K$ the individual $wj$ is in $\{wv\colon v\in \cU_\infty, v_1>K\}$.
		
		Second, suppose that $w\in \{uv\colon v\in \cU_\infty, v_1>K\}$ for some $u\in \cA_-$ with $D^{(u)}=\infty$. It then directly follows that $wj\in \{uv\colon v\in \cU_\infty, v_1>K\}$ for any $j\in\N$ (irrespective of the value of $D^{(w)}$). 
		
		Third, suppose that $w\in \{uv\colon v\in \cU_\infty\}$ for some $u\in\cA_+$. Again, it directly follows that $wj\in \{uv\colon v\in\cU_\infty\}$ for any $j\in\N$ (irrespective of the value of $D^{(w)}$).
		
		Using this superset, we then construct a superset for $\cP_2$. To this end, we observe that $u\in \cP_2$  implies that $D^{(u)}=\infty$ and that, for any $w_1\neq w_2\in \text{Pre}(u)$, either $u$ catches up to $w_1$ or $u$ catches up to $w_2$, as $u$ catches up to all but at most one predecessor. Abbreviating `catches up to' by c.u.t.\ and using~\eqref{eq:Ufinincl},  we thus have the inclusion
		\be \ba \label{eq:P2incl}
		\cP_2\subseteq \cA_- {}&\cup \bigcup_{\substack{u\in \cA_-\setminus\{V_2\}\\ D^{(u)}=\infty}}\{uv\colon v\in \cU_\infty, v_1>K, D^{(uv)}=\infty, uv \text{ c.u.t. }u\text{ or }V_2\}\\
		&\cup \{V_2v\colon v\in \cU_\infty, v_1>K, D^{(V_2v)}=\infty, V_2v \text{ c.u.t. }V_1\text{ or }V_2\}\\
		&\cup \bigcup_{u\in \cA_+}\{uv\colon v\in \cU_\infty, D^{(uv)}=\infty, uv\text{ c.u.t. }V_1\text{ or }V_2\}.
		\ea\ee  
		Again, this inclusion is clear. By~\eqref{eq:Ufinincl}, we take all individuals in $u\in\cA_-\setminus\{V_2\}$ such that $D^{(u)}=\infty$  and require an individual in the set $\{uv\colon v\in \cU_\infty, v_1>K\}$ to catch up to either $u$ or to $V_2$. Similarly, we require any individual in the set $\{V_2v\colon v\in \cU_\infty, v_1>K\}$  to catch up to either  $V_1$ or $V_2$. Finally, for any individual in a set $\{uv\colon v\in \cU_\infty\}$  for some $u\in\ \cA_+$, we require they catch up to either $V_1$ or $V_2$. 
		
		It remains to show that each of the sets on the right-hand side of the inclusion are finite $\mathbb P_\cS$-almost surely. First, since $T$ is finite $\mathbb P_\cS$-almost surely it follows from Lemma~\ref{lemma:finbp} that $\cA_-\subseteq \bp(T)$ is finite $\mathbb P_\cS$-almost surely. As $I(u)$ is finite almost surely for each $u\in\cU_\infty$, it also follows that $\cA_+$ is finite $\mathbb P_\cS$-almost surely. We also have the inclusions 
		\be \ba \label{eq:A-sets}
		\{uv{}\colon{}& v\in \cU_\infty, v_1>K, D^{(uv)}=\infty, uv \text{ c.u.t. }u\text{ or }V_2\}\\
		\subseteq{}& \{uv\colon v\in \cU_\infty, v_1>K, D^{(uv)}=\infty, uv \text{ c.u.t. }u\}\\
		&\cup \{uv\colon v\in \cU_\infty, v_1>K, D^{(uv)}=\infty, uv \text{ c.u.t. }V_2\},
		\ea\ee 
		and 
		\be\ba\label{eq:V2sets}
		\{V_2v\colon {}&v\in \cU_\infty, v_1>K, D^{(V_2v)}=\infty, V_2v \text{ c.u.t. }V_1\text{ or }V_2\}\\
		\subseteq{}&\bigcup_{j=1}^2 \{V_2v\colon v\in \cU_\infty, v_1>K, D^{(V_2v)}=\infty, V_2v \text{ c.u.t. }V_j\},
		\ea\ee 
		and
		\be \ba \label{eq:A+sets}
		\{uv{}\colon{}& v\in \cU_\infty, D^{(uv)}=\infty, uv \text{ c.u.t. }V_1\text{ or }V_2\}\\
		\subseteq{}& \bigcup_{j=1}^2\{uv\colon v\in \cU_\infty,  D^{(uv)}=\infty, uv \text{ c.u.t. }V_j\}.
		\ea\ee
		We thus obtain the desired result if we   show that, for any $u\in \cA_-$ such that $D^{(u)}=\infty$, the sets on the right-hand side of~\eqref{eq:A-sets} are finite $\mathbb P_\cS$-almost surely, the sets on the right-hand side of~\eqref{eq:V2sets} are finite $\mathbb P_\cS$-almost surely, and, for any $u\in\cA_+$, the sets on the right-hand side of~\eqref{eq:A+sets} are finite $\mathbb P_\cS$-almost surely. We show that this is indeed the case, distinguishing per equation. The proof for each case is similar to the proof of Proposition~\ref{prop:catch}, but with slight, case specific, differences. 
		
		\textbf{Sets in~\eqref{eq:A-sets}. } It follows directly from Proposition~\ref{prop:catch} that the set $\{uv\colon v\in \cU_\infty, v_1>K, D^{(uv)}=\infty, uv \text{ c.u.t. }u\}$ is finite $\mathbb P_\cS$-almost surely for any individual $u\in \cA_-$ such that $D^{(u)}=\infty$ (by replacing $\varnothing$ by $u$ and `$v$ catches up to $\varnothing$' by $uv$ catches up to $u$'). Note that this also includes the case $u=V_2$, so that the set $\{V_2v\colon v\in \cU_\infty, v_1>K, D^{(uv)}=\infty, V_2v \text{ c.u.t. }V_2\}$ in~\eqref{eq:V2sets} is finite $\mathbb P_\cS$-almost surely as well. We continue with the sets
		\be 
		\{uv\colon v\in \cU_\infty, v_1>K, D^{(uv)}=\infty, uv \text{ c.u.t. }V_2\}\qquad\text{for }u\in \cA_-\setminus\{V_2\}\text{ such that }D^{(u)}=\infty. 
		\ee 
		Let $\cF_T$ be the $\sigma$-algebra generated by $(\bp(t); t\leq T)$. We note that $T$, $\cA_-$, $\cA_+$, $V_1$, and $V_2$ are all measurable with respect to $\cF_T$. Take $u\in \cA_-\setminus\{V_2\}$ with $D^{(u)}=\infty$. We distinguish between two cases. \\
		\textbf{Case (i)} $u\neq V_1$. Similar to the proof of Proposition~\ref{prop:catch}, we can write for any $uv$ with $v=v_1\cdots v_m\in \cU_\infty$ and with $m\in\N_0$, $v_1>K$, and $v_2,\ldots, v_m\in\N$,  
		\be \ba
		\mathbb P_\cS\big(uv{}&\text{ catches up to }V_2\text{ and }D^{(uv)}=\infty\,\big|\, \cF_T\big)\\
		=\mathbb P_\cS\Bigg({}&D^{(uv)}=\infty,D^{(uv_1\cdots v_{\ell-1})}\geq v_\ell\text{ for all }\ell\in[m],\\ 
		&  \exists j\in\N_0\colon \cB(u)+\sum_{\ell=1}^m\sum_{i=0}^{v_\ell-1}E^{(uv_1\cdots v_{\ell-1}) }_i+\sum_{i=0}^{K+j}E^{(uv)}_i\leq \cB(V_2K)+\!\!\!\sum_{i=K}^{K+j}E^{(V_2)}_i\,\Bigg|\, \cF_T\Bigg).
		\ea\ee 
		Since $V_2$ produces its $K^{\mathrm{th}}$ child at time $T$ by the definition of $T$ and $V_2$, it follows that $\cB(V_2K)=T$. Furthermore, since we suppose that $u\neq V_1$, we know that $u$ has produced at most $K$ children (in fact, fewer than $K$ children) by time $T$, so that $\cB(uK)\leq T$. We can thus omit $\cB(u)+\sum_{i=0}^{K-1} E^{(u)}_i=\cB(uK)$ and $\cB(V_2K)$ from the event to obtain the upper bound
		\be \ba
		\mathbb P_\cS\Bigg({}&D^{(uv)}=\infty,D^{(uv_1\cdots v_{\ell-1})}\geq v_\ell\text{ for all }\ell\in[m],\\ 
		&   \exists j\in\N_0\colon\sum_{i=K}^{v_1-1}E^{(u)}_i+\sum_{\ell=2}^m\sum_{i=0}^{v_\ell-1}E^{(uv_1\cdots v_{\ell-1})}_i+\sum_{i=0}^{K+j}E^{(uv)}_i\leq \sum_{i=K}^{K+j}E^{(V_2)}_i\,\Bigg|\, \cF_T\Bigg).
		\ea\ee    
		We now observe that all random variables in the conditional probability, with the exception of $D^{(u)}$, are independent of $\cF_T$. Since we know, however, that $D^{(u)}=\infty$ and thus $D^{(u)}\geq v_1$ is satisfied, the probability almost surely equals 
		\be \ba 
		\mathbb P_\cS\Bigg({}&D^{(uv)}=\infty,D^{(uv_1\cdots v_{\ell-1})}\geq v_\ell\text{ for all }\ell\in\{2,\ldots, m\},\\ 
		&  \exists j\in\N_0\colon \sum_{i=K}^{v_1-1}E^{(u)}_i+\sum_{\ell=2}^m\sum_{i=0}^{v_\ell-1}E^{(uv_1\cdots v_{\ell-1})}_i+\sum_{i=0}^{K+j}E^{(uv)}_i\leq \sum_{i=K}^{K+j}E^{(V_2)}_i\Bigg).
		\ea \ee 
		By bounding $\Ps{\cE}\leq \P{\cE}/\P{\cS}$ for any event $\cE$, using the independence of the offspring and inter-birth random variables, and applying Lemma~\ref{lemma:catch} (where we note that the condition in~\eqref{eq:infprodexp} is satisfied for any $\lambda>0$ by choosing $K$ large enough, see Remark~\ref{rem:sumconvass}), we can bound the probability from above by 
		\be\ba  
		\frac{\P{D=\infty}}{\P{\cS}}{}&\prod_{i=K}^\infty \E{\e^{\lambda(E_i'-E_i)}} \E{\exp\bigg(-\lambda \sum_{i=0}^{K-1} E_i\bigg)}\E{\exp\bigg(-\lambda \sum_{i=K}^{v_1-1}E_i\bigg)}\\
		&\times \prod_{\ell=2}^m \P{D\geq v_\ell}\E{\exp\bigg(-\lambda \sum_{i=0}^{v_\ell-1}E_i\bigg)}. 
		\ea \ee 
		Since the inter-birth random variables in the final two terms on the first line are independent, we can combine them into a single expected value. By also bounding $\P{D=\infty}\leq \P{D\geq v_1}$, we thus obtain the upper bound
		\be 
		\frac{1}{\P{\cS}}\prod_{i=K}^\infty \E{\e^{\lambda(E_i'-E_i)}}\prod_{\ell=1}^m \P{D\geq v_\ell}\E{\exp\bigg(-\lambda \sum_{i=0}^{v_\ell-1}E_i\bigg)}.
		\ee 
		We set $c=\P{\cS}^{-1}$. By summing over all $uv$ with $v\in \cU_\infty$ such that $v_1>K$, we thus obtain 
		\be \ba 
		\mathbb E_\cS\Big[{}&|\{uv\colon v\in \cU_\infty, v_1>K,uv\text{ catches up to }V_2\text{ and }D^{(uv)}=\infty\}|\,\Big|\, \cF_T\Big]\\
		\leq{}& \sum_{m=1}^\infty \sum_{v_1=K+1}^\infty \sum_{v_2=1}^\infty \cdots \sum_{v_m=1}^\infty c\prod_{i=K}^\infty \E{\e^{\lambda(E_i'-E_i)}}\prod_{\ell=1}^m \P{D\geq v_\ell}\E{\exp\bigg(-\lambda \sum_{i=1}^{v_\ell}E_i\bigg)}\\
		&\leq  c\prod_{i=K}^\infty \E{\e^{\lambda(E_i'-E_i)}}\sum_{m=1}^\infty \Bigg(\sum_{v=1}^\infty \P{D\geq v}\E{\exp\bigg(-\lambda \sum_{i=0}^{v-1}E_i\bigg)}\Bigg)^m,
		\ea \ee 
		where in the last step we interchanged summation and products and extended the range of the sum over $v_1$ to all $v_1\in\N$. We observe that the inner sum on the last line equals $\widehat\mu(\lambda)$, so that by Assumption~\eqref{ass:mufin} we can take $\lambda>0$ large enough such that $\widehat\mu(\lambda)<1$, as follows from~\eqref{eq:muless1}. The upper bound is thus finite, from which it follows that 
		\be 
		\Ps{\{uv\colon v\in \cU_\infty, v_1>K,uv\text{ catches up to }V_2\text{ and }D^{(uv)}=\infty\}|<1\,\Big|\, \cF_T}=1,\quad \mathbb P_\cS\text{-a.s.}
		\ee 
		\textbf{Case (ii)} $u=V_1$. In this case, let 
		\be 
		\cT_T(V_1)\coloneq \{V_1v\colon v\in \cU_\infty, v_1>K, \cB(V_1v)\leq T\}
		\ee 
		be the sub-tree rooted at $V_1$, restricted to the children $V_1(K+1),V_1(K+2),\ldots$ and their descendants that are born by time $T$. It is clear that  $\cT_T(V_1)\subset \bp(T)$, so that it is finite $\mathbb P_\cS$-almost surely, and that it is $\cF_T$ measurable. For $w\in \cT_T(V_1) $, we define
		\be 
		K(w)\coloneq \max\{0\leq k\leq D^{(w)}: wk\in \cT_T(V_1) \}
		\ee 
		to be the `last' child of $w$ that is included in $\cT_T(V_1) $. We then have the inclusion
		\be \ba 
		\{V_1v:{}& v\in \cU_\infty, v_1>K,D^{(V_1v)}=\infty,  V_1v\text{ c.u.t.\ }V_2\}\\ 
		&\subseteq  \cT_T(V_1)\cup \bigcup_{w\in \cT_T(V_1) }\{wv: v\in \cU_\infty, v_1>K(w), D^{(wv)}=\infty, wv \text{ c.u.t.\ }V_2\}.
		\ea \ee 
		As a result, we have
		\be \ba \label{eq:wexp}
		\mathbb E_\cS{}&\Big[|\{V_1v: v\in \cU_\infty, v_1>K, \cB(V_1v)>T,D^{(V_1v)}=\infty,  V_1v\text{ c.u.t.\ }V_2\}|\,\Big|\, \cF_T\Big]\\ 
		\leq{}&|\cT_T(V_1)|+\!\!\!\! \sum_{w\in \cT_T(V_1) }\!\!\!\! \!\!\Es{|\{wv: v\in \cU_\infty, v_1>K(w), D^{(wv)}=\infty, wv \text{ c.u.t.\ }V_2\}|\,\Big|\, \cF_T}.
		\ea \ee 
		As already concluded, the first term on the right-hand side is finite $\mathbb P_\cS$-almost surely. Each expected value in the sum on the right-hand side can be written as
		\be\label{eq:probsum}
		\sum_{m=1}^\infty \sum_{v_1=K(w)+1}^{D^{(w)}}\sum_{v_2=1}^\infty \cdots \sum_{v_m=1}^\infty \Ps{D^{(wv)}=\infty, wv\text{ c.u.t.\ }V_2\,\Big|\, \cF_T}, 
		\ee
		and each such probability equals 
		\be\ba  
		\mathbb P_\cS\bigg({}&\exists j\in\N_0\colon \cB(wK(w))+\!\!\!\!\!\sum_{i=K(w)}^{v_1-1}\!\!\!\!E^{(w)}_i+\sum_{\ell=2}^m\sum_{i=0}^{v_\ell-1}E^{(wv_1\cdots v_{\ell-1})}_i+\!\sum_{i=0}^{K+j} \!E^{(wv)}_i\!\leq \cB(V_2K)+\! \sum_{i=K}^{K+j}\!  E^{(V_2)}_i,\\ 
		&D^{(wv)}=\infty, D^{(wv_1\cdots v_{\ell-1})}\geq v_\ell\text{ for all }\ell\in[m]  \,\bigg|\, \cF_T\bigg).
		\ea\ee 
		We observe that $\cB(w(K(w)+1))=\cB(wK(w))+E^{(w)}_{K(w)}>T=\cB(V_2K)$ by the definition of $K(w)$, so that we obtain the upper bound 
		\be \ba 
		\mathbb P_\cS\bigg({}&\exists j\in\N_0: \!\!\!\!\sum_{i=K(w)+1}^{v_1-1}\!\!\!\!E^{(w)}_i+\sum_{\ell=2}^m\sum_{i=0}^{v_\ell-1}E^{(wv_1\cdots v_{\ell-1})}_i+\!\sum_{i=0}^{K+j} \!E^{(wv)}_i\!\leq  \sum_{i=K}^{K+j}\!  E^{(V_2)}_i,\\ 
		&D^{(wv)}=\infty, D^{(wv_1\cdots v_{\ell-1})}\geq v_\ell\text{ for all }\ell\in[m]  \,\bigg|\, \cF_T\bigg).\textbf{}
		\ea\ee 
		Only the random variables $D^{(w)}$, $K(w)$, and $V_2$ depend on $\cF_T$. Also, the random variable $K(w)$ is independent of the inter-birth times $E^{(w)}_i$ for $i\geq K(w)+1$. As a result, we can apply Lemma~\ref{lemma:catch} (where we note that the condition in~\eqref{eq:infprodexp} is satisfied for any $\lambda>0$ by choosing $K$ large enough, see Remark~\ref{rem:sumconvass}) to obtain the upper bound 
		\be\ba 
		\frac{\P{D=\infty}}{\P{\cS}}{}&\ind_{\{D^{(w)}\geq v_1\}} \prod_{i=K}^\infty \E{\e^{\lambda(E_i'-E_i)}}\E{\exp\bigg(-\lambda \sum_{i=0}^{K-1}E_i\bigg)} \prod_{i=K(w)+1}^{v_1-1}\E{\e^{-\lambda E_i}}\\ 
		&\times \prod_{\ell=2}^m \P{D\geq v_\ell}\E{\exp\bigg(-\lambda \sum_{i=0}^{v_\ell-1}E_i\bigg)}.
		\ea\ee 
		Substituting this upper bound in~\eqref{eq:probsum}, we thus arrive at the upper bound 
		\be\ba 
		\mathbb E_\cS\Big[{}&|\{wv: v\in \cU_\infty, v_1>K(w), D^{(wv)}=\infty, wv \text{ c.u.t.\ }V_2\}|\,\Big|\, \cF_T\Big]\\
		\leq{}& 	\sum_{m=1}^\infty \sum_{v_1=K(w)+1}^{D^{(w)}}\sum_{v_2=1}^\infty \cdots \sum_{v_m=1}^\infty \frac{\P{D=\infty}}{\P{\cS}}  \prod_{i=K}^\infty  \E{\e^{\lambda(E_i'-E_i)}}\E{\exp\bigg(-\lambda \sum_{i=0}^{K-1}E_i\bigg)} \\ 
		&\times \prod_{i=K(w)+1}^{v_1-1}\E{\e^{-\lambda E_i}}\prod_{\ell=2}^m \P{D\geq v_\ell}\E{\exp\bigg(-\lambda \sum_{i=0}^{v_\ell-1}E_i\bigg)}.
		\ea \ee 
		Switching summation and products yields 
		\be\ba 
		\frac{\P{D=\infty}}{\P{\cS}} \prod_{i=K}^\infty{}& \E{\e^{\lambda(E_i'-E_i)}}\E{\exp\bigg(-\lambda \sum_{i=0}^{K-1}E_i\bigg)}\\
		\times  \sum_{v_1=K(w)+1}^{D^{(w)}}{}&\prod_{i=K(w)+1}^{v_1-1}\!\!\!\E{\e^{-\lambda E_i}} \sum_{m=1}^\infty\Bigg(\sum_{v=1}^\infty\P{D\geq v}\E{\exp\bigg(-\lambda \sum_{i=0}^{v-1}E_i\bigg)}\Bigg)^{m-1}.
		\ea\ee 
		As in Case (i), we can write the inner sum as $\widehat\mu(\lambda)$ and by Assumption~\eqref{ass:mufin} we can choose $\lambda>0$ such that $\widehat\mu(\lambda)<1$ (see the proof of Lemma~\ref{lemma:finbp}). We then observe that 
		\be \ba 
		\sum_{v_1=K(w)+1}^{D^{(w)}}\prod_{i=K(w)+1}^{v_1-1}\!\!\!\!\!\!\E{\e^{-\lambda E_i}}&\leq  
		\sum_{v_1=K(w)+1}^\infty  \prod_{i=K(w)+1}^{v_1-1}\!\!\!\!\!\E{\e^{-\lambda E_i}}\\
		&\leq \P{D=\infty}^{-1}\!\!\!\prod_{i=0}^{K(w)}\!\E{\e^{-\lambda E_i}}^{-1} \sum_{v=1}^\infty \P{D\geq v}\prod_{i=1}^{v-1} \E{\e^{-\lambda E_i}}\\
		&=\P{D=\infty}^{-1}\widehat\mu(\lambda)\prod_{i=1}^{K(w)}\E{\e^{-\lambda E_i}}^{-1} .
		\ea\ee 
		This upper bound is $\mathbb P_\cS$-almost surely finite for any $w\in \cT_T$, since $K(w)\leq K$ by the definition of $T$ and $V_1,V_2$ and since we can take $\lambda>0$ such that $\widehat\mu(\lambda)<\infty$ by Assumption~\eqref{ass:mufin}. As a result, since $|\cT_T(V_1)|$ is finite $\mathbb P_\cS$-almost surely,  the upper bound in~\eqref{eq:wexp} is finite $\mathbb P_\cS$-almost surely. We thus conclude that sets on the left-hand side of~\eqref{eq:A-sets} are finite $\mathbb P_\cS$-almost surely conditionally on $\cF_T$, for any $u\in \cA_-\setminus\{V_2\}$ . 
		
		\textbf{Sets in~\eqref{eq:V2sets}. } We already concluded that the set $\{V_2v\colon v\in \cU_\infty, v_1>K, D^{(V_2v)}=\infty, V_2v \text{ c.u.t. }V_2\}$ is finite $\mathbb P_\cS$-almost surely by Proposition~\ref{prop:catch}, so that we are left to bound the size of the set $\{V_2v\colon v\in \cU_\infty, v_1>K, D^{(V_2v)}=\infty, V_2v \text{ c.u.t. }V_1\}$.  Define 
		\be \label{eq:Kj}
		K_j\coloneq \sup\{k\in\N\colon \cB(V_jk)\leq T\}\qquad \text{for }j\in\{1,2\}. 
		\ee   
		Here, we note that $K_2=K$ by definition of $T$ and $V_2$, and $K_1\geq K$. We then write
		\be\ba \label{eq:V2setexp}
		\mathbb E_\cS{}&\Big[|\{V_2v\colon v\in \cU_\infty, v_1>K, D^{(V_2v)}=\infty, V_2v\text{ c.u.t.\ }V_1\}|\,\Big|\, \cF_T\Big]\\
		&=\sum_{m=1}^\infty \sum_{v_1=K+1}^\infty \sum_{v_2=1}^\infty \cdots \sum_{v_m=1}^\infty \Ps{D^{(V_2v)}=\infty, V_2v\text{ c.u.t.\ }V_1\,|\, \cF_T}.
		\ea \ee 
		Since $K_2=K$ by definition of $T$ and $V_1,V_2$, we can write each probability as 
		\be\ba 
		\mathbb P_\cS\bigg({}&\exists j\in\N_0\colon \cB(V_2K)+\sum_{i=K}^{v_1-1}E^{(V_2)}_i+\sum_{\ell=2}^m \sum_{i=0}^{v_\ell-1}E^{(V_2v_1\cdots v_{\ell-1})}_i+\!\!\sum_{i=0}^{K_1+j}\!E^{(V_2v)}_i\leq \cB(V_1K_1)+\!\!\sum_{i=K_1}^{K_1+j}\!E^{(V_1)}_i,\\
		&
		D^{(V_2v)}=\infty, D^{(V_2v_1\cdots v_{\ell-1})}\geq v_\ell\text{ for all }\ell\in\{2,\ldots, m\}\,\bigg|\, \cF_T\bigg).
		\ea \ee 
		Since $\cB(V_1K_1)\leq \cB(V_2K)=T\leq \cB(V_1(K_1+1))$ by the definition of $T$ and $K_1,K_2$, and by omitting the event that $D^{(V_2v)}=\infty$, we can bound this probability from above by 
		\be\ba 
		\mathbb P_\cS\bigg({}&\exists j\in\N_0\colon \!\sum_{i=K}^{v_1-1}\!E^{(V_2)}_i+\sum_{\ell=2}^m \sum_{i=0}^{v_\ell-1}E^{(V_2v_1\cdots v_{\ell-1})}_i+\!\!\!\sum_{i=1}^{K_1+j}\!E^{(V_2v)}_i\leq \cB(V_1K_1)-T+\!\sum_{i=K_1}^{K_1+j}\!E^{(V_1)}_i,\\
		& D^{(V_2v_1\cdots v_{\ell-1})}\geq v_\ell\text{ for all }\ell\in\{2,\ldots, m\}\,\bigg|\, \cF_T\bigg). 
		\ea\ee
		Only the random variables $V_1,V_2,K_1,T$, and $E^{(V_1)}_{K_1}$ depend on $\cF_T$. All other random variables are independent of $\cF_T$. By the memoryless property of the exponential distribution and the strong Markov property, we have 
		\be 
		\cB(V_1K_1)-T+E^{(V_1)}_{K_1}\sim E_{K_1}', 
		\ee 
		where $ E_{K_1}'$ is an independent copy of $E^{(V_1)}_{K_1}$, which is also independent of $\cF_T$. The probability thus equals 
		\be\ba 
		\mathbb P_\cS\bigg({}&\exists j\in\N_0\colon \sum_{i=K}^{v_1-1}\!E^{(V_2)}_i+\sum_{\ell=2}^m \sum_{i=0}^{v_\ell-1}E^{(V_2v_1\cdots v_{\ell-1})}_i+\!\!\!\sum_{i=1}^{K_1+j}\!E^{(V_2v)}_i\leq E_{K_1}'+\!\sum_{i=K_1}^{K_1+j}\!E^{(V_1)}_i,\\
		&
		D^{(V_2v_1\cdots v_{\ell-1})}\geq v_\ell\text{ for all }\ell\in\{2,\ldots, m\}\,\bigg|\, \cF_T\bigg). 
		\ea\ee
		We now apply Lemma~\ref{lemma:catch} (where we note that the condition in~\eqref{eq:infprodexp} is satisfied for any $\lambda>0$ by choosing $K$ large enough, see Remark~\ref{rem:sumconvass}), to obtain the upper bound 
		\be\ba \label{eq:V2setineq}
		\frac{1}{\P{\cS}}{}&\prod_{i=K_1}^\infty \E{\e^{\lambda(E_i'-E_i)}}\bigg(\prod_{i=0}^{K_1}\E{\e^{-\lambda E_i}}\bigg)\P{D\geq v_1}\E{\exp\bigg(-\lambda \sum_{i=K}^{v_1-1}E_i\bigg)}\\
		&\times \prod_{\ell=2}^m \P{D\geq v_\ell}\E{\exp\bigg(-\lambda \sum_{i=0}^{v_\ell-1}E_i\bigg)}.
		\ea\ee  
		Since $\E{\exp(\lambda(E_i'-E_i))}\geq 1$ by Jensen's inequality and  $\E{\exp(-\lambda E_i)}\leq 1$ for any $i\in\N$ and as $K_1\geq K$ by definition, we obtain an upper bound when replacing $K_1$ with $K$ in the first two products. Then, by the independence of the inter-birth times, we have
		\be 
		\bigg(\prod_{i=0}^{K}\E{\e^{-\lambda E_i}}\bigg)\P{D\geq v_1}\E{\exp\bigg(-\lambda\!\!\! \sum_{i=K}^{v_1-1}\!\!\!E_i\bigg)}=\P{D\geq v_1}\E{\exp\bigg(-\lambda \sum_{i=0}^{v_1-1}E_i\bigg)}.
		\ee 
		The right-hand side can thus be incorporated in the product on the second line of~\eqref{eq:V2setineq}, so that this expression (with $K_1$ bounded by $K$) simplifies to
		\be
		\frac{1}{\P{\cS}}{}\prod_{i=K}^\infty \E{\e^{\lambda(E_i'-E_i)}}\prod_{\ell=1}^m \P{D\geq v_\ell}\E{\exp\bigg(-\lambda \sum_{i=0}^{v_\ell-1}E_i\bigg)}.
		\ee 
		Using this upper bound in~\eqref{eq:V2setexp} and writing $c=1/\P{\cS}$, we thus obtain 
		\be\ba 
		\mathbb E_\cS\Big[{}&|\{V_2v\colon v\in \cU_\infty, v_1>K, D^{(V_2v)}=\infty, V_2v\text{ c.u.t.\ }V_1\}|\,\Big|\, \cF_T\Big]\\
		\leq{}& \sum_{m=1}^\infty \sum_{v_1=K+1}^\infty \sum_{v_2=1}^\infty \cdots \sum_{v_m=1}^\infty c\prod_{i=K}^\infty \E{\e^{\lambda(E_i'-E_i)}}\prod_{\ell=1}^m \P{D\geq v_\ell}\E{\exp\bigg(-\lambda \sum_{i=0}^{v_\ell-1}E_i\bigg)}.
		\ea\ee 
		Switching summation and products and extending the range of the variable $v_1$ to $v_1\in\N$ yields the upper bound
		\be
		c \prod_{i=K}^\infty  \E{\e^{\lambda(E_i'-E_i)}}\sum_{m=1}^\infty  \Bigg(\sum_{v=1}^\infty  \P{D\geq v}\E{\exp\bigg(-\lambda \sum_{i=0}^{v-1}E_i\bigg)}\Bigg)^m. 
		\ee 
		By writing the inner sum as $\widehat\mu(\lambda)$ and using that, by Assumption~\eqref{ass:mufin}, we can take $\lambda>0$ large enough so that $\widehat\mu(\lambda)<1$, as follows from~\eqref{eq:muless1}, we conclude that the upper bound is finite. Using this in~\eqref{eq:V2sets} combined with the earlier conclusion that the set $\{V_2v\colon v\in\cU_\infty, v_1>K, D^{(V_2v)}=\infty, V_2v\text{ c.u.t.\ }V_1\}$ is finite $\mathbb P_\cS$-almost surely, we obtain
		\be
		\Ps{|\{V_2v\colon v\in \cU_\infty, v_1>K, D^{(V_2v)}=\infty, V_2v\text{ c.u.t.\ }V_1\text{ or }V_2\,\Big|\, \cF_T}=1 \qquad \mathbb P_\cS\text{-almost surely}. 
		\ee 
		\textbf{Sets in~\eqref{eq:A+sets}. } This proof follows a similar approach as for the sets in~\eqref{eq:V2sets}. By conditioning on $\cF_T$, we have 
		\be\ba \label{eq:expecsum}
		\mathbb E_\cS{}&\Big[|\{uv\colon v\in \cU_\infty,   D^{(uv)}=\infty, uv \text{ c.u.t. }V_j\}|\,\Big|\, \cF_T\Big]\\
		&=\sum_{m=1}^\infty \sum_{v_1=1}^\infty \cdots \sum_{v_m=1}^\infty \Ps{D^{(uv)}=\infty, uv \text{ c.u.t. }V_j\,\Big|\, \cF_T}.
		\ea\ee 
		Recalling $K_j$ from~\eqref{eq:Kj}, observing that that $K_1\geq K$ and $K_2=K$ by definition of $T$ and $V_1,V_2$, we can write each of the probabilities in~\eqref{eq:expecsum} as 
		\be\ba 
		\mathbb P_\cS\bigg({}&\exists n\in\N_0\colon \cB(u)+\sum_{\ell=1}^m \sum_{i=0}^{v_\ell-1}E^{(uv_1\cdots v_{\ell-1})}_i+\!\!\!\sum_{i=0}^{K_j+n}\!E^{(uv)}_i\leq \cB(V_jK_j)+\!\!\!\sum_{i=K_j}^{K_j+n}\!E^{(V_j)}_i, \\
		&D^{(uv)}=\infty, D^{(uv_1\cdots v_{\ell-1})}\geq v_\ell\text{ for all }\ell \in [m]\,\bigg|\, \cF_T\bigg).
		\ea\ee 
		Since $u\in \cA_+$, we know that $\cB(u)>T\geq \cB(V_jK_j)$, so that we obtain the upper bound 
		\be \ba 
		\mathbb P_\cS\bigg({}&\exists n\in\N_0\colon \sum_{\ell=1}^m \sum_{i=0}^{v_\ell-1}E^{(uv_1\cdots v_{\ell-1})}_i+\!\!\!\sum_{i=0}^{K_j+n}\!E^{(uv)}_i\leq \cB(V_jK_j)-T+\!\!\!\sum_{i=K_j}^{K_j+n}\!E^{(V_j)}_i, \\
		&D^{(uv)}=\infty, D^{(uv_1\cdots v_{\ell-1})}\geq v_\ell\text{ for all }\ell \in [m]\,\bigg|\, \cF_T\bigg).
		\ea \ee 
		Only the random variables $K_j$, $T$,  $V_j$, and $E^{(V_j)}_{K_j}$  depend on $\cF_T$ (the latter only for $j=1$). Since the inter-birth times are exponentially distributed and by the strong Markov property, it follows that
		\be 
		\cB(V_jK_j)+E^{(V_j)}_{K_j}-T \sim E_{K_j}' \qquad \text{for both }j\in\{1,2\}, 
		\ee 
		where $E_{K_j}'$ is an i.i.d.\ copy of $E^{(V_j)}_{K_j}$ that is independent of $\cF_T$. In fact, for $j=2$, since $K_2=K$  we have $\cB(V_2K_2)+E^{(V_2)}_{K_2}-T=E^{(V_2)}_{K_2}$, which is independent of $\cF_T$. For both $j=1$ and $j=2$, we thus obtain 
		\be \ba 
		\mathbb P_\cS\bigg({}&\exists n\in\N_0\colon \sum_{\ell=1}^m \sum_{i=0}^{v_\ell-1}E^{(uv_1\cdots v_{\ell-1})}_i+\!\!\!\sum_{i=0}^{K_j+n}\!E^{(uv)}_i\leq E_{K_j}'+\!\!\!\sum_{i=K_j+1}^{K_j+n}\!E^{(V_j)}_i, \\
		&D^{(uv)}=\infty, D^{(uv_1\cdots v_{\ell-1})}\geq v_\ell\text{ for all }\ell \in [m]\,\bigg|\, \cF_T\bigg).
		\ea \ee 
		We can then apply Lemma~\ref{lemma:catch} (where we note that the condition in~\eqref{eq:infprodexp} is satisfied for any $\lambda>0$ by choosing $K$ large enough, see Remark~\ref{rem:sumconvass}), to obtain the upper bound 
		\be\ba 
		\frac{1}{\P{\cS}}\prod_{i=K_j}^\infty\!\! \E{\e^{\lambda (E_i'-E_i)}}\prod_{i=0}^{K_j-1}\E{\e^{-\lambda E_i}}\prod_{\ell=1}^m\Bigg(  \P{D\geq v_\ell}\E{\exp\bigg(-\lambda \sum_{i=0}^{v_\ell-1}E_i\bigg)}\Bigg).
		\ea\ee 
		Since $\E{\exp(\lambda(E_i'-E_i))}\geq 1$ by Jensen's inequality and that $\E{\exp(-\lambda E_i)}\leq 1$ for any $i\in\N$ and as $K_j\geq K$ for both $j\in\{1,2\}$, we can obtain an upper bound by replacing $K_j$ with $K$. Substituting this upper bound in~\eqref{eq:expecsum} and writing $c=\P{\cS}^{-1}$, we obtain 
		\be\ba 
		\mathbb E_\cS{}&\Big[|\{uv\colon v\in \cU_\infty,   D^{(uv)}=\infty, uv \text{ c.u.t. }V_j\}|\,\Big|\, \cF_T\Big]
		\leq \sum_{m=1}^\infty \sum_{v_1=1}^\infty \cdots \sum_{v_m=1}^\infty  \prod_{i=K}^\infty \E{\e^{\lambda (E_i'-E_i)}}\\
		&c \times \prod_{i=0}^{K-1}\E{\e^{-\lambda E_i}}\prod_{\ell=1}^m\bigg(  \P{D\geq v_\ell}\E{\exp\bigg(-\lambda \sum_{i=0}^{v_\ell-1}E_i\bigg)}\bigg).
		\ea\ee 
		By switching summation and products, this equals 
		\be \ba 
		c \prod_{i=K}^\infty \E{\e^{\lambda (E_i'-E_i)}}  \prod_{i=0}^{K-1}\E{\e^{-\lambda E_i}}\sum_{m=1}^\infty\Bigg( \sum_{v=1}^\infty    \P{D\geq v }\E{\exp\bigg(-\lambda \sum_{i=0}^{v-1}E_i\bigg)} \Bigg)^m.
		\ea\ee 
		By writing the inner sum as $\widehat\mu(\lambda)$ and using that, by Assumption~\eqref{ass:mufin}, we can take $\lambda>0$ large enough so that $\widehat\mu(\lambda)<1$, as follows from~\eqref{eq:muless1}, we conclude that the upper bound is finite. Using this in~\eqref{eq:A+sets},  we thus have that 
		\be 
		\Ps{|\{uv{}\colon  v\in \cU_\infty, D^{(uv)}=\infty, uv \text{ c.u.t. }V_1\text{ or }V_2\}|<\infty\,\Big|\, \cF_T}=1 \qquad \mathbb P_\cS\text{-almost surely}.
		\ee 
		We can thus use the inclusion for $\cP_2$ in~\eqref{eq:P2incl}, combined with having shown that all sets on the right-hand side of~\eqref{eq:A-sets}, \eqref{eq:V2sets}, and~\eqref{eq:A+sets} are finite $\mathbb P_\cS$-almost surely conditionally on $\cF_T$, to finally arrive at 
		\be 
		\Ps{|\cP_2|<\infty\,|\, \cF_T}=1\qquad \mathbb P_\cS\text{-almost surely}.
		\ee 
		By taking expectation with respect to $\cF_T$, we thus arrive at the desired result and conclude the proof.
	\end{proof} 

	\textbf{Adaptations for the proof of Proposition~\ref{prop:catch} with $m=1$ and $m>2$. } When $m=1$, the proof simplifies somewhat. In this case, we instead define the stopping time $T$, as in~\eqref{eq:TTdef}, as the time at which there is at least $1$ individual $u$ such that $D^{(u)}=\infty$ that has produced at least $K$ children by time $T$. As the inter-birth times are exponentially distributed, it follows that this individual, say $V_1$, has produced exactly $K$ children at time $T$. The domination of $\cP_1$ is similar as in~\eqref{eq:P2incl}, though now we no longer need the second line (as $V_2$ does not exist in this case). Furthermore, in the first line we require the individuals $uv$ to catch up to $u$ only, and in the third line to $V_1$ only. Bounding the set on the first line then directly follows from Proposition~\ref{prop:catch}, and bounding the sets on the third line follows in an analogous way to bounding the sets on the right-hand side of~\eqref{eq:A+sets}. 
	
	When $m>2$, we define $T$ as the first time that $m$ many individuals $V_1,\ldots, V_m$ such that $D^{(V_i)}=\infty$ have produced at least $K$ children. We bound $\cP_m$ in the same way as in~\eqref{eq:P2incl}, namely
	\be \ba 
	\cP_m\subseteq \cA_-{}&\cup\!\!\!\!\!\!\!\!\!\!\bigcup_{\substack{u\in\cA_-\setminus\{V_2,\ldots, V_m\}\\ D^{(u)}=\infty}}\!\!\!\!\!\!\!\!\!\!\!\!\!\!\!\{uv\colon v\in \cU_\infty, v_1>K, D^{(uv)}=\infty, uv\text{ c.u.t.\ }u\text{ or }V_i\text{ for some }i\in[m]\setminus\{1\}\}\\
	&\cup \bigcup_{j=2}^m \{V_jv\colon v\in \cU_\infty, v_1>K, D^{(uv)}=\infty, uv\text{ c.u.t.\ }V_i\text{ for some }i\in[m]\}\\
	&\cup\bigcup_{u\in\cA_+}\{uv\colon v\in \cU_\infty, D^{(uv)}=\infty, uv\text{ c.u.t.\ }V_i\text{ for some }i\in[m]\}.
	\ea\ee   
	Proving that each of the sets on the right-hand side is finite then follows similar arguments as in the proof for $m=2$. 

	\section{Persistence in the CTPAVD branching process}\label{sec:finlife}
	
	In this section we prove Theorem~\ref{thrm:conv}. Unlike in the analysis carried out in Section~\ref{sec:inflife}, a persistent $m$-hub, as in the sense of~\eqref{eq:pershub}, does \emph{not} exist $\mathbb P_\cS$-almost surely for any $m\in\N$, since individuals live an almost surely finite time (i.e.\ $L<\infty$ a.s.). Indeed, an almost surely finite lifetime implies that $O_t^\cont$ diverges to infinity with $t$ $\mathbb P_\cS$-almost surely. Since $I_t^{(1),\cont}\geq O_t^\cont$ for all $t$, we thus also have $I_t^{(1),\cont}\to\infty$ $\mathbb P_\cS$-almost surely. To show instead that persistence occurs, in the sense of~\eqref{eq:pers}, we use a more quantitative approach where we provide precise asymptotic expansions of the random variables $O_t^\cont$ and $I_t^{(1),\cont}$ to show that their difference is a tight sequence of random variables (with respect to $t$), as in~\eqref{eq:persnoperscont}. That is, for a mapping $s\mapsto f(s)$ to be determined, we show that for any $\eps>0$ small and $A>0$ large, there exist $t',K,M>0$ such that for all $t\geq t'$, the events
	\begin{align}
		|O_s^\cont -f(s)|&<K \quad \text{for all }s\in[t-A,t+A],\label{eq:Os}
		\intertext{and}  
		I_s^\cont -O_s^\cont&<K \quad \text{for all }s\in [t-A,t+A],\label{eq:IsOs}
	\end{align}
	occur with probability at least $1-\eps$. We first collect several preliminary results in Section~\ref{sec:prelim} and prove~\eqref{eq:Os} in Section~\ref{sec:old}. We then prove several results on the offspring of `young' and `old' individuals in Section~\ref{sec:I}, to finally  prove~\eqref{eq:IsOs} in Section~\ref{sec:pers}. 
	
	\subsection{Preliminary results} \label{sec:prelim}
	
	We first state some preliminary results that are used throughout the remainder of this section. We recall the definition of the CTPAVD branching process in Section~\ref{sec:embed}, and that we write $D$ for the offspring of an individual, $S_k$ for the time an individual needs to give birth to $k$ children, and $L\coloneq S_{D+1}$ for the lifetime of an individual. Furthermore, we  recall the sequences $\varphi_1,\varphi_2,$ and $\rho_1$ from~\eqref{eq:seqs1} and define the sequence
	\be \label{eq:rho2}
	\rho_2(k)\coloneq \sum_{i=0}^{k-1}\Big(\frac{d(i)}{b(i)+d(i)}\Big)^2\qquad\text{for }k\in\N. 
	\ee 
	We extend the domain of $\rho_2$ to $\R_+$ by linear interpolation. We then have the following results.
	
	\begin{lemma}[Offspring tail bounds, Lemmas $5.3$ and $5.4$ in~\cite{HeyLod25}]\label{lemma:Dtail}
		Let $D$ be as in~\eqref{eq:D}. Then,
		\be 
		\P{D\geq k}\leq \e^{-\rho_1(k)-\frac12 \rho_2(k)}. 
		\ee 
		In particular, $\P{D=\infty}=0$ if and only if Assumption~\eqref{ass:A2} is satisfied. Furthermore, assume that $b$ and $d$ are such that Assumption~\eqref{ass:A2} is satisfied. When $\rho_2$ diverges and $d=o(b)$, 
		\be 
		\P{D\geq k}=\e^{-\rho_1(k)-(\frac12+o(1))\rho_2(k)}.
		\ee 
		If, instead, $\lim_{k\to\infty}\rho_2(k)$ exists, 
		\be 
		\P{D\geq k}=\e^{-\rho_1(k) -\cO(1)}.
		\ee 
	\end{lemma} 
	
	The lemma follows from an asymptotic expansion of the tail distribution of $D$, as in~\eqref{eq:D}.

	\begin{lemma}[Probability to survive with many children, Lemma $5.9$ in~\cite{HeyLod25}]\label{lemma:survdeg}
		$\;$\\ $(i)$ Suppose that there exist $x\geq 0$ and $K=K(x)\in\N_0$ such that $d(i)\geq x$ for all $i\geq K$. Then, for all $k\geq K$ and all $t'\geq t\geq0$,
		\be 
		\P{D\geq k, S_k\leq t, S_{D+1}>t'}\leq \e^{-x(t'-t)}\P{D\geq k}\E{\ind_{\{S_k\leq t\}}\e^{x(S_k-t)}}.
		\ee 
		$(ii)$ Suppose that there exist $x\geq 0$ and $K=K(x)\in\N$ such that $d(i)\leq x$ for all $i\geq K$. Then, for all $k\geq K$ and all $t'\geq t\geq0$,  
		\be 
		\P{D\geq k, S_k\leq t, S_{D+1}>t'}\geq \e^{-x(t'-t)}\P{D\geq k}\E{\ind_{\{S_k\leq t\}}\e^{x(S_k-t)}}.
		\ee 
	\end{lemma} 
	
	Recall that $\cS$ denotes the event that the branching process survives, and that
	\be \label{eq:Ps} 
	\Ps{\cdot}\coloneq \P{\cdot\,|\,\cS} \quad\text{and}\quad \mathbb E_{\cS}[\cdot]\coloneq \E{\cdot\,|\,\cS}
	\ee 
	denote the conditional probability measure and its corresponding expected value, under the event $\cS$. For $0\leq s<t<\infty$, we further define 
	\be 
	\cB(s,t)\coloneq \{u\in \cU_\infty: \cB(u)\in [s,t]\}
	\ee 
	as the set of individuals born in the time interval $[s,t]$. Also, recall $\cA_t^\cont$ from~\eqref{eq:At} as the set of individuals alive at time $t$, and $\cR$ from~\eqref{eq:D} as the point process that governs the production of offspring of individuals. We have the following results regarding the growth-rate of  $\cA^\cont_t$ and $\cB(s,t)$.  The first is similar to the well-known Kesten-Stigum theorem and is a direct consequence of results for general CMJ branching processes in~\cite{Don72,Ner81}; the second is Corollary $6.6$ in~\cite{HeyLod25}.
	
	\begin{proposition}[Growth rate of branching process]\label{prop:growth}
		Suppose that $b$ and $d$ are such that Assumptions~\eqref{ass:A1} and \eqref{ass:C1} are satisfied. There exists a non-negative random variable $W$, such that 
		\be \label{eq:Aconv}
		\e^{-\lambda^*\!t} |\cA^\cont_t|\toas W. 
		\ee 
		Moreover, with
		\be 
		\widehat\cR^{\lambda^*}(t)\coloneq \int_0^t \e^{-\lambda^* u}\,\cR(\dd u), 
		\ee 
		the following are equivalent:
		\begin{enumerate}
			\item $\E{ \widehat\cR^{\lambda^*}(\infty)\log^+ \widehat\cR^{\lambda^*}(\infty)}<\infty$, 
			\item $\E{W}>0$,
			\item $\E{\e^{-\lambda^*\!t}|\cA^\cont_t|}\to \E{W}$ as $t\to\infty$, 
			\item $W>0$ almost surely on $\cS$. 
		\end{enumerate}
	\end{proposition}
	
	Proposition~\ref{prop:growth} is a direct result of~\cite[Theorem $5.4$]{Ner81} and~\cite{Don72} (see also~\cite[Proposition $1.1$]{Ner81} for a condensed version).
	
	\begin{corollary}[Corollary $6.6$ in \cite{HeyLod25}]\label{cor:growth}
		Suppose that $b$ and $d$ are such that Assumptions~\eqref{ass:A1} and \eqref{ass:C1} are satisfied, and let $\cS\coloneq \{\lim_{t\to\infty}|\bp(t)|=\infty\}$ be the event that the process $\mathrm{BP}$ survives. Then,  
		\begin{align}  
			\lim_{M\to\infty} \P{\sup_{t\geq 0} |\cB(0,t)|\e^{-\lambda^*\!t}\geq M}&=0,\label{eq:supB} 
			\intertext{and} 
			\lim_{M\to\infty} \Ps{\inf_{t\geq 0} |\cB(0,t)|\e^{-\lambda^*\!t}\leq 1/M}&=0.\label{eq:infB}
		\end{align}
	\end{corollary}

	We conclude this sub-section with a few elementary results regarding the functions $\varphi_1, \varphi_2$, and $\rho_1, \rho_2$, as well as $\alpha$ and $\cK_\alpha$, as introduced in Section~\ref{sec:results}. 
	
	\begin{lemma}[Lemma $5.13$, \cite{HeyLod25}]\label{lemma:func} 
		$\,$ 
		\begin{enumerate}[label=(\alph*)]
			\item Suppose that Assumption~\eqref{ass:A2} is satisfied and $d=o(b)$. Then, $\rho_2=o(\rho_1)$. 
			\item Suppose that Assumption~\eqref{ass:A1} is satisfied and that $\overline d\coloneq \limsup_{i\to\infty}d(i)<\infty$. Then, $\limsup_{k\to\infty} \rho_1(k)/\varphi_1(k)\leq \overline d$. In particular, when $d$ converges to zero, then $\rho_1=o(\varphi_1)$. Similarly, if $\underline d\coloneq \limsup_{i\to\infty}d(i)>0$, then $\liminf_{k\to\infty}\rho_1(k)/\varphi_1(k)\geq \underline d$. In particular, if $d$ tends to infinity, then $\varphi_1=o(\rho_1)$.
			\item Suppose that Assumption~\eqref{ass:A1} is satisfied and $\lim_{i\to\infty}d(i)=d^*\in[0,\infty)$. Then, we have $\cK_\alpha(t)=o(t)$. Moreover, for any function $s\colon (0,\infty)\to\R_+$ such that $s(t)=o(t)$, we have $\cK_\alpha(t)-\cK_\alpha(t-s(t))=o(s(t))$.
		\end{enumerate}
	\end{lemma}

	\subsection{The oldest alive individual}\label{sec:old}
	
	In this sub-section we study the quantity $O_t^\cont$; the birth-time of the oldest alive individual in $\bp(t)$ conditionally on survival. To this end, we first obtain precise asymptotic distributional properties of the lifetime $L=S_{D+1}$ of individuals in the branching process, to then prove an asymptotic expansion for $O_t^\cont$ as $t\to\infty$. 
	
	\subsubsection{Lifetime distribution} \label{sec:lifetime}
	
	We know from~\cite[Lemma $6.1$]{HeyLod25} that the lifetime $L$ of an individual is asymptotically distributed as an exponential random variable with rate $d^*$ when $d$ converges to $d^*\leq R$. That is $\P{L>t}=\exp(-(d^*+o(1))t)$ as $t$ tends to infinity. A more precise asymptotic expansion is provided in~\cite[Lemma $6.3$]{HeyLod25} when Assumption~\eqref{ass:varphi2} is \emph{not} satisfied. We complement this result by the following lemma, which provides  a (different) higher-order expansion when Assumption~\eqref{ass:varphi2} is satisfied.

	\begin{lemma}[Third-order asymptotic lifetime distribution]\label{lemma:lifetime2nd}
		Suppose that $b$ and $d$ are such that Assumptions~\eqref{ass:A1}, \eqref{ass:A2}, and~\eqref{ass:varphi2} are satisfied, and recall $R$ from~\eqref{eq:R}. Also suppose that $\lim_{i\to\infty}d(i)=d^*\in[0,R)$. Recall the function $\cK_\alpha$ from~\eqref{eq:Ks}, and suppose that Assumption~\hyperref[ass:Kalpha]{$\cK_\alpha$} is satisfied.  Then,
		\be 
		\sup_{t\geq 0}\big|\log(\P{L>t})+d^*t+\cK_\alpha(t)\big|<\infty.
		\ee 
	\end{lemma}

	\begin{proof}
		We start by making the following observations. Since $\varphi_2$ converges and $d$ converges, it follows that $\rho_2$, as in~\eqref{eq:rho2}, converges. Furthermore, as $\varphi_2(k)=\Var(S_k)=\Var(S_k-\varphi_1(k))$ converges and $\varphi_1(k)=\E{S_k}$, it follows that $S_k-\varphi_1(k)$ converges almost surely to some random variable $M_\infty$ by Kolmogorov's two series theorem~\cite[Theorem 2.5.6]{Durr19}.
		
		\textbf{Lower bound.} Set $k\coloneq \floor{\varphi_1^{-1}(t)}$. Then, using Lemma~\ref{lemma:Dtail} and the above observations,
		\be\ba 
		\P{L>t}&\geq \P{D\geq k}\P{S_k>t}\\ 
		&=\exp(-\rho_1(k)+\cO(1))\P{S_k>t}\\ 
		&\geq \exp(-\rho_1(\varphi_1^{-1}(t))+\cO(1))\P{S_k-\varphi_1(k)>0}\\ 
		&=\exp(-d^*t-\cK_\alpha(t)+\cO(1))(\P{M_\infty>0}+o(1)),
		\ea \ee
		which yields the desired lower bound.
		
		\textbf{Upper bound.} We set
		\be 
		M=M(t)\coloneq \ceil{(1-\eps)t} \qquad\text{and}\qquad k_i\coloneq \ceil{\varphi_1^{-1}(t-i)}\quad \text{for }i\in \{0,\ldots, M\}.
		\ee 
		Further, we fix $\theta\in(d^*,R)$ and take $\eps\in(0,1)$ small enough such that $\theta(1-\eps)>d^*$. Then, using Lemma~\ref{lemma:Dtail} and Chernoff bounds, 
		\be \ba 
		\P{L>t}&\leq \P{D\geq k_0}+\P{S_{k_M}>t}+\sum_{i=0}^{M-1}\P{D\geq k_{i+1}}\P{S_{k_i}>t}\\ 
		&\leq  \e^{-\rho_1(k_0)}+\e^{-\theta t}\prod_{i=0}^{k_M-1}\frac{b(i)+d(i)}{b(i)+d(i)-\theta}+\e^{-\theta t}\sum_{i=0}^{M-1}\e^{-\rho_1(k_{i+1})}\prod_{i=0}^{k_i-1}\frac{b(i)+d(i)}{b(i)+d(i)-\theta}.
		\ea \ee 
		Since $\varphi_2$ converges, it follows that for any $i\leq M$,
		\be 
		\prod_{i=0}^{k_i-1}\frac{b(i)+d(i)}{b(i)+d(i)-\theta}=\exp(\theta \varphi_1(k_i)+\cO(1))= \exp(\theta t-\theta i+\cO(1)), 
		\ee 
		where the $\cO(1)$ term is independent of $t$ and $i$. By the definition of $\alpha$ in~\eqref{eq:alpha}, we thus arrive at the upper bound 
		\be \label{eq:bigub}
		\e^{-d^*t-\cK_\alpha(t)}+\e^{-\theta (1-\eps)t}+\sum_{i=0}^{M-1}\e^{-d^*t-\cK_\alpha(t-(i+1))-(\theta -d^*)i +\cO(1)}.
		\ee 
		As $\theta(1-\eps)>d^*$ by the choice of $\theta$ and $\eps$, and $\cK_\alpha=o(t)$ by Lemma~\ref{lemma:func}$(c)$, it follows that the second term is negligible compared to the first term. To deal with the sum, we define 
		\be 
		\wt\alpha(k)=\sum_{i=0}^{k-1}\frac{|d(i)-d^*|}{b(i)+d(i)}\qquad \text{for }k\in\N_0, \quad\text{and}\quad \cK_{\wt\alpha}(t)\coloneq \wt\alpha\big(\varphi_1^{-1}(t)\big) \qquad \text{for }t\geq0, 
		\ee 
		where we extend the domain of $\wt\alpha$ to $\R_+$ by linear interpolation. We then distinguish between two cases. 
		
		\textbf{$\boldsymbol{\cK_{\wt \alpha}}$ converges. } In this case, it follows that $\cK_\alpha$ is bounded, as $|\cK_\alpha(t)|\leq \cK_{\wt \alpha}(t)$ for all $t\geq 0$. Hence the term $\cK_\alpha(t-(i+1))$ in the exponential summands in~\eqref{eq:bigub} can be included in the $\cO(1)$ term  for each $i\leq M-1$, to yield
		\be 
		\sum_{i=0}^{M-1}\e^{-d^*t-\cK_\alpha(t-(i+1))-(\theta -d^*)i +\cO(1)}=\sum_{i=0}^{M-1}\e^{-d^*t-(\theta -d^*)i +\cO(1)}\leq C\e^{-d^*t}, 
		\ee 
		for some constant $C>0$. Combined with the first two terms in~\eqref{eq:bigub} this implies the desired result, since $\cK_\alpha$ is bounded. 
		
		\textbf{$\boldsymbol{\cK_{\wt \alpha}}$ tends to infinity. } In this case, we let $C'>2(\theta-d^*)^{-1}$ be a constant and split the sum in~\eqref{eq:bigub} into two parts, namely 
		\be \ba \label{eq:twosum}
		\sum_{i=0}^{M-1}\e^{-d^*t-\cK_\alpha(t-(i+1))-(\theta -d^*)i +\cO(1)}={}&\sum_{i=0}^{C' \cK_{\wt \alpha}(t)}\e^{-d^*t-\cK_\alpha(t-(i+1))-(\theta -d^*)i +\cO(1)}\\ 
		&+\sum_{i=C'\cK_{\wt\alpha}(t)}^{M-1}\e^{-d^*t-\cK_\alpha(t-(i+1))-(\theta -d^*)i +\cO(1)}.
		\ea \ee 
		We note that this is possible since $\cK_{\wt\alpha}(t)=o(t)$ (this follows in the same way as $\cK_\alpha(t)=o(t)$ in Lemma~\ref{lemma:func}$(c)$) and $M\geq (1-\eps)t$. Using the second part of Lemma~\ref{lemma:func}$(c)$  yields
		\be 
		\cK_\alpha(t-(i+1))-\cK_\alpha(t)=o(i) \quad \text{for any }i\leq C'\cK_{\wt\alpha}(t). 
		\ee 
		As a result, since $\theta>d^*$, the first sum on the right-hand side of~\eqref{eq:twosum} can be bounded from above for all $t$ large by 
		\be \label{eq:firstsumbound}
		\e^{-d^*t-\cK_\alpha(t)}\sum_{i=0}^{C'\cK_{\wt\alpha}(t)}\e^{-(\theta-d^*+o(1))i+\cO(1)}\leq C_2\e^{-d^*t-\cK_\alpha(t)}, 
		\ee 
		for some constant $C_2>0$. For the second sum on the right-hand side of~\eqref{eq:twosum}, we bound for each index $i$ in the sum,
		\be 
		-\cK_\alpha(t-(i+1))\leq \sup_{x\in [\eps t,t]}|\cK_\alpha(x)|\leq \sup_{x\in[\eps t,t]}\cK_{\wt\alpha}(x)=\cK_{\wt\alpha}(t), 
		\ee  
		where the final step uses that $\cK_{\wt\alpha}$ is increasing. As a result, we obtain 
		\be 
		\sum_{i=C'\cK_{\wt\alpha}(t)}^{M-1}\!\!\!\!\!\!\!\exp(-d^*t-\cK_\alpha(t-(i+1))-(\theta -d^*)i +\cO(1))\leq\exp(-d^*t+\cK_{\wt \alpha}(t)-(\theta -d^*)C'\cK_{\wt \alpha}(t) +\cO(1)).
		\ee 
		By the choice of $C'$, the upper bound is at most 
		\be 
		\exp(-d^*t-\cK_{\wt \alpha}(t)+\cO(1))\leq \exp(-d^*t-\cK_\alpha(t)+\cO(1)).
		\ee 
		Combined with the bound in~\eqref{eq:firstsumbound} for the first sum on the right-hand side of~\eqref{eq:twosum} and the first two terms in~\eqref{eq:bigub}, we thus arrive at the desired result and conclude the proof.
	\end{proof} 
	
	\subsubsection{Non-survival of old individuals}\label{sec:oldproof}
	
	We conclude this sub-section by proving asymptotic results for $O_t^\cont$. In essence, we show that all individuals that are `too old' have all died by time $t$ and that there exists `young enough' individuals that survive up to time $t$ with probability close to one. This is made precise in the following result.
	
	\begin{proposition}[The oldest alive individual]\label{prop:Otpers}
		Suppose that the sequences $b$ and $d$ are such that Assumptions~\eqref{ass:A1}, \eqref{ass:A2}, ~\eqref{ass:C1}, and~\eqref{ass:varphi2} are satisfied. Recall $R$ from~\eqref{eq:R} and recall that the event $\cS$ denotes that the branching process $\bp$ survives. Suppose that $\lim_{i\to\infty} d(i)=d^*\in[0,R)$ and that Assumption~\hyperref[ass:Kalpha]{$\cK_\alpha$} is satisfied. Then, for any $\eps,A>0$ there exist $K,t'>0$ such that for all $t>t'$, 
		\be 
		\Ps{\forall s\in [t-A,t+A]: \Big|O_s^\cont-\frac{d^*}{\lambda^*+d^*}t-\frac{1}{\lambda^*+d^*}\cK_\alpha(r(t))\Big|<K}>1-\eps. 
		\ee 
	\end{proposition} 
	
	\begin{proof}
		For $K>0$, we define 
		\be 
		\ell_t\coloneq \frac{d^*}{\lambda^*+d^*}t+\frac{1}{\lambda^*+d^*}\cK_\alpha(r(t))-K, \qquad u_t\coloneq \frac{d^*}{\lambda^*+d^*}t+\frac{1}{\lambda^*+d^*}\cK_\alpha(r(t))+K.
		\ee 
		Then, as $O_s^\cont$ is increasing in $s$, 
		\be \ba\label{eq:Otincr}
		\mathbb P_\cS{}&\bigg(\forall s\in[t-A,t+A]:\Big|O^\cont_s-\frac{d^*}{\lambda^*+d^*}t-\frac{1}{\lambda^*+d^*}\cK_\alpha(r(t))\Big|<K\bigg)\\ 
		&\geq 1-\Ps{O^\cont_{t-A}\leq \ell_t}-\Ps{O^\cont_{t+A}\geq u_t}. 
		\ea \ee 
		By conditioning on the number of individuals born up to time $\ell_t$ and using a union bound, we bound the first probability on the right-hand side from above by 
		\be \ba
		\Ps{O^\cont_{t-A}\leq \ell_t}&=\Ps{\exists v\in \cB(0,\ell_t): L^{(v)}>t-A-\cB(v)}\\ 
		&\leq \frac{M}{\P{\cS}}\e^{\lambda^* \ell_t}\P{L>t-A-\ell_t}+\Ps{\sup_{t\geq 0}|\cB(0,t)|\e^{-\lambda^* t}>M}.
		\ea\ee 
		The second probability on the right-hand side is at most $\eps/4$ when we choose $M=M(\eps)$ sufficiently large, as follows from~\eqref{eq:supB} in Corollary~\ref{cor:growth}. For the first probability, we use Lemma~\ref{lemma:lifetime2nd} and the definition of $\ell_t$ to further bound this from above by 
		\be \ba 
		\exp\big(\lambda^* {}&\ell_t-d^*(t-A-\ell_t)-\cK_\alpha(t-A-\ell_t)+C_1\big)\\ 
		= \exp\Big({}& \cK_\alpha(r(t))-\cK_\alpha\Big(\frac{\lambda^*}{\lambda^*+d^*}t-\frac{1}{\lambda^*+d^*}\cK_\alpha(r(t))+ K-A\Big)-(\lambda^*+d^*) K+d^*A+C_1\Big),
		\ea \ee 
		where $C_1>0$ is a sufficiently large constant. We use Lemma~\ref{lemma:func}$(c)$ to obtain that for any $K,A>0$ there exists $t'=t'(K,A)>0$ such that for all $t\geq t'$, 
		\be \label{eq:Kalphabound}
		\Big|\cK_\alpha\Big(\frac{\lambda^*}{\lambda^*+d^*}t-\frac{1}{\lambda^*+d^*}\cK_\alpha(r(t))+ K-A\Big)-\cK_\alpha\Big(\frac{\lambda^*}{\lambda^*+d^*}t-\frac{1}{\lambda^*+d^*}\cK_\alpha(r(t))\Big)\Big|<1. 
		\ee 
		Then,  Assumption~\hyperref[ass:Kalpha]{$\cK_\alpha$} yields that there exists a constant $C_2>0$ such that for all $t>0$, 
		\be \label{eq:Kalphabound2}
		\Big|\cK_\alpha\Big(\frac{\lambda^*}{\lambda^*+d^*}t-\frac{1}{\lambda^*+d^*}\cK_\alpha(r(t))\Big)-\cK_\alpha(r(t))\Big|\leq C_2.
		\ee 
		We thus arrive for all $t\geq t'$ at the upper bound
		\be 
		\Ps{O^\cont_{t-A}\leq \ell_t}\leq \exp\big(C_2+1-(\lambda^*+d^*) K+d^*A+C_1\big)+\eps/4\leq \eps/2,
		\ee  
		where we choose $K$ sufficiently large so that the exponential term is at most $\eps/4$ in the final step. We note that we can first choose such a $K$ and then take $t'=t'(K,A)$ sufficiently large such that~\eqref{eq:Kalphabound} holds.
		
		We then bound the second probability on the right-hand side of~\eqref{eq:Otincr}. In a similar way as for the first term, with $\wt u_t=u_t-K/2$, 
		\be \label{eq:Otbound}
		\Ps{O^\cont_{t+A}\geq u_t}\leq \Ps{\forall v\in \cB(\wt u_t,u_t):L^{(v)}\leq t+A-\wt u_t}.
		\ee 
		For ease of writing, we define the events 
		\be 
		\cE_t\coloneq \{\forall v\in \cB(\wt u_t,u_t): L^{(v)}\leq  t+A-\wt u_t\}, \qquad \cD_t\coloneq \{|\cB(\wt u_t,u_t)|>m\} \quad \text{for }m\in\N,t \geq0.
		\ee 
		We can thus bound 
		\be 
		\Ps{O^\cont_{t+A}\geq u_t}\leq \Ps{\cE_t}=\P{\cS}^{-1}\E{\ind_\cS\ind_{\cE_t}}\leq \P{\cS}^{-1}\E{\ind_\cS\ind_{\cE_t}\ind_{\cD_t}}+\P{\cS}^{-1}\E{\ind_\cS\ind_{\cD_t^c}}. 
		\ee 
		The second term on the right-hand side equals the probability of the event $\cD_t^c$ under the conditional probability measure $\mathbb P_\cS$. For the first term, we condition on two things. First, we condition $\cF_{\wt u_t}$, the $\sigma$-algebra generated by the branching process up to time $\wt u_t$. Second, we condition on $|\cB(\wt u_t, u_t)|$, the number of individuals born in the branching process in the interval $(\wt u_t,u_t)$. We can thus write
		\be \ba 
		\P{\cS}^{-1}\E{\ind_\cS\ind_{\cE_t}\ind_{\cD_t}}&=\P{\cS}^{-1}\E{\E{\ind_\cS\ind_{\cE_t}\ind_{\cD_t}\,|\, \cF_{\wt u_t}, |\cB(\wt u_t,u_t)|}}\\ 
		&\leq \P{\cS}^{-1}\E{\ind_{\cD_t}\E{\ind_{\cE_t}\,|\, \cF_{\wt u_t}, |\cB(\wt u_t,u_t)|}}\\ 
		&=\P{\cS}^{-1}\E{\ind_{\cD_t}\P{\cE_t\,|\, \cF_{\wt u_t}, |\cB(\wt u_t,u_t)|}}.
		\ea \ee 
		Conditionally on $|\cB(\wt u_t, u_t)|$ and upon the event $\cD_t$, we can use the independence of the lifetimes of distinct individuals to bound the terms in the expected value from above by 
		\be 
		\ind_{\cD_t}\P{\cE_t\,|\, \cF_{\wt u_t}, |\cB(\wt u_t,u_t)|}\leq \ind_{\cD_t}\big(1-\P{L\geq t+A-\wt u_t}\big)^m\leq \exp\big(-m\P{L\geq t+A-\wt u_t}\big),
		\ee 
		where we have omitted the indicator random variable and used that $1-x\leq \e^{-x}$ for $x\in\R$ in the last step. Combining all of the above  we thus arrive at 
		\be \label{eq:Otmsplit}
		\Ps{O_{t+A}^\cont\geq u_t}\leq \P{\cS}^{-1}\exp\big(-m\P{L\geq t+A-\wt u_t}\big)+\Ps{|\cB(\wt u_t,u_t)|\leq m}.
		\ee 
		We choose $m=M^{-1}\exp(\lambda^* u_t)$ for a large constant $M>0$ to be determined later to obtain the upper bound
		\be \label{eq:utbound2}
		\P{\cS}^{-1}\exp\big(-M^{-1}\P{L\geq t+A-\wt u_t}\e^{\lambda^*u_t }\big)+\Ps{|\cB(\wt u_t,u_t)|\leq M^{-1}\e^{\lambda^* u_t}}.
		\ee 
		We bound the first term from above and take care of the second term later. We use Lemma~\ref{lemma:lifetime2nd} and the definition of $u_t$ and $\wt u_t$ to arrive at the upper bound
		\be \ba 
		\P{\cS}^{-1}\exp\big(-M^{-1}\exp\big(-d^*(t{}&+A-\wt u_t)-\cK_\alpha(t+A-\wt u_t)-C_3+\lambda^* u_t\big)\big)\\ 
		=\mathbb P(\cS)^{-1}\exp\Big(-M^{-1}\exp\Big({}&\cK_\alpha(r(t))-\cK_\alpha\Big(\frac{\lambda^*}{\lambda^*+d^*}t-\frac{1}{\lambda^*+d^*}\cK_\alpha(r(t))-\frac K2\Big)\\ 
		&+(d^*/2+\lambda^*)K-d^*A-C_3 \Big)\Big),
		\ea \ee 
		where $C_3>0$ is a large constant. We again use Lemma~\ref{lemma:func}$(c)$ and Assumption~\hyperref[ass:Kalpha]{$\cK_\alpha$} to obtain similar bounds as in~\eqref{eq:Kalphabound} and~\eqref{eq:Kalphabound2}. We thus obtain, for some constants $C_4>0$ and  $t'=t'(K,A)>0$ that for all $t>t'$ we have have the upper bound
		\be 
		\mathbb P(\cS)^{-1}\exp\Big(-M^{-1}\exp\Big(C_4+(d^*/2+\lambda^*)K-d^*A-C_3 \Big)\Big).
		\ee 
		Given $A$, we bound $C_4-d^*A-C_3\geq -C_5$ for some large constant $C_5=C_5(A,d^*)$>0. Then, we choose $K$ sufficiently large so that the entire term is at most $\eps/4$. Again, we can first choose such a $K$ and then choose $t'$.
		
		It remains to bound the final term in~\eqref{eq:utbound2} from above by $\eps/4$ to obtain the desired result. To this end, we write
		\be\ba
		\Ps{|\cB(\wt u_t,u_t)|\leq M^{-1}\e^{\lambda^* u_t}}={}&\Ps{|\cB(0,u_t)|\e^{-\lambda^* u_t}-|\cB(0,\wt  u_t)|\e^{-\lambda^* \wt u_t-\lambda^* K/2}\leq M^{-1}}\\
		\leq {}&\Ps{|\cB(0,u_t)|\e^{-\lambda^* u_t}\leq \frac2M}\\
		&+\Ps{|\cB(0,\wt u_t)|\e^{-\lambda^*u_t}\geq \frac1M \e^{\lambda^*\frac K2}}.
		\ea\ee 
		We can  bound the first term on the right-hand side from above by 
		\be 
		\Ps{\inf_{t\geq 0}|\cB(0,t)|\e^{-\lambda^*t}\leq \frac2M}, 
		\ee 
		which can be made smaller than $\eps/6$ when taking $M$ sufficiently large by~\eqref{eq:infB} in Corollary~\ref{cor:growth}. Given such a fixed $M$, we then choose $K$ sufficiently large so that, by~\eqref{eq:supB} in Corollary~\ref{prop:growth}, 
		\be 
		\Ps{\sup_{t\geq 0}|\cB(0,t)|\e^{-\lambda^*t} \geq \frac1M\e^{\lambda^*K/2}}<\eps/12. 
		\ee 
		Combined, we thus obtain
		\be 
		\Ps{|\cB(\wt u_t,u_t)|\leq M^{-1}\e^{\lambda^* u_t}}<\eps/4,
		\ee 
		as desired, which concludes the proof.
	\end{proof}

	\subsection{The richest individual}\label{sec:I}
	
	In this sub-section we provide an asymptotic expansion of $I^\cont_t$, the birth-time of the richest individual in $\bp(t)$. That is, the individual that is alive at time $t$ and has produced the largest offspring by time $t$. To this end, we show that individuals that are `too young' produce `small' offspring, whereas there exist `old' individuals that produce large offspring. Combined, these then imply that $I_t^\cont$ is close to $O_t^\cont$.
	
	\subsubsection{Young individuals have small offspring and no individual has too large offspring} 
	
	We start by showing that individuals that are alive in an interval of time centred at $ t$ but born `too late', that is, born too long after the birth-time $O_t^\cont$ of the oldest alive individual (at time $t$), have too small offspring to attain the largest offspring among all alive individuals. Similarly, we show that there does not exist any individual that attains a `too large' offspring by time $\approx t$.
	
	Fix $C_1,C_1'>0$ and set 
	\be\ba  \label{eq:kut}
	k&\coloneq \ceil{\varphi_1^{-1}\big(\tfrac{\lambda^*}{\lambda^*+d^*}t-\tfrac{1}{\lambda^*+d^*}\cK_\alpha(r(t))-C_1\big)},\qquad u_t\coloneq \frac{d^*}{\lambda^*+d^*}t+\frac{1}{\lambda^*+d^*}\cK_\alpha(r(t)),\\ 
	k'&\coloneq \ceil{\varphi_1^{-1}\big(\tfrac{\lambda^*}{\lambda^*+d^*}t-\tfrac{1}{\lambda^*+d^*}\cK_\alpha(r(t))+C_1'\big)},
	\ea\ee  
	where we recall the function $r$ from Assumption~\hyperref[ass:Kalpha]{$\cK_\alpha$}. Then, we have the following result. 
	
	\begin{lemma}\label{lemma:oldsmalloff}
		Suppose $b$ and $d$ satisfy Assumptions~\eqref{ass:A1}, \eqref{ass:A2}, \eqref{ass:C1}, and \eqref{ass:varphi2}. Recall $R$ from~\eqref{eq:R} and suppose that $\lim_{i\to\infty}d(i)=d^*\in[0,R)$. Finally, suppose that Assumption~\hyperref[ass:Kalpha]{$\cK_\alpha$} is satisfied. For any $\eps, A,C_1>0$ there exist $C_2,t'>0$ such that for all $t>t'$,
		\be 
		\Ps{\forall s\in[t-A,t+A]\ \forall v\in \cB(u_t+C_2,t)\cap \cA_s^\cont: \deg_s(v)<k}\geq 1-\eps.
		\ee 
		Similarly, for any $\eps,A>0$ there exists $C_1',t'>0$ such that for all $t>t'$, 
		\be 
		\Ps{\forall s\in[t-A,t+A]\ \forall v\in \cA_s^\cont: \deg_s(v)<k'}\geq 1-\eps.
		\ee 
	\end{lemma}
	
	The first result states that all individuals born `too late', namely after time $u_t+C_2$, that survive up to time $s\in [t-A,t+A]$, have a degree smaller than $k$ with probability close to $1$. The second result tells us that no individual that survives up to time $s\in[t-A,t+A]$ has degree larger than $k'$ with probability close to one.
	
	\begin{proof}
		We bound the probability of the complement of the event of interest from above by $\eps$. That is, we bound
		\be \ba 
		\mathbb P_\cS{}&\big(\exists s\in[t-A,t+A]\ \exists v\in \cB(u_t+C_2,t)\cap \cA_s^\cont: \deg_s(v)\geq k\big)\\ 
		={}&\Ps{\exists s\in[t-A,t+A]\ \exists v\in \cB(u_t+C_2,t): D^{(v)}\geq k, S_k^{(v)}\leq s-\cB(v), S^{(v)}_{D^{(v)}+1}>s-\cB(v)}.
		\ea \ee 
		We now introduce the event $\{\sup_{t\geq0 }|\cB(0,t)|\e^{-\lambda^* t}\leq M\}$, for some large but fixed $M>0$ to be determined. We then  partition the set $\cB(u_t,t)$ into sets of individuals born in the intervals $(u_t+C_2+(j-1),u_t+C_2+j)$, i.e.\ into sets $\cB(u_t+C_2+(j-1),u_t+C_2+j)$, with $j\in [\ceil{t-u_t-C_2}]$ and apply a union bound. For individuals $v\in \cB(u_t+C_2+(j-1),u_t+C_2+j)$, we can bound their birth-time $\cB(v)$ from below and above by $u_t+C_2+(j-1)$ and $u_t+C_2+j$, respectively. Similarly we can bound $s$ from below and above by $t-A$ and $t+A$, respectively. As a result, we bound the above probability from above by
		\be \ba \label{eq:sumub}
		\sum_{j=1}^{\ceil{t-u_t-C_2}}\!\!\!\!{}&\frac{M}{\P{\cS}}\e^{\lambda^*(u_t+j)}\P{D\geq k, S_k\leq t+A-(u_t+C_2+(j-1)), S_{D+1}>t-A-(u_t+C_2+j)}\\ 
		&+\mathbb P_\cS\Big(\sup_{t\geq0 }|\cB(0,t)|\e^{-\lambda^* t}\geq M\Big).
		\ea\ee 
		We now choose $M$ large enough so that, by~\eqref{eq:supB} in Corollary~\ref{cor:growth}, the probability on the second line is at most $\eps/2$. It thus remains to bound the sum from above by $\eps/2$ to arrive at the desired result.  For any $j\in[\ceil{t-u_t}]$, we omit the part of the event that $S_{D+1}>t-A-(u_t+C_2+j)$, use the independence of $D$ and $S_k$ and use a Chernoff bound with $\theta>0$ to arrive at the upper bound
		\be 
		\P{D\geq k}\E{\e^{-\theta S_k}}\exp\big(\theta(t+A-(u_t+C_2+(j-1)))\big).
		\ee 
		Applying the inequality $\log(1+x)\geq x-x^2/2$ for $x>0$, we can bound
		\be 
		-\log \E{\e^{-\theta S_k}}=\sum_{i=0}^{k-1}\log\Big(1+\frac{\theta}{b(i)+d(i)}\Big)\geq \theta \varphi_1(k)-\frac{\theta^2}{2}\varphi_2(k). 
		\ee 
		Since $\varphi_2(k)$ is increasing and converges to $\varphi_2(\infty)\coloneq \lim_{k\to\infty}\varphi_2(k)$ as $k$ tends to infinity by Assumption~\eqref{ass:varphi2}, and by using Lemma~\ref{lemma:Dtail}, we finally arrive at the upper bound
		\be \ba
		\P{D\geq k}{}&\E{\e^{-\theta S_k}}\exp\big(\theta(t+A-(u_t+C_2+(j-1)))\big)\\ 
		\leq \exp{}&\Big(-\rho_1(k)-\theta\varphi_1(k)+\frac{\theta^2}{2}\varphi_2(\infty)+\theta(t+A-(u_t+C_2+(j-1)))\Big).
		\ea\ee 
		Note that $\rho_1$ and $\varphi_1$ are increasing and that $\rho_1(x)=d^*\varphi_1(x)+\alpha(x)$ for $x\geq 0$. Then, also using the definitions of $k$ and $u_t$ in~\eqref{eq:kut}, we obtain the upper bound
		\be \ba 
		\exp\Big({}&-\frac{\lambda^*d^*}{\lambda^*+d^*}t+\frac{d^*}{\lambda^*+d^*}\cK_\alpha(r(t))+d^*C_1-\cK_\alpha\big(\tfrac{\lambda^*}{\lambda^*+d^*}t-\tfrac{1}{\lambda^*+d^*}\cK_\alpha(r(t))-C_1\big)\\ 
		&+\theta \big(C_1-C_2+A-(j-1)\big)+\frac{\theta^2}{2}\varphi_2(\infty)\Big).
		\ea\ee 
		We apply Lemma~\ref{lemma:func}$(c)$ and Assumption~\hyperref[ass:Kalpha]{$\cK_\alpha$} and set $\theta=(C_2+(j-1)-(A+C_1))/\varphi_2(\infty)$ (which is positive for all $j\geq 1$ when $C_2$ is large), to bound this for all $t>t'$ by
		\be 
		\exp\Big(-\lambda^*\Big(\frac{d^*}{\lambda^*+d^*}t+\frac{1}{\lambda^*+d^*}\cK_\alpha(r(t))\Big)+C+d^*C_1-\frac12\frac{(C_2+(j-1)-(A+C_1))^2}{\varphi_2(\infty)}\Big),
		\ee 
		where $C>0$ and $t'$ are fixed constants, where $C$ is independent of all other constants and parameters, and $t'=t'(C_1,d^*,\lambda^*)$. By substituting this upper bound in the sum in~\eqref{eq:sumub}, we can bound it from above by 
		\be 
		\frac{M}{\P{\cS}}\sum_{j=1}^{\ceil{t-u_t-C_2}}\exp\Big(\lambda^*C_2+\lambda^*j+C+d^*C_1-\frac12\frac{(C_2+(j-1)-(A+C_1))^2}{\varphi_2(\infty)}\Big).
		\ee 
		We now choose $C_2$ large enough so that $\lambda^*-(C_2-(A+C_1))/\varphi_2(\infty)<0$. This implies that the summands are exponentially decreasing in $j$ and are maximised when $j=1$. This yields, for some constant $c=c(C_2)>0$, the upper bound 
		\be 
		\frac{Mc}{\P{\cS}}\exp\Big(\lambda^*C_2+\lambda^*+C+d^*C_1-\frac12\frac{(C_2-(A+C_1))^2}{\varphi_2(\infty)}\Big).
		\ee 
		The constant $c$ is decreasing in $C_2$, so that we can make this upper bound smaller than $\eps/2$ by again choosing $C_2$ sufficiently large. Using this in~\eqref{eq:sumub} thus yields the desired result.
		
		For the second result in Lemma~\ref{lemma:oldsmalloff}, an analogous proof can be used (where, at least intuitively, one would take $C_2<0$ such that $|C_2|$ is large but fixed and $C_1<0$ such that $|C_1|$ is  sufficiently large and set $C_1'=-C_1$). We leave the details to the reader.
	\end{proof}
	
	\subsubsection{There exists an old individual with large offspring} 
	
	The next step is to show that there exists an `old' individual that survives until time $\approx t$ and has a sufficiently large offspring. Combined with Proposition~\ref{prop:Otpers} and Lemma~\ref{lemma:oldsmalloff}, this implies that the birth-time of the individual with the largest offspring (i.e.\ $I_t^\cont$) cannot be `too large' compared to $O_t^\cont$. We make this precise in the following result.
	
	\begin{lemma}\label{lemma:existoldoff}
		Suppose $b$ and $d$ satisfy Assumptions~\eqref{ass:A1}, \eqref{ass:A2}, \eqref{ass:C1}, and \eqref{ass:varphi2}. Recall $R$ from~\eqref{eq:R} and suppose that $\lim_{i\to\infty}d(i)=d^*\in[0,R)$. Finally, suppose that Assumption~\hyperref[ass:Kalpha]{$\cK_\alpha$} is satisfied. Recall $k=k(C_1)$ and $u_t$ from~\eqref{eq:kut}. For any $\eps,A>0$ there exist $C_1,t'>0$ such that for all $t>t'$, 
		\be 
		\Ps{\forall s\in[t-A,t+A]\ \exists v\in \cB(u_t,u_t+\tfrac12C_1)\cap \cA_s^\cont: \deg_s(v)\geq k}\geq 1-\eps. 
		\ee  
	\end{lemma}
	
	\begin{proof}
		We first rewrite the probability as 
		\be \ba 
		\mathbb P_\cS{}&\big(\forall s\in[t-A,t+A]\ \exists v\in \cB(u_t,u_t+\tfrac12C_1): D^{(v)}\geq k, S^{(v)}_k\leq s-\cB(v), S^{(v)}_{D^{(v)}+1}>s-\cB(v)\big)\\ 
		\geq {}&\mathbb P_\cS\big(\exists v\in \cB(u_t,u_t+\tfrac12C_1): D^{(v)}\geq k, S^{(v)}_k\leq t-A-u_t-\tfrac12C_1, S^{(v)}_{D^{(v)}+1}>t+A-u_t\big).
		\ea \ee 
		Here we bound $s$ and $\cB(v)$ from below and above to obtain the inequality. Similar to the steps in~\eqref{eq:Otbound} through~\eqref{eq:Otmsplit}, we bound the right-hand side from below for any $m\in\N$ by 
		\be \ba \label{eq:mlb}
		1{}&-\Ps{|\cB(u_t,u_t+\tfrac12C_1)|\leq m}\\ 
		&-\P{\cS}^{-1}\exp\Big(-m \P{D\geq k, S_k\leq t-A-u_t-\tfrac12C_1, S_{D+1}>t+A-u_t}\Big).
		\ea\ee 
		We then bound the probability in the exponential term from below. For any $\eta>0$ there exists $t'>0$ such that $d(i)\leq d^*+\eta$ for all $i\geq k$ and all $t>t'$ (since $k$ tends to infinity with $t$). Hence, by using Lemma~\ref{lemma:survdeg},
		\be\ba \label{eq:problb}
		\mathbb P({}&D\geq k, S_k\leq t-A-u_t-\tfrac12C_1, S_{D+1}>t+A-u_t)\\ 
		\geq{}&\P{D\geq k}\E{\ind\{S_k\leq t-u_t-A-\tfrac12C_1\}\exp\big(-(d^*+\eta)(t-u_t-A-\tfrac12C_1-S_k)\big)}\\ 
		&\times \exp\big(-(d^*+\eta)(2A+\tfrac12C_1)\big).
		\ea\ee 
		We bound the expected value from below by restricting the range of $S_k$ in the indicator to $S_k\in[t-u_t-A-C_1,t-u_t-A-\tfrac12C_1]$, to obtain 
		\be \ba \label{eq:explb}
		\mathbb E\big[{}&\ind\{S_k\leq t-u_t-A-\tfrac12C_1\}\exp\big(-(d^*+\eta)(t-u_t-A-\tfrac12C_1-S_k)\big)\big]\\ 
		\geq{}& \mathbb E\big[\ind\{S_k\in[t-u_t-A-C_1, t-u_t-A-\tfrac12C_1]\}\exp\big(-(d^*+\eta)(t-u_t-A-\tfrac12C_1-S_k)\big)\big]\\ 
		\geq{}& \P{S_k\in [t-u_t-A-C_1, t-u_t-A-\tfrac12C_1]}\e^{-\frac12(d^*+\eta)C_1}.
		\ea\ee 
		We then use that $S_k-\varphi_1(k)$ converges almost surely to some random variable $M_\infty$ as $k$ tends to infinity, since $\varphi_2(k)=\Var(S_k)$ converges by Assumption~\eqref{ass:varphi2}. By the choice of $k$ (as in~\eqref{eq:kut}) and since $b+d$ tends to infinity (since $\varphi_2$ converges), we have
		\be \label{eq:varphi1exp}
		\varphi_1(k)=\frac{\lambda^*}{\lambda^*+d^*}t-\frac{1}{\lambda^*+d^*}\cK_\alpha(r(t))-C_1+o(1)=t-u_t-C_1+o(1). 
		\ee 
		As a result, we can write 
		\be \ba 
		\P{S_k\in [t-u_t-A-C_1, t-u_t-A-\tfrac12C_1]}&=\P{S_k-\varphi_1(k)\in [-A+o(1),-A+\tfrac12C_1+o(1)]}\\ 
		&=\P{M_\infty\in[-A,-A+\tfrac12C_1]}+o(1).
		\ea \ee 
		We observe that the right-hand side is strictly positive for any choice of $A$ and $C_1>0$ and can thus be bounded from below by 
		\be \label{eq:cinf}
		c_\infty(A,C_1)\coloneq \frac12\P{M_\infty\in[-A,-A+\tfrac12C_1]},
		\ee 
		for all $t$ (and thus $k$) sufficiently large. Combining this with~\eqref{eq:explb} in~\eqref{eq:problb} yields the lower bound
		\be
		\mathbb P(D\geq k, S_k\leq t-A-u_t, S_{D+1}>t+A-u_t+\tfrac12C_1)\geq c_\infty(A,C_1)\P{D\geq k}\e^{-(d^*+\eta)(2A+C_1)}.
		\ee 
		Then, we use Lemma~\ref{lemma:Dtail} to obtain, for some small constant $\kappa>0$, the lower bound $\P{D\geq k}\geq \kappa \exp(-\rho_1(k))$. As in~\eqref{eq:varphi1exp} we can write, by the definition of $\alpha$ and $\cK_\alpha$ in~\eqref{eq:alpha} and~\eqref{eq:Ks}, respectively,
		\be \ba 
		\rho_1(k)={}&\rho_1\big(\varphi_1^{-1}\big(\tfrac{\lambda^*}{\lambda^*+d^*}t-\tfrac{1}{\lambda^*+d^*}\cK_\alpha(r(t))-C_1\big)\big)+o(1)\\ 
		={}& d^*\Big[\frac{\lambda^*}{\lambda^*+d^*}t-\frac{1}{\lambda^*+d^*}\cK_\alpha(r(t))-C_1\Big]+\cK_\alpha\big(\tfrac{\lambda^*}{\lambda^*+d^*}t-\tfrac{1}{\lambda^*+d^*}\cK_\alpha(r(t))-C_1\big)+o(1).
		\ea \ee 
		Using Lemma~\ref{lemma:func}$(c)$ and Assumption~\hyperref[ass:Kalpha]{$\cK_\alpha$}, we can bound the right-hand side from above for all $t>t''$ by 
		\be 
		\lambda^*\Big[\frac{d^*}{\lambda^*+d^*}t+\frac{1}{\lambda^*+d^*}\cK_\alpha(r(t))\Big]-d^*C_1+C=\lambda^*u_t-d^*C_1+C, 
		\ee 
		where $C>0$ is some large constant, independent of all other constants and parameters, and $t''=t''(C_1,d^*,\lambda^*)$. Hence,
		\be
		\P{D\geq k}\geq \kappa\e^{-\lambda^*u_t+d^*C_1-C}.
		\ee 
		To conclude, we thus obtain the lower bound, 
		\be \ba 
		\mathbb P{}&(D\geq k, S_k\leq t-A-u_t-\tfrac12C_1, S_{D+1}>t+A-u_t)\\ 
		&\geq \kappa c_\infty(A,C_1)\exp\big(-\lambda^*u_t-\eta C_1-2A(d^*+\eta)-C\big).
		\ea\ee 
		We use this in~\eqref{eq:mlb} and set, for some large constant $M>0$ to be determined, 
		\be 
		m\coloneq M^{-1}\exp\big(\lambda^*(u_t+\tfrac12C_1)\big),
		\ee 
		to arrive at 
		\be \ba \label{eq:finbound}
		\mathbb P_\cS(\forall{}& s\in[t-A,t+A]\ \exists v\in \cB(u_t,u_t+\tfrac12C_1)\cap \cA_s^\cont: \deg_s(v)\geq k)\\ 
		\geq1{}&-\Ps{|\cB(u_t,u_t+\tfrac12C_1)|\leq M^{-1}\exp\big(\lambda^*(u_t+\tfrac12C_1)\big)}\\ 
		&-\P{\cS}^{-1}\exp\Big(- \frac{\kappa c_\infty(A,C_1)}{M}\exp\big([\tfrac12\lambda^*-\eta]C_1-2A(d^*+\eta)-C\big)\Big).
		\ea \ee 
		We bound the probability on the second line by
		\be \ba 
		\mathbb P_\cS\big({}&|\cB(u_t,u_t+\tfrac12C_1)|\leq M^{-1}\exp\big(\lambda^*(u_t+\tfrac12C_1)\big)\\ 
		={}&\Ps{|\cB(0,u_t+\tfrac12C_1)|-|\cB(0,u_t)|\leq M^{-1}\exp\big(\lambda^*(u_t+\tfrac12C_1)\big)}\\ 
		\leq{}&\Ps{|\cB(0,u_t+\tfrac12C_1)|\exp\big(-\lambda^*(u_t+\tfrac12C_1)\big) \leq 2M^{-1}}\\
		&+\Ps{ |\cB(0,u_t)|\exp\big(-\lambda^*u_t\big)\geq M^{-1}\exp\big(\tfrac{\lambda^*}{2}C_1\big)}\\ 
		\leq{}&\Ps{\inf_{t\geq0}|\cB(0,t)|\e^{-\lambda^*t} \leq 2M^{-1}}+\P{\sup_{t\geq 0}|\cB(0,t)|\e^{-\lambda^*t}\geq M^{-1}\exp\big( \tfrac{\lambda^*}{2}C_1\big)}.
		\ea \ee 
		We first choose $M$ sufficiently large so that, by~\eqref{eq:infB} in Corollary~\ref{cor:growth}, the first term on the last line is smaller than $\eps/3$. Next, given this $M$, we notice that the second term on the last line is decreasing in $C_1$. Hence, by~\eqref{eq:supB} in Corollary~\ref{cor:growth}, there exists $C_1'>0$ such that for all $C_1>C_1'$, the second term is at most $\eps/3$. We finally turn to the term on the last line of~\eqref{eq:finbound}. Here, we note that $c_\infty(A,C_1)$, as defined in~\eqref{eq:cinf}, is increasing in $C_1$, so that 
		\be 
		c_\infty(A,C_1)\geq c_\infty(A,C'_1)>0,
		\ee 
		for all $C_1>C_1'$. Similarly, by choosing $\eta<\lambda^*/2$, the inner exponential term on the last line of~\eqref{eq:finbound} is also increasing in $C_1$. As a result, we can take $C_1>C_1'$ sufficiently large so that the last line of~\eqref{eq:finbound} is at most $\eps/3$. This yields the desired lower bound and concludes the proof.
	\end{proof}

	\subsection{Persistence in the PAVD model} \label{sec:pers}
	
	In this sub-section we combine the result from the previous sub-sections to prove that $(I^\cont_t-O^\cont_t)_{t\geq 0}$ is a tight sequence of random variables. We then make the correspondence in~\eqref{eq:OtItequiv} precise to translate this into the results in Theorem~\ref{thrm:conv}.
	
	The first part is made precise in the following result. 
	
	\begin{proposition}\label{prop:ItOttight} 
		Suppose $b$ and $d$ satisfy Assumptions~\eqref{ass:A1}, \eqref{ass:A2}, \eqref{ass:C1}, and \eqref{ass:varphi2}. Recall $R$ from~\eqref{eq:R} and suppose that $\lim_{i\to\infty}d(i)=d^*\in[0,R)$. Finally, suppose that Assumption~\hyperref[ass:Kalpha]{$\cK_\alpha$} is satisfied. For any $\eps,A>0$ there exists $M=M(\eps,A),t'>0$ such that for all $t>t'$,
		\be 
		\Ps{\forall s\in[t-A,t+A]: I_s^\cont-O^\cont_s<M}\geq 1-\eps.
		\ee 
	\end{proposition}
	
	\begin{proof}
		Fix $\eps>0$ small, $A>0$ and $t'>0$ large (according to Proposition~\ref{prop:Otpers}, Lemma~\ref{lemma:oldsmalloff}, and Lemma~\ref{lemma:existoldoff}) and recall $k=k(C_1)$ and $u_t$ from~\eqref{eq:kut}. Then, for any $t>t'$:
		\begin{enumerate}
			\item We take $K>0$ sufficiently large so that, with probability at least $1-\eps/3$, for all $s\in [t-A,t+A]$ we have  $O_s^\cont >u_t-K$. 
			\item We take $C_1>0$ sufficiently large so that, with probability at least $1-\eps/3$, for all $s\in[t-A,t+A]$ there exists $ v\in \cB(u_t,u_t+\tfrac12C_1)\cap \cA_s^\cont$ such that $\deg_s(v)\geq k$. 
			\item We take $C_2>\tfrac12C_1$ sufficiently large so that, with probability at least $1-\eps/3$, for all $s\in[t-A,t+A]$ and for all  $v\in \cB(u_t+C_2,t)\cap \cA_s^\cont$ we have $\deg_s(v)<k$.
		\end{enumerate} 
		$(a)$, $(b)$, and $(c)$ follow from Proposition~\ref{prop:Otpers}, Lemma~\ref{lemma:existoldoff}, and Lemma~\ref{lemma:oldsmalloff}, respectively. Then, $(b)$ and $(c)$ combined imply that $I_s^\cont\leq u_t+C_2$ for all $s\in[t-A,t+A]$, with probability at least $1-\tfrac23\eps$.  Combined with $(a)$, we thus conclude that the result follows for $M=C_2+K$.
	\end{proof}
	
	We finally prove Theorem~\ref{thrm:conv} by translating the results proved in this section for the CTPAVD branching process to the desired results for the discrete time PAVD model, leveraging 	Proposition~\ref{prop:embed}. 
	
	\begin{proof}[Proof of Theorem~\ref{thrm:conv}]
		We fix $\eps>0$ and use~\eqref{eq:Nconv} to conclude that there exist $n_0\in\N$ and $A=A(\eps,n_0)>0$ sufficiently large such that 
		\be \label{eq:taubound}
		\Ps{\tau_n \in\big[\tfrac{1}{\lambda^*}\log n-A,\tfrac{1}{\lambda^*}\log n+A\big]}\geq 1-\eps/3, \qquad \text{for all }n\geq n_0.
		\ee 
		We let $F_{n,A}\coloneq \big[\frac{1}{\lambda^*}\log n-A,\frac{1}{\lambda^*}\log n+A\big]$ for ease of writing.	Then, 
		\be \ba \label{eq:Otaun}
		\mathbb P_\cS{}&\bigg(\Big|O_{\tau_n}^\cont-\frac{d^*}{\lambda^*(\lambda^*+d^*)}\log n-\frac{1}{\lambda^*+d^*}\cK_\alpha\big(r\big(\tfrac{1}{\lambda^*}\log n\big)\big)\Big|<M\bigg)\\
		\geq{}&\Ps{\forall s\in F_{n,A}:\Big| O_{s}^\cont-\frac{d^*}{\lambda^*(\lambda^*+d^*)}\log n-\frac{1}{\lambda^*+d^*}\cK_\alpha\big(r\big(\tfrac{1}{\lambda^*}\log n\big)\big)\Big|<M}\\ 
		&-\Ps{\tau_n\not\in F_{n,A}}.
		\ea \ee 
		The second term on the right-hand side is at most $\eps/3$ for all $n\geq n_0$ by~\eqref{eq:taubound}. Using Proposition~\ref{prop:Otpers} with $t=\frac{1}{\lambda^*}\log n$ we can take $M=M(A,\eps)$ and $n_1=n_1(A,\eps)\geq n_0$ such that the first term on the right-hand side is at least $1-\eps/3$ for all $n\geq n_1$, so that the entire lower bound is at least $1-\frac23\eps$.  We then define
		\be 
		\ell_n\coloneq \frac{d^*}{\lambda^*+d^*}\log n+\frac{\lambda^*}{\lambda^*+d^*}\cK_\alpha\big(r\big(\tfrac{1}{\lambda^*}\log n\big)\big),
		\ee 
		and write 
		\be  
		\bigg\{\Big|\log O_n-\frac{d^*}{\lambda^*+d^*}\log n-\frac{\lambda^*}{\lambda^*+d^*}\cK_\alpha\big(r\big(\tfrac{1}{\lambda^*}\log n\big)\big)\Big|>M'\bigg\}=\big\{O_n\not\in\big[\e^{-M'+\ell_n},\e^{M'+\ell_n}\big]\big\}.
		\ee 
		Now, using that $O_n \overset{\mathrm{d}}{=} N(O_{\tau_n}^\cont)$, where we recall that $(N(t))_{t\geq 0}$, as in~\eqref{eq:Nt}, counts the number of births and deaths in the branching process $(\bp(t))_{t\geq 0}$ up to time $t$, we obtain the lower bound
		\be \ba\label{eq:Ontight} 
		\mathbb P_\cS{}&\bigg(\Big|\log O_n-\frac{d^*}{\lambda^*+d^*}\log n-\frac{\lambda^*}{\lambda^*+d^*}\cK_\alpha\big(r\big(\tfrac{1}{\lambda^*}\log n\big)\big)\Big|<M'\bigg)\\ 
		={}&\Ps{O_n\in\big[\e^{-M'+\ell_n},\e^{M'+\ell_n}\big]}\\ 
		\geq{} & \Ps{O_{\tau_n}^\cont  \in \Big[\frac{1}{\lambda^*}\ell_n-\frac{M'}{2\lambda^*},\frac{1}{\lambda^*}\ell_n+\frac{M'}{2\lambda^*}\Big]}-\Ps{N\Big(\frac{1}{\lambda^*}\ell_n-\frac{M'}{2\lambda^*}\Big)\leq \e^{-M'+\ell_n}}\\ 
		&-\Ps{N\Big(\frac{1}{\lambda^*}\ell_n+\frac{M'}{2\lambda^*}\Big)\geq \e^{M'+\ell_n}}.
		\ea\ee 
		The first probability is at least $1-\frac23\eps$ by taking $M'=2\lambda^* M$ and applying~\eqref{eq:Otaun} with $M$ and $n$ sufficiently large. Then, since $N(t)\geq |\cB((0,t))|$ for all $t\geq 0$ deterministically, 
		\be \ba 
		\Ps{N\Big(\frac{1}{\lambda^*}\ell_n-\frac{M'}{2\lambda^*}\Big)\leq \e^{-M'+\ell_n}}&\leq \Ps{\Big|\cB\Big(0,\frac{1}{\lambda^*}\ell_n-\frac{M'}{2\lambda^*}\Big)\Big|\leq \e^{-M'+\ell_n}}\\ 
		&\leq \Ps{\inf_{t\geq 0}|\cB(0,t)|\e^{-\lambda^*t}\leq \e^{-M'/2}},
		\ea\ee 
		which can be made smaller than $\eps/6$ when taking $M'=2\lambda^* M$ large enough by Corollary~\ref{cor:growth}. In a similar manner, using Corollary~\ref{cor:growth} and that $N(t)\leq 2|\cB(0,t)|$, we can bound the final probability on the right-hand side of~\eqref{eq:Ontight} from above by $\eps/6$ as well when $M$ is large enough, to finally yield the lower bound $1-\eps$, which concludes the proof of~\eqref{eq:Onconv}. 
		
		We continue by proving that $I_n/O_n$ is tight. We apply Proposition~\ref{prop:ItOttight} and~\eqref{eq:taubound} to conclude that for any $\eps>0$ there exist $n_1\geq n_0$ and $M=M(\eps,A,n_1)>0$ such that for all $n\geq n_1$,  
		\be\label{eq:contbound}
		\Ps{I^\cont_{\tau_n}-O^\cont_{\tau_n}<M}\geq \Ps{\forall s\in F_{n,A}: I^\cont_s-O^\cont_s<M}-\Ps{\tau_n \not\in F_{n,A}}\geq 1-\frac23\eps.
		\ee 
		Then, we again use that $I_n\overset{\mathrm{d}}{=} N(I^\cont_{\tau_n})$ and $O_n\overset{\mathrm{d}}{=} N(I^\cont_{\tau_n})$. As a result, for some $C,M'>0$ to be determined, 
		\be \ba \label{eq:discbound}
		\Ps{\frac{I_n}{O_n}<M'}=\Ps{\frac{N(I^\cont_{\tau_n})}{N(O^\cont_{\tau_n})}<M'}\geq{}& \Ps{C^2\exp\big(\lambda^*\big(I^\cont_{\tau_n}-O^\cont_{\tau_n}\big)\big)<M'}\\ 
		&-\Ps{N(I^\cont_{\tau_n})\geq C\exp\big(\lambda^*I^\cont_{\tau_n}\big)}\\ 
		&-\Ps{N(O^\cont_{\tau_n})\leq C^{-1}\exp\big(\lambda^*O^\cont_{\tau_n}\big)}.
		\ea \ee 
		We can bound the first term on the right-hand side from below by $1-2\eps/3$ for all $n\geq n_1$ using~\eqref{eq:contbound} and setting $M'\coloneq C^2\e^{\lambda^*M}$. The second term can be bounded from above by using that $N(t)\leq 2|\cB(0,t)|$ deterministically for all $t\geq0$, so that 
		\be \ba 
		\Ps{N(I^\cont_{\tau_n})\geq C\exp\big(\lambda^*I^\cont_{\tau_n}\big)}&\leq \Ps{|\cB(0,I^\cont_{\tau_n})|\exp\big(-\lambda^*I^\cont_{\tau_n}\big)\geq \tfrac C2}\\ 
		&\leq \Ps{\sup_{t\geq0}|\cB(0,t)|\e^{-\lambda^*t}\geq \tfrac C2}.
		\ea \ee 	
		Applying Corollary~\ref{cor:growth} then yields the upper bound $\eps/6$ when we choose $C$ sufficiently large. In a similar way, since $N(t)\geq |\cB(0,t)|$ deterministically for any $t\geq 0$, 
		\be \ba 
		\Ps{N(O^\cont_{\tau_n})\leq C^{-1}\exp\big(\lambda^*O^\cont_{\tau_n}\big)}&\leq \Ps{|\cB(0,O^\cont_{\tau_n})|\exp\big(-\lambda^*O^\cont_{\tau_n}\big)\leq C^{-1}}\\ 
		&\leq \Ps{\inf_{t\geq 0}|\cB(0,t)|\e^{-\lambda^*t}\leq C^{-1}}.
		\ea \ee 
		Again, Corollary~\ref{cor:growth} yields that the last line is at most $\eps/6$ for $C$ large enough. Combining all three bounds in~\eqref{eq:discbound} finally yields
		\be 
		\Ps{\frac{I_n}{O_n}<M'}\geq 1-\eps \qquad\text{for all }n\geq n_1.
		\ee 
		Combined with~\eqref{eq:Onconv}, we directly arrive at~\eqref{eq:Inconv} as well. 
		
		We conclude by proving~\eqref{eq:maxdegconv}. Here, we use that, by Proposition~\ref{prop:embed}, 
		\be 
		\{\varphi_1(\max_{v\in\cA_n}\deg_n(v):n\in\N\}\overset{\mathrm{d}}{=}\{\varphi_1(\max_{v\in \cA_{\tau_n}^\cont}\deg_{\tau_n}(v)):n\in\N\}.
		\ee 
		We thus are required to prove the result for the quantity on the right-hand side. To this end, we write 
		\be \ba
		\mathbb P_\cS{}&\bigg( \Big|\varphi_1(\max_{v\in \cA_{\tau_n}^\cont}\deg_{\tau_n}(v))-\frac{1}{\lambda^*+d^*}\log n+\frac{1}{\lambda^*+d^*}\cK_\alpha\big(r\big(\tfrac{1}{\lambda^*}\log n\big)\big)\Big|<K\bigg)\\ 
		\geq {}& \mathbb P_\cS\bigg(\forall s \in F_{n,A}:  \Big|\varphi_1(\max_{v\in \cA_s^\cont}\deg_s(v))-\frac{1}{\lambda^*+d^*}\log n+\frac{1}{\lambda^*+d^*}\cK_\alpha\big(r\big(\tfrac{1}{\lambda^*}\log n\big)\big)\Big|<K\bigg)\\ 
		&-\Ps{\tau_n\not\in F_{n,A}}.
		\ea\ee 
		The final term is at most $\eps/2$ for $A$ and $n$ sufficiently large by~\eqref{eq:taubound}. The first probability on the right-hand side is also at least $1-\eps/2$ by setting $C_1'=K$ in Lemma~\ref{lemma:oldsmalloff} and $C_1=K$ in Lemma~\ref{lemma:existoldoff} and taking $K$ sufficiently large (with $t=\frac{1}{\lambda^*}\log n$ and $n$ sufficiently large). We thus arrive at the desired result, which concludes the proof.
	\end{proof}
	
	\section{Discussion and Open Problems}\label{sec:disc}
	
	In this article we have analysed the existence persistent $m$-hubs and the occurrence of persistence for the \emph{preferential attachment tree with vertex death} model. In both the `infinite lifetime' regime and the `finite lifetime' regime, we have seen that the summability condition in~\eqref{ass:varphi2} is crucial for the emergence of either persistent hubs or of persistence. Complemented by the study of lack of persistence in~\cite{HeyLod25}, this provides a good overview of the different regimes and the behaviour therein. Still, there are several further extensions and open problems that we would like to state here.

	\textbf{Characterising the difference between the richest and oldest individual. }  In the `rich are old' regime, we have showed that the difference between the birth-time of the richest and oldest individual in $(\bp(t))_{t\geq 0}$, that is, $I^\cont_t-O^\cont_t$, is a tight sequence of random variables with respect to the probability measure $\mathbb P_\cS$. It remains unclear if more precise results can be proved. 
	
	\begin{problem}
		Under the conditions in Theorem~\ref{thrm:conv}, characterize the random process $(I^\cont_t-O^\cont_t)_{t\geq0}$. 
	\end{problem}
	
	\textbf{Bounded non-converging death rates. } The results presented in the `finite lifetime' regime do not consider the case when the death rates are bounded but do not converge. This is mainly due to the difficulty to determine the lifetime distribution of individuals in complete generality. As an example, let us consider the case of converging birth rates and alternating convergent death rates. That is, take three sequences $(c_i)_{i\in\N}$, $(d_{1,i})_{i\in\N}$ and $(d_{2,i})_{i\in\N}$ such that $\lim_{i\to\infty}c_i=c>0$ and $\lim_{i\to\infty}d_{j,i}=d_j$ for $j\in\{1,2\}$, with $d_1\neq d_2\geq 0$. We set
	\be 
	d=(d_{1,1},d_{2,1},d_{1,2},d_{2,2},\ldots), \qquad b(i)=c_i.
	\ee 
	We then claim (without proof) that the lifetime distribution satisfies 
	\be 
	\lim_{t\to\infty} \frac1t\log \P{L>t}=-\min\bigg\{c+\frac{d_1+d_2}{2}-\sqrt{c^2+\Big(\frac{d_1-d_2}{2}\Big)^2},R\bigg\}=:-\theta. 
	\ee 
	This already non-trivial result can be obtained due to the somewhat simple choice of the birth and death rates. More general birth or death rates (e.g.\ $b$ that diverge to infinity) make it much harder to determine the limiting value of $\theta$ that characterises the exponential decay of the lifetime distribution. 
	
	We expect that in this more general setting, persistence can be shown to occur when $\theta<R$ and when Assumption~\eqref{ass:varphi2} is satisfied.  We leave this as an open problem. 
	
	\begin{problem}
		Suppose the sequences $b$ and $d$ are such that Assumptions~\eqref{ass:A1} and~\eqref{ass:C1} are satisfied. Assume that $\theta\coloneq \lim_{t\to\infty}-\frac1t\log\P{L>t}$ exists and recall $R$ from~\eqref{eq:R}. If $\theta<R$ and Assumption~\eqref{ass:varphi2}, does persistence occur in the sense of~\eqref{eq:pers}?
	\end{problem} 
	
	\textbf{Non-tree graphs. } The analysis carried out here considers trees only. It would be interesting to also study evolving graphs with vertex death that are not trees, e.g.\ when every new vertex connects to $k>1$ many alive vertices. The continuous-time embedding into a CMJ branching process is central in the analysis in this paper and naturally restricts us to the setting of trees. 
	
	In the case of \emph{affine} preferential attachment graphs without vertex death, a procedure known as \emph{collapsing branching processes} is used to extend the branching process analysis to non-tree graphs as well~\cite{GarHof18}. However, this is limited to the affine case only. When including vertex death, this could only have a potential when the death rates are affine as well. Even then, this introduces difficulties due to the fact that now the lifetime of distinct vertices could have dependencies, which further complicates the analysis.
	
	\textbf{Other modelling choices. } In the (embedding of the) PAVD model the offspring of an individual equals the number of children said individual has produced. Similarly, the rate at which an individual produces its next child or dies depends on the total number of children produced. Instead, one could let the birth and death rates depend on the number of \emph{alive} children only. Either way of defining the model can be motivated from applications. However, analysing this alternative model presents more challenges, since now an individual needs to produce $n$ children out of which $n-k$ should die to have an offspring of size $k$. The birth and death rates of an individual thus depend the lifetimes of its children. Deijfen concludes from simulations presented in~\cite{Dei10} that the asymptotic behaviour of the degree distribution should not change significantly, and we expect this to be the case for other properties such as persistence as well. Still, it would be interesting to further investigate these models too.
	
	\section*{Acknowledgements} 
	
	Bas Lodewijks has received funding from the European Union’s Horizon 2022 research and innovation programme under the Marie Sk\l{}odowska-Curie grant agreement no.~$101108569$, ``DynaNet". He also thanks Tejas Iyer for a discussion regarding the proof of Proposition~\ref{prop:P}.
	
	\appendix
	
	\section{Persistent hubs in general CMJ branching processes}\label{app:genproof}
	
	In this appendix we discuss how part of Theorem~\ref{thrm:genpers} can be proved for more general branching processes. To this end, we recall the general definition of CMJ branching processes at the start of Section~\ref{sec:embed}. We make the following assumptions regarding the offspring random variable $D$ and the inter-birth time random variables $(X(i))_{i\in\N}$. 
	
	\begin{assumption}\label{ass:appendix}$\,$
		\begin{enumerate}[label=(\roman*)]
			\item $D$ and $(X(i))_{i\in\N}$ are mutually independent. 
			\item $X(i)<\infty$ for all $i\in\N$ almost surely. 
			\item\label{item:subcritimmediate} $\sum_{k=1}^\infty \P{D\geq k}\prod_{i=0}^{k-1}\P{X(i)=0}<1$. 
			\item\label{item:Dinf} $\P{D=\infty}>0$.
		\end{enumerate}
	\end{assumption}
	
	Part $(i)$ is natural for many classical branching process models (e.g.\ the Bienaym\'e-Galton-Watson branching process, the Yule branching process, and the Bellman-Harris branching process). Part $(ii)$ implies that $D$ corresponds to the number of children born to an individual, since each child will be born a finite amount of time after their parent is born. Part $(iii)$ can be viewed as the mean number of \emph{immediate} children. That is, the number of children born at the same time as their parent. We assume this number is sub-critical, to avoid the branching process being able to grow infinitely large in finite time with positive probability. Finally, Part $(iv)$ states that individual produce infinite offspring with positive probability, which is necessary for the (non-)existence of persistent hubs to be a non-trivial question. It is relatively straightforward to check that the CTPAVD branching process defined in Section~\ref{sec:embed} that we analyse in Sections~\ref{sec:inflife} and~\ref{sec:finlife} is an example of a branching process that satisfies Assumption~\ref{ass:appendix}.
	
	Let us now discuss the further assumptions used in Section~\ref{sec:inflife}. The point process $\cR$, as in~\eqref{eq:D}, is now defined as 
	\be 
	\cR \coloneq \sum_{k=1}^D \delta_{\sum_{i=1}^k X(i)}.  
	\ee  
	With $\mu$ the density of $\cR$ and $\widehat \mu$ its Laplace transform, we have 
	\be 
	\widehat\mu(\lambda)=\E{\sum_{k=1}^D\exp\bigg(-\lambda \sum_{i=1}^k X(i)\bigg)}.
	\ee 
	Assumption~\eqref{ass:mufin} then remains identical, though now for the more general $\widehat \mu$ stated here. Similar to the argument in~\eqref{eq:muless1}, Assumption~\ref{ass:appendix}\ref{item:subcritimmediate} implies that Assumption~\eqref{ass:mufin} is equivalent to the existence of a $\lambda>0$ such that $\widehat\mu(\lambda)<1$. Then, Assumption~\eqref{ass:A1} is generalised to 
	\be
	\sum_{i=1}^\infty X(i)=\infty\qquad \text{almost surely}, \tag{N-E'}\label{ass:NEgen}
	\ee 
	which implies that no individual can produce infinite offspring in finite time, preventing the branching process from growing infinitely large in finite time with positive probability. In Section~\ref{sec:inflife} we worked with Assumption~\eqref{ass:A2} \emph{not} being satisfied, which ensured that individuals can produce infinite offspring with positive probability. This is now replaced by Assumption~\ref{ass:appendix}\ref{item:Dinf}. Finally, Assumption~\eqref{ass:varphi2}, in light of Lemma~\ref{lemma:sumexp} (and Remark~\ref{rem:sumconvass}), is now replaced by the more general 
	\be 
	\sum_{i=1}^\infty (X(i)-X(i)')<\infty \qquad \text{almost surely}, \tag{C-V'}\label{ass:CVgen}
	\ee 
	where $(X(i)')_{i\in\N}$ is an i.i.d.\ copy of $(X(i))_{i\in\N}$. 
	
	With these more general definitions and assumptions at hand, we have the following result, which generalises part of Theorem~\ref{thrm:genpers}.
	
	\begin{theorem}\label{thrm:persappendix}
		Consider the CMJ branching process $(\bp(t))_{t\geq0}$ as defined in Section~\ref{sec:embed}, where the offspring $D$ and inter-birth times $(X(i))_{i\in\N}$ satisfy Assumption~\ref{ass:appendix}. 
		\begin{enumerate}[label=(\alph*)]
			\item If Assumption~\eqref{ass:CVgen} is not satisfied, then, for any $m\in\N$, the branching process $(\bp(t))_{t\geq 0}$ does not contain a persistent $m$-hub $\mathbb P_\cS$-almost surely. 
			\item If Assumption~\eqref{ass:CVgen} is satisfied and  there exists $\lambda>0$ and $K\in\N$  such that 
			\be \label{eq:perscond}
			\widehat\mu(\lambda)<1\qquad \text{and}\qquad \prod_{i=K+1}^\infty\E{\e^{\lambda(X(i)'-X(i))}}<\infty.  
			\ee    
			Then, $\mathbb P_\cS$-almost surely $(\bp(t))_{t\geq0}$ contains a persistent $1$-hub. 
		\end{enumerate}
	\end{theorem} 
	
	\begin{remark}
		Theorem~\ref{thrm:persappendix} does not include the existence of persistent $m$-hubs for $m>1$. This is due to the fact that, in the proof of Theorem~\ref{thrm:genpers}(b), we rely on the memoryless property of the exponential distribution when $m>1$.
		
		Part $(b)$ does not require Assumption~\eqref{ass:NEgen}, as it is implied by $\widehat\mu(\lambda)<1$ and Assumption~\eqref{ass:CVgen}. \ensymboldremark
	\end{remark}
	
	Iyer~\cite[Theorem 2.8]{Iyer24} present sufficient conditions under which a persistent $1$-hub, when it exists, is unique, in the sense that there exists a unique individual in the branching process such that it attains the largest offspring for all but a finite amount of time. This result can directly be extended to the more general class of CMJ branching process as well as to persistent $m$-hubs for any $m\in\N$.  introduced in this appendix. 
	
	\begin{theorem}[Uniqueness of persistent $m$-hubs, Theorem 2.8 in~\cite{Iyer24}]\label{thrm:uniqueappendix}
		Consider the CMJ branching process $(\bp(t))_{t\geq0}$ as defined in Section~\ref{sec:embed}, where the offspring $D$ and inter-birth times $(X(i))_{i\in\N}$ satisfy Assumption~\ref{ass:appendix}. Fix $m\in\N$ and suppose that $(\bp(t))_{t\geq 0}$ contains a persistent $m$-hub $\mathbb P_\cS$-almost surely. Also suppose one of the following conditions is satisfied. 
		\begin{enumerate}[label=(\arabic*)]
			\item We have 
			\be 
			\sum_{k=1}^\infty\P{0\leq \sum_{i=1}^k (X(i)'-X(i))-\sum_{i=1}^k X(i)<X(k+1)}<\infty.
			\ee  
			\item Both of the following conditions hold: 
			\begin{enumerate}[label=(\roman*)]
				\item We have $\sum_{i=1}^\infty \P{X(i)\neq X(i)'}=\infty$, or there exists $i\in\N$ such that $\P{X(i)\neq X(i)'}=1$. 
				\item For any $\eps>0$ we have $\sum_{i=1}^\infty \P{X(i)>\eps}<\infty$.
				\end{enumerate}  
		\end{enumerate}
		Then, the persistent $m$-hub is $\mathbb P_\cS$-almost surely unique.
	\end{theorem} 
	
	\begin{remark}
		We have used Condition (2) to prove the uniqueness of the persistent $m$-hub in the proof of Theorem~\ref{thrm:genpers}(b). \ensymboldremark
	\end{remark}
	
	We conclude this section with a brief discussion where the proof of Theorem~\ref{thrm:genpers}  requires adjustments to be able to prove Theorem~\ref{thrm:persappendix}. 
	
	\textbf{Adapting the proof of Theorem~\ref{thrm:genpers}$(a)$. } 
	
	 The proof of Part $(a)$ of Theorem~\ref{thrm:genpers} relies on Part $(c)$ of Lemma~\ref{lemma:tauwin}, where we replace Assumption~\eqref{ass:varphi2} with Assumption~\eqref{ass:CVgen}. The proof of Lemma~\ref{lemma:tauwin} can be adapted easily by simply 
	 replacing the birth times $(E^{(u)}_i)_{i\in\N}$ and $(E^{(v)}_i)_{i\in\N}$ with $(X(ui))_{i\in\N}$ and $(X(vi))_{i\in\N}$, respectively. Part $(a)$ leverages Lemma~\ref{lemma:finbp}, which holds as we still assume that $\widehat\mu(\lambda)<1$ for some $\lambda<1$. Part $(b)$, assuming Assumption~\eqref{ass:CVgen} is satisfied, requires one additional step. Namely, we distinguish between whether the distribution of the series $\sum_{j=N}^\infty (X(ui)-X(vi))$ has an atom. If it does not, the same proof can be applied. If it does, then~\cite[Theorem 1.13]{Iyer24} tells us that there exists a sequence $(c_j)_{j\in\N}$ of real numbers, where $c_j=0$ for all but finitely many $j$, such that 
	\be 
	\P{X(uj)-X(vj)=c_j}>0\quad \text{for all }j\in\N, \text{ and}\quad \sum_{j=1}^\infty \P{X(uj)-X(vj)\neq c_j}<\infty. 
	\ee 
	Hence, by the Borel-Cantelli lemma, $X(uj)=X(vj)$ for all but finitely many $j$ (as $c_j=0$ for these $j$). As a result, we have that either~\eqref{eq:winv} or~\eqref{eq:winu} is satisfied, or that $\mathbb P_\cS$-almost surely,
	\be 
	\cB(u)-\cB(v)+\sum_{j=0}^{k-1}(X(uj)-X(vj))=0\qquad \text{for all }k\geq K_0.
	\ee  
	In any of these three cases, the conclusion of the proof, that $\cB(uk)\leq \cB(vk)$ for all $k\geq K_0$ or $\cB(vk)\leq \cB(uk)$ for all $k\geq K_0$, is still valid. 
	
	Finally, for Part $(c)$ of Lemma~\ref{lemma:tauwin}, replacing $E^{(v)}_i$ with $X(vi)$ for $v=u$ and $v=uj$ and using that Assumption~\eqref{ass:CVgen} is not sufficient for the proof to hold.  
	
	The proof of Theorem~\ref{thrm:genpers}$(a)$ then follows in the same manner, again replacing $E_i^{(v)}$ with $X(vi)$ for the appropriate individuals $v$, and assuming Assumption~\eqref{ass:CVgen} is not satisfied. This yields Theorem~\ref{thrm:persappendix}$(a)$. 
	
	\textbf{Adapting the proof of Theorem~\ref{thrm:genpers}$(b)$.}
	
	As in Part $(a)$, we substitute $E^{(v)}_i$ with $X(vi)$ for any individual $v\in \cU_\infty$. Instead of Assumptions~\eqref{ass:A1}, \eqref{ass:mufin}, \eqref{ass:varphi2}, and~\eqref{ass:A2}, we instead use Assumptions~\ref{ass:appendix}, \eqref{ass:CVgen}, and~\eqref{eq:perscond}. Then, as the second~\eqref{eq:perscond} allows us to use Lemma~\ref{lemma:catch}, Proposition~\ref{prop:catch} follows in the same manner. For Proposition~\ref{prop:P} when $m=1$ requires minor adaptations, as discussed after the proof of Proposition~\ref{prop:P} already. These adaptations are sufficient for the proof to work in the more general setting we are in, here, as well. When $m>1$, we were not able to avoid using the memoryless property of the exponential distribution at several places, so that a full generalisation of Theorem~\ref{thrm:genpers}$(b)$ is out of reach, though we expect the result to be true. The proof of Theorem~\ref{thrm:genpers}$(b)$ itself then requires no further adaptations, with the exception of the proof of the uniqueness of persistent $1$-hubs. Here, by using one of the conditions in~\ref{thrm:uniqueappendix}, one can prove that the condition in~\eqref{eq:eqdeg} is satisfied with probability zero.

	\bibliographystyle{abbrv}
	\bibliography{paverbib}

\end{document}